\documentclass[bibliography=totoc,11pt,BCOR5mm]{article}

\usepackage[utf8]{inputenc}
\usepackage[T1]{fontenc}
\usepackage[english]{babel}
\usepackage{amsfonts}
\usepackage{amsmath}
\usepackage{amssymb}
\usepackage{amsthm}
\begingroup
\makeatletter
\@for\theoremstyle:=definition,remark,plain\do{%
	\expandafter\g@addto@macro\csname th@\theoremstyle\endcsname{%
		\addtolength\thm@preskip\parskip
	}
}
\endgroup
\usepackage{booktabs}
\usepackage{caption}
\usepackage{color,soul}
\usepackage{chngcntr}
\usepackage{constants}
\usepackage{dsfont}
\usepackage[shortlabels]{enumitem}

\usepackage{epsfig}
\usepackage{epstopdf}
\usepackage{etoolbox}
\usepackage{fancyhdr}
\usepackage{float}
\usepackage{framed}
\usepackage[hmargin=1in,vmargin=1in,a4paper]{geometry}
\usepackage{graphicx}
\usepackage[ngerman]{isodate}
\usepackage{keyval}
\usepackage{lscape}
\usepackage{mathdots}
\usepackage{mathrsfs}
\usepackage{mathtools}
\usepackage{multirow}
\usepackage{needspace}
\usepackage{ngerman}
\usepackage[onehalfspacing]{setspace}
\usepackage[percent]{overpic}
\usepackage{pifont}
\usepackage{rotating}
\usepackage{setspace}
\usepackage{subfigure}
\usepackage{tikz}
\usepackage{titling}
\usepackage{verbatim}
\usepackage{wasysym}
\usepackage{hyperref}
\usepackage{cleveref}

\usepackage[yyyymmdd]{datetime}

\usepackage[normalem]{ulem} 

\usepackage[draft]{todonotes}

\title{Invariance principles and Log-distance of F-KPP fronts in a random medium}

\author{Alexander Drewitz%
  \thanks{Universit\"at zu K\"oln,
    Mathematisches Institut,
    Weyertal 86--90,
    50931 K\"oln, Germany.
    Email: \url{drewitz@math.uni-koeln.de}}
\quad
  Lars Schmitz%
  \thanks{Universit\"at zu K\"oln,
    Mathematisches Institut,
    Weyertal 86--90,
    50931 K\"oln, Germany.
    Email: \url{lschmit2@math.uni-koeln.de}}}

\date{\today}

{ 
	\newtheorem{theorem}{Theorem} 
	\newtheorem{corollary}[theorem]{Corollary} \numberwithin{theorem}{section}
	\newtheorem{lemma}[theorem]{Lemma}\numberwithin{theorem}{section}
	\newtheorem{claim}[theorem]{Claim} \numberwithin{theorem}{section}
	 \numberwithin{theorem}{section}	
	\newtheorem{proposition}[theorem]{Proposition} \numberwithin{theorem}{section}
	 \numberwithin{theorem}{section}
	 \numberwithin{theorem}{section}
	 \numberwithin{theorem}{section}

	\theoremstyle{remark}
	\newtheorem{remark}[theorem]{Remark} \numberwithin{theorem}{section}
	\newtheorem{remark*}{Remark}
	
}

{
}

\newcommand*\diff{\mathop{}\!\mathrm{d}} 

\newcommand{\E}{\mathbb{E}} 
\newcommand{\p}{\mathbb{P}} 
\newcommand{\Eprob}{\texttt{E}} 
\newcommand{\Eprobth}{\emph{\texttt{E}}}
\newcommand{\prob}{\texttt{P}}
\newcommand{\probth}{\emph{\texttt{P}}}
\newcommand{\ei}{\texttt{ei}}
\newcommand{\es}{\texttt{es}}

\newcommand{\e}{\varepsilon} 
\newcommand{\given}{\, |\, }
\newcommand{\munder}{\overline{m}}
\newcommand{\N}{\mathbb{N}}
\newcommand{\R}{\mathbb{R}}

\newcommand{\overeta}{\eta^*}
\newcommand{\undereta}{\eta_*}

\newcommand{\Var}{\mathrm{Var}}
\newcommand{\Z}{\mathbb{Z}}
\newcommand{\F}{\mathcal{F}}
\newcommand{\tend}[2]{\displaystyle\mathop{\longrightarrow}_{#1\rightarrow#2}}
\newcommand{\varnet}[2]{\lfloor{#2}\rfloor_{#1}}
\newcommand{\epsnet}[1]{\varnet{\varepsilon}{#1}}
\newcommand{\te}{{t}'}
\newcommand{\y}{{y}'}

\newcommand{\parenthezises}[1]{\arabic{#1}}

\newconstantfamily{index_prob_cont_time}{
	symbol=\mathcal{T},
	format=\parenthezises,
}

\numberwithin{equation}{section}

\newconstantfamily{index_prob}{
	symbol=\mathcal{N},
	format=\parenthezises,
}
\newconstantfamily{constant}{
	symbol=\mathcal{C},
	format=\parenthezises,
}
\newconstantfamily{small}{
	symbol=c,
	format=\parenthezises,
}

\DeclareMathOperator*{\esssup}{ess\,sup}
\DeclareMathOperator*{\essinf}{ess\,inf}

\begin{document}

\singlespacing

\maketitle

\begin{center}
\small{{\bf Preliminary version}}
\end{center}
\bigskip
\begin{abstract}
	We study the front of the solution to the F-KPP equation with randomized non-linearity. Under suitable assumptions on the randomness involving spatial mixing behavior and boundedness, we show that the front of the solution  lags at most logarithmically in time  behind the front of the solution of the corresponding linearized equation, i.e.\  the parabolic Anderson model. This can be interpreted as a partial generalization of Bramson's findings \cite{bramson78} for the homogeneous setting.
 	Building on this result, we establish functional central limit theorems for the fronts of the solutions to both equations.
\end{abstract}

\tableofcontents

\section{Introduction and main results}

\subsection{The classical F-KPP equation}
\label{sec:hom_xi}

The F-KPP equation is the initial value problem given by
	\begin{align}
	\label{eq:homKPP1}
	\begin{split}
	w_t(t,x) &= \frac{1}{2}w_{xx}(t,x) + w(t,x)(1-w(t,x)),\quad t>0,\ x\in\R,\\
	w(0,x) &= w_0(x),\quad x\in\R, 
	\end{split}
	\end{align}
	with initial condition $w_0: \R \to [0,1].$ Its investigation has a long history, with seminal results dating back to Fisher \cite{fisher1937} and Kolmogorov-Petrovskii-Piskunov \cite{kolmogorov1937}. Their research had been motivated by pioneering works in genetics, where the equation has been used to model a randomly mating diploid population living in a one-dimensional habitat. Further applications can be found in chemical combustion theory or flame propagation, see \cite{aronson_diffusion}, \cite{fife_mcleod}, as well as \cite{Sa-03} and references therein. 
	
In \cite{kolmogorov1937} it has been shown that for reasonably general non-linearities (see \eqref{eq:standard_condi} at the beginning of Section \ref{sec:furtherNot} for further details) and initial conditions, the solution to \eqref{eq:homKPP1} approaches a  \emph{traveling wave} $g:\R\to [0,1]$. I.e., there exists a function $m: [0,\infty) \to \R$ -- generally referred to as the \emph{front} or \emph{breakpoint} --  such that
\[ w(t,\cdot +m(t)) \tend{t}{\infty} g\quad\text{uniformly.} \]
The limit $g$ is known (see \cite[Theorems~14 and 17]{kolmogorov1937}) to solve the differential equation
\[ \frac{1}{2}g''(x) + \sqrt{2}g'(x) + g(x)(1-g(x)) = 0,\ x\in\R, \]
\[ 0<g(x)<1\quad\text{for all }x\in\R\quad\text{ and }\quad g(x)\tend{x}{\infty} 0,\ g(x)\tend{x}{-\infty}1, \]
and it is unique up to spatial translations. Additionally, for the important  case of Heaviside-type initial condition $w_0=\mathds{1}_{(-\infty,0]},$ in \cite{kolmogorov1937} the first order asymptotics \( m(t)/t \tend{t}{\infty} \sqrt{2}\) has been derived. 
A couple of decades later, Bramson in his seminal work \cite{bramson78} improved this result by computing the second order correction up to additive constants.
More precisely, he showed that for each $\e\in(0,1),$ the choice $m(t)=m^\e(t):=\sup\{x\in\R: w(t,x)=\e \}$ fulfills
\begin{equation} \label{eq:homNonLinFront}
 m(t)  = \sqrt{2}t - \frac{3}{2\sqrt{2}}\ln t +\mathcal{O}(1)\quad \text{ as } t\to\infty. 
 \end{equation}
Later on, Bramson  \cite[Theorem~3]{bramson1983convergence} generalized this result to more general initial conditions; roughly speaking, he demanded the initial condition $w_0(x)$ to have a sufficiently fast exponential decay for $x\to\infty$ and to be non-vanishing for $x\to-\infty$. 
One of the main tools employed in the proof was the McKean representation of the solution to \eqref{eq:homKPP1} in terms of expectations of branching Brownian motion, see \cite{branch_markov_I} and \cite{mckean}.
Another important ingredient was the comparison of the solution of \eqref{eq:homKPP1} to the solution of its linearized version
\begin{align}
\label{eq:PAM_hom}
\begin{split}
u_t(t,x) &= \frac{1}{2}u_{xx}(t,x) + u(t,x),\quad t>0,\ x\in\R,\\
u(0,x) &= w_0(x). 
\end{split}
\end{align}
Indeed, since $m(t)$ describes the front of the solution where $u$ is small and hence $1-u \approx 1,$ the heuristics is that $m(t)$ is in some sense well-approximated by the front of the solution to \eqref{eq:PAM_hom}. 
More precisely, for $\munder(t):=\sup\{x\in\R: u(t,x)=\e \}$ and Heaviside-type initial condition $w_0=\mathds{1}_{(-\infty,0]}$, standard Gaussian computations entail
\begin{equation} \label{eq:homLinFront}
\munder(t) = \sqrt{2}t - \frac{1}{2\sqrt{2}}\ln t + \mathcal{O}(1).
\end{equation}
In combination with \eqref{eq:homNonLinFront}, this results in a respective logarithmic backlog of the two fronts in the sense that
\begin{equation}
\label{eq:log_delay_hom}
\munder(t) - m(t) = \frac{1}{\sqrt{2}}\ln t + \mathcal{O}(1).
\end{equation}
The main goal of this article is to investigate the effect of introducing a random potential in the non-linearity of \eqref{eq:homKPP1} as well as in its linearization \eqref{eq:PAM_hom},
on the logarithmic backlog derived in \eqref{eq:log_delay_hom} (cf.\
Theorem \ref{th: log-distance breakpoint median}). Taking advantage of this result we will then also derive functional central limit theorems for the fronts of the respective solutions to the randomized equations, see Theorem \ref{th:InvFrontPAM} and Corollary \ref{cor: functional clt for m} below, which can be interpreted as analogues to homogeneous-case results \eqref{eq:homLinFront} and \eqref{eq:homNonLinFront}.

\subsection{The randomized F-KPP equation and the parabolic Anderson model}
\label{sec:random_pot}
Already in Fisher's seminal paper \cite{fisher1937}, where he investigated the setting \eqref{eq:homKPP1} of homogeneous branching rates, it has been observed that a more realistic model would be obtained by considering spatially heterogeneous rates of the transformation of recessive to advantageous alleles. This, alongside a mathematical interest, is our guiding motivation to consider the setting of random $\xi$. Replacing the term $w(1-w)$ of \eqref{eq:homKPP1}  by a more general non-linearity $F(w)$ fulfilling suitable \emph{standard conditions} (see \eqref{eq:standard_condi} below), and overriding the notation $u$ and $w$ from the homogeneous setting of the previous section, we then arrive at the equation
\begin{align}
	\label{eq:KPP1}\tag{F-KPP}
	\begin{split}
	w_t &= \frac{1}{2}w_{xx} +  \xi(x,\omega)\cdot F(w),\quad t>0,\ x\in\R,\\
		w(0,x) &= w_0(x),\quad x\in\R,
	\end{split}
\renewcommand*{\theHequation}{notag.\theequation}
\end{align}
as well as its linearized version,  the \emph{parabolic Anderson model} 
\begin{align}
\label{eq:PAM}\tag{PAM}
\begin{split}
u_t &= \frac{1}{2}u_{xx} + \xi(x,\omega)\cdot u,\quad t>0,\ x\in\R,\\
u(0,x) &= u_0(x), \quad  x\in\R.
\end{split}
\renewcommand*{\theHequation}{notag2.\theequation}
\end{align}
Here, the stochastic process $\xi:\R\times\Omega\to(0,\infty)$ models the random medium. For the case $F(w)=w(1-w)$ and degenerate $\xi \equiv 1$, these two equations yield the special cases \eqref{eq:homKPP1} and \eqref{eq:PAM_hom}, respectively.
It has long been known, see e.g.\ Freidlin \cite[{Theorem 7.6.1}]{Fr-85},  that under suitable assumptions there exists $v_0>0$, such that the solution $w(t,x)$ to \eqref{eq:KPP1} converges  to $0$ (resp.\ $1$), uniformly  for all $x\geq \overline{v}t$ with $\overline{v}>v_0$ (resp.\ for all $x\leq\underline{v}t$ with $\underline{v}<v_0$), as $t$ tends to infinity. In a similar way, this result can be shown to hold for the solution to \eqref{eq:PAM} with the same $v_0$ as well, showing that the {\em speeds} or {\em velocities} of both fronts, the one of the solution to \eqref{eq:KPP1} as well as the one of the solution to \eqref{eq:PAM}, coincide.
Consequently, as in the homogeneous case, the question of second order corrections arises naturally.

\subsection{Summary of results}

In order to address this question and to be able to summarize our results, we introduce some notation.
Let $\e\in(0,1)$ and $a>0$. Furthermore, write $u=u^{\xi,u_0}$ for the solution to \eqref{eq:PAM} with initial condition $u_0,$ and $w=w^{\xi,F,w_0}$ for the solution to \eqref{eq:KPP1} with initial condition $w_0$. As in the previous section, the fronts of the respective solutions, denoted by
\begin{equation}
\label{eq:def_munder}
\begin{split}
\munder^{\xi,u_0,a}(t) &:= \sup\big\{  x\in\mathbb{R}: u(t,x)\geq a\big\},  \\
m^{\xi,F,w_0,\e}(t) &:= \sup\big\{ x\in\mathbb{R}: w(t,x)\geq \e \big\},
\end{split}
\end{equation}
are of special interest. Since we will from now on focus on the heterogeneous setting, we override the notation from Section \ref{sec:hom_xi}, where it was used to denote the fronts in the  homogeneous case, and define
\begin{align} \label{eq:munder}
\begin{split}
	\munder(t)&:=\munder^{a}(t):= \munder^{\xi,\mathds{1}_{(-\infty,0]},a}(t), \\
m(t)&:=m^{\varepsilon}(t):=  m^{\xi,F,\mathds{1}_{(-\infty,0]},\varepsilon}(t).
\end{split}
\end{align}
Our findings are motivated by the respective results 
of \eqref{eq:homLinFront}, \eqref{eq:homNonLinFront}, and
\eqref{eq:log_delay_hom} for the homogeneous case, which provide information about the position of the fronts of the solutions to the respective equations, and thus their respective backlog as well. Under suitable assumptions, our results are summarized in the following two statements: 
\begin{enumerate}	
	\item There exist a constant $C \in (0,\infty)$ and a $\p$-a.s.\ finite random time $\mathcal{T}(\omega)$ such that for all $t \ge  \mathcal{T}(\omega),$
	\begin{equation} \label{eq:logUB}
	\munder(t) - m(t) \le C \ln t;
	\end{equation}
	see Theorem \ref{th: log-distance breakpoint median} below. 
	
	\item After centering and diffusive rescaling, the stochastic processes $[0,\infty) \ni t \to \munder(t)$ and $[0,\infty) \ni t \to m(t)$ satisfy invariance principles; see 	
	Theorem \ref{th:InvFrontPAM} and Corollary \ref{cor: functional clt for m} below. 
\end{enumerate}

As is shown in the companion article \cite{CeDrSc-20}, in a certain sense there is a logarithmic lower bound for $\munder(t) - m(t) $  corresponding to \eqref{eq:logUB} as well; cf.\ also Section \ref{sec:discussion} below for further discussion.

\subsection{Further notation}
\label{sec:furtherNot}
In order to be able to precisely formulate the previously summarized results, we  have to introduce some further notation. We start with introducing the \emph{standard conditions} for the non-linearity, i.e., $F$ in \eqref{eq:KPP1} has to fulfill the following:
\begin{equation}
	\label{eq:standard_condi}
	\begin{split}
		&F \in C^1([0,1]),\quad F(0)=F(1)=0,\quad	 	 F(w)>0 \quad \forall w\in(0,1),\quad  \\
	F'&(0)=1=\sup_{w>0} F(w)w^{-1},\quad F'(1)<0,\quad \limsup_{w\downarrow 0}\frac{1-F'(w)}{w}<\infty.
	\end{split}
	\tag{SC}
\renewcommand*{\theHequation}{notag3.\theequation}
\end{equation}
\begin{figure}[h!]
	\centering
	\includegraphics[width=0.5\linewidth]{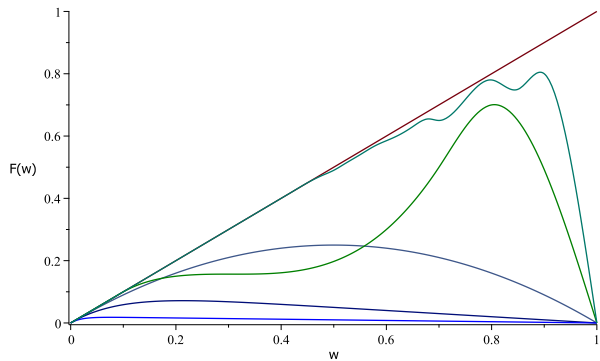}
	\caption{Sketches of functions fulfilling \eqref{eq:standard_condi}, all of which are dominated by the identity function.}
\end{figure}

We will now specify the classes of initial conditions under consideration for both, \eqref{eq:KPP1} and \eqref{eq:PAM}. For this purpose, we 
fix $\delta'\in(0,1)$ and $C'>1,$ and require an initial condition $u_0$ of \eqref{eq:PAM} to fulfill
\begin{equation}
\label{eq:PAM_initial}
\tag{PAM-INI}
\begin{split}
& \delta'\mathds{1}_{[-\delta',0]} \leq  u_0 \leq C'\mathds{1}_{(-\infty,0]}.
\end{split}
\renewcommand*{\theHequation}{notag4.\theequation}
\end{equation}
Our results also hold for initial conditions that decay sufficiently fast at infinity, as well as for  initial conditions that grow towards minus infinity with sufficiently small exponential rate.
However, in order to avoid further technical complications we stick to the above set of initial conditions.

In addition, let us introduce a tail condition for the initial condition of \eqref{eq:KPP1}, which is the same as the one for the case $\xi\equiv1$ stated in \cite[(1.17)]{bramson1983convergence}. For this purpose,
 we fix $N,N'>0$, and require $w_0$ as in \eqref{eq:KPP1} to fulfill 
\begin{equation}
\label{eq:KPP_initial}
\tag{KPP-INI}
\begin{split}
0\leq w_0\leq \mathds{1}_{(-\infty,0]} \qquad\text{and}\qquad \int_{[x-N,x]}w_0(y)\diff y \geq \delta'\quad \forall x\leq -N'.
\end{split}
\renewcommand*{\theHequation}{notag5.\theequation}
\end{equation}
Denote by $\mathcal{S}^1$ the class of functions $f:\R\to [0,\infty)$ which are  pointwise limits of increasing sequences of continuous functions, and let 
\begin{align*}
\mathcal{I}_{\text{PAM}} &:= \big\{ u_0\in\mathcal{S}^1:\ u_0\text{ fulfills \eqref{eq:PAM_initial}} \big\},\quad
\mathcal{I}_{\text{F-KPP}} := \big\{ w_0\in\mathcal{S}^1:\ w_0\text{ fulfills \eqref{eq:KPP_initial}}\big\},
\end{align*}
which will be the classes of initial conditions under consideration.
An emblematic example which is contained in both, $\mathcal{I}_{\text{F-KPP}}$ and $\mathcal{I}_{\text{PAM}},$  is the function $ \mathds{1}_{(-\infty,0]}$ of Heaviside type.

We will assume $\xi=(\xi(x))_{x\in\mathbb{R}}$ to be a stochastic process on a probability space $(\Omega,\F,\p)$ having   H\"{o}lder continuous paths, i.e., there exists $\alpha=\alpha(\xi)>0$ and $C=C(\xi)>0$, such that 
\begin{equation}
\label{eq:hoelder_cont}
\left| \xi(x)-\xi(y) \right| \leq C\,  |x-y|^\alpha\quad \forall x,y\in \R,
\tag{H\"OL}\renewcommand*{\theHequation}{notagHOELD.\theequation}
\end{equation}
and such that the  following conditions are fulfilled: 
\begin{itemize}
	\item $\xi$ is \emph{uniformly bounded away from $0$ and $\infty$}, i.e., there exist constants $0<\ei\leq\es<\infty$ such that $\p$-a.s.,
	\begin{equation}\label{eq:POT}
	\ei\leq \xi(x)\leq\es\quad \text{ for all }x\in\mathbb{R}\tag{BDD};
	\renewcommand*{\theHequation}{notag6.\theequation}
	\end{equation}
	\item $\xi$ is \emph{stationary}, i.e. for every $h\in\mathbb{R}$ 
	we have
	\begin{equation}
	\label{eq:STAT}
	(\xi(x))_{x\in\R} \overset{d}{=} (\xi(x+h))_{x\in\R};
	\tag{STAT}
	\renewcommand*{\theHequation}{notag7.\theequation}
	\end{equation}
	\item $\xi$ fulfills a  \emph{$\psi$-mixing} condition:  Let $\F_x:=\sigma(\xi(z):z\leq x)$ and $\F^y:=\sigma(\xi(z):z\geq y)$, $x,y\in\mathbb{R}$ and assume that there is a continuous, non-increasing function $\psi:[0,\infty)\to[0,\infty)$, such that for all $j,k\in\Z$ with
	$j\leq k$ as well as $X\in\mathcal{L}^1(\Omega,\F_j,\p)$ and $Y\in\mathcal{L}^1(\Omega,\F^k,\p)$ we have
	\begin{align}
	\label{eq:MIX}
	\big| \E\big[X-\E[X] \given \F^{k}\big]\big| &\leq \E[|X|]\cdot \psi(k-j), \notag \\
	\big| \E\big[Y-\E[Y] \given  \F_j\big]\big| &\leq \E[|Y|]\cdot \psi({k-j}),  \tag{MIX} \\
	\sum_{k=1}^\infty &\psi(k)<\infty. \notag
	\renewcommand*{\theHequation}{notag8.\theequation}
	\end{align}
	Note that \eqref{eq:MIX} implies the ergodicity of $\xi$ with respect to the shift operator $\theta_y$ acting via  $\xi(\cdot)\circ\theta_y=\xi(\cdot+y)$, $y\in\R$. 
\end{itemize}
Summarizing, we arrive at the following standing assumptions: 
\begin{equation}\label{eq:SASS} \tag{{\bf Standing assumptions}}
	\begin{gathered}
  \text{We will assume
   conditions \eqref{eq:hoelder_cont}, \eqref{eq:POT}, \eqref{eq:STAT} and \eqref{eq:MIX}}\\ \text{to be fulfilled from now on, if not explicitly mentioned otherwise.} 
\end{gathered}
\renewcommand*{\theHequation}{notag9.\theequation}
	\end{equation}
Since we allow for non-smooth initial conditions in both equations, \eqref{eq:KPP1} and \eqref{eq:PAM}, we shortly comment on the notion of solution in this setting. We call a function $w$ to \eqref{eq:KPP1} a \emph{generalized solution to \eqref{eq:KPP1}}, if it satisfies
\begin{equation}
\label{eq:genSol}
w(t,x) = E_x \Big[ \exp\Big\{ \int_0^t \xi(B_s)F(w(t-s,B_s))/w(t-s,B_s)\diff s \Big\} \, w_0(B_t) \Big],\quad \forall\ (t,x)\in[0,\infty)\times\R,
\end{equation}
where $E_x$ is the corresponding expectation of the probability measure $P_x$, under which $(B_t)_{t\geq0}$ is a standard Brownian motion starting in $x\in\R$. 
This equation can be interpreted as a mild formulation of \eqref{eq:KPP1} and one can show (see e.g.\ \cite[(1.4), p.\ 354, and (a), p.\ 355]{Fr-85})  that every classical solution to \eqref{eq:KPP1} also is a generalized one. Generalized solutions can be shown to exist under weak assumptions, and, vice versa, in many instances they turn out to be classical solutions indeed,  see Proposition \ref{prop:genClass} in Appendix \ref{sec:appendix}.

\begin{remark}
	\label{rem:gen_approx}
	A useful observation is that for initial conditions in $\mathcal{I}_{\text{F-KPP}}$, generalized solutions can be approximated by classical solutions. That is, if $(w_0^{(n)})_{n\in\N}\subset \mathcal{I}_{\text{F-KPP}}$ is a sequence of continuous functions which increase pointwise to $w_0$, then by Corollary~\ref{cor:mon_solutions} the corresponding sequence of (by Proposition~\ref{prop:genClass} classical) solutions  $(w^{(n)})_{n\in\N}=(w^{w_0^{(n)}})_{n\in\N}$ to \eqref{eq:KPP1} is also monotone and thus the limit $w(t,x):=\lim_{n\to\infty} w^{{(n)}}(t,x)$ exists for all $(t,x)\in[0,\infty)\times\R$. Dominated convergence and the fact that $w^{(n)}$ also is a generalized solution then imply
	\begin{align*}
	w(t,x) &= \lim_{n\to\infty} w^{(n)}(t,x) = \lim_{n\to\infty} E_x \Big[ \exp\Big\{ \int_0^t \xi(B_s)\frac{F(w^{(n)}(t-s,B_s))}{w^{(n)}(t-s,B_s)}\diff s \Big\} \, w^{(n)}_0(B_t) \Big] \\
	&=  E_x \Big[ \exp\Big\{ \int_0^t \xi(B_s)\frac{F(w(t-s,B_s))}{w(t-s,B_s)}\diff s \Big\} \, w_0(B_t) \Big].
	\end{align*}
	I.e., $w$ is the generalized solution to \eqref{eq:KPP1} with initial condition $w_0$.
\end{remark}

A similar concept is possible for the solution to \eqref{eq:PAM} as well. That is, for $u_0\in\mathcal{I}_{\text{PAM}}$, we call a function
\begin{align}
u(t,x)=E_x\Big[ \exp\Big\{ \int_0^{t} \xi(B_s)\diff s \Big\}u_0(B_t) \Big], \label{eq:feynman_kac}
\end{align} 
a \emph{generalized solution to \eqref{eq:PAM}}. 
Note that this is an explicit expression of the solution in terms of a Brownian path, while in \eqref{eq:genSol} the solution is only given implicitly.
Furthermore, if $u_0$ is continuous, then  by \cite[Remark~4.4.4]{karatzas_shreve}, there exists a unique solution $u$ to \eqref{eq:PAM} which fulfills $u\in C^{1,2}((0,\infty)\times\R)$ and $ u(0,\cdot)=u_0.$  
As a consequence of these observations, in the following, we will always consider generalized solutions.

\subsection{The linearized equation}
As already mentioned in the first section, we expect that investigating the solution to \eqref{eq:PAM} might also provide some insight into the solution to \eqref{eq:KPP1}. Therefore, starting with the first order of the front as a function of time, it turns out useful to consider the so-called \emph{Lyapunov exponent}
\begin{equation} 
\label{eq:lyapunov}
\Lambda(v) := \lim_{t\to\infty} \frac{1}{t} \ln u(t,vt).
\end{equation}
Due to Proposition~\ref{prop:lyapunov}, the Lyapunov exponent exists $\p$-a.s.\ for all $v\in \R$, is non-random, and -- as a consequence of Corollary~\ref{cor:lyapunov_alt} -- does not depend on the initial condition in $\mathcal{I}_{\text{PAM}}$ under consideration.
Furthermore, the function $[0,\infty)\ni v\mapsto\Lambda(v)$ is concave, tends to $-\infty$ as $v\to\infty$ and $\Lambda(0)=\es$, where $\es$ is defined in \eqref{eq:POT}. $\Lambda(v)$ describes the asymptotic exponential growth of the solution in the linear regime with speed $v$. By Proposition~\ref{prop:lyapunov}, there exists a unique $v_0>0$, such that 
\[\Lambda(v_0) = 0,\] 
which we will call \emph{velocity} or \emph{speed} of the solution to \eqref{eq:PAM}. 
Using  the properties of the Lyapunov exponent, we immediately infer the first order asymptotics for $\munder$ to $\p\text{-a.s.}$ satisfy
\begin{equation*}
\frac{\munder(t)}{t} \tend{t}{\infty} v_0
\end{equation*}
It will turn out that our methods work if we require $v_0$ to be strictly larger than some ‘‘critical'' value $v_c$, defined in Lemma~\ref{le:expectedLogMGfct} (d). Roughly speaking, the condition $v>v_c$ allows us to find a suitable additive tilting parameter in the exponent  of  the  Feynman-Kac representation \eqref{eq:feynman_kac}, which depends on $v$ and makes the solution $u(t,x)$ to \eqref{eq:PAM} amenable to the investigation by standard tools for values $x\approx vt$ and large $t$. Hence, we work under the assumption 
\begin{equation}
\label{eq:VEL}
\tag{VEL}
v_0>v_c.
\renewcommand*{\theHequation}{notag10.\theequation}
\end{equation}
As will be shown in Section~\ref{sec:vel}, this assumption is fulfilled for a rich class of potentials $\xi.$

We start with investigating the fluctuations of the function $t \mapsto \ln u(t,vt)$ around $t\Lambda(v)$ for values $v$ in a neighborhood of $v_0$, 
which are interesting in their own right. To this end, on the space $C([0,\infty))$ of continuous functions from $[0,\infty)$ to $\R,$ we define the metric
\begin{equation}
\label{eq:def_metric}
\rho(f,g):=\sum_{j=1}^\infty 2^{-j} \frac{\lVert f-g\rVert_j}{1+\lVert f-g\rVert_j}, \quad f,g\in C([0,\infty)),
\end{equation}
 where we write $\lVert f-g\rVert_j:=\sup_{x\in[0,j]}|f(x)-g(x)|.$
 This makes $(C([0,\infty)),\rho)$ a complete separable metric space.

\begin{theorem}
	\label{th:InvPAM}
	Let \eqref{eq:VEL} be fulfilled, $u_0\in\mathcal{I}_{\text{\emph{PAM}}},$ and $u=u^{\xi,u_0}$ be the corresponding solution to \eqref{eq:PAM}. Furthermore, let $V\subset(v_c,\infty)$ be a compact interval such that $v_0\in \text{\emph{int}}(V)$.	Then for each $v\in V$, as $n\to\infty$ the sequence of random variables $(nv)^{-1/2}\big( \ln u(n,vn)-n\Lambda(v) \big),$ $n \in \N,$ converges in $\p$-distribution to a centered Gaussian random variable with variance  $\sigma_v^2\in[0, \infty)$, where $\sigma_v^2$ is defined in \eqref{eq:sigmav}. If $\sigma_v^2>0$, the sequence of processes
	\[ [0,\infty) \ni t\mapsto \frac{1}{\sqrt{nv\sigma_v^2}} \big( \ln u(nt,vnt) - nt \Lambda(v) \big),\quad n\in\N, \]
	converges as $n\to\infty$ in $\p$-distribution to a standard Brownian motion  in the sense of weak convergence of measures on $C([0,\infty))$ endowed with the metric $\rho$ from \eqref{eq:def_metric}.
\end{theorem}
In combination with perturbation estimates for $u,$ we will use this result in order to infer an invariance principle for the front of the solution to \eqref{eq:PAM}. Note that since the function $t\mapsto \munder_t$ may be discontinuous, we consider convergence in the Skorohod space $D([0,\infty))$  in the following result.

\begin{theorem}
	\label{th:InvFrontPAM}
	Let \eqref{eq:VEL} be fulfilled, $u_0\in\mathcal{I}_{\text{\emph{PAM}}}$ and $M>0$. Then for $n \to \infty,$ the sequence $n^{-1/2}\big( \munder^{\xi,u_0,M}(n)-v_0n \big),$ $n \in \N,$  converges in $\p$-distribution to a centered Gaussian random variable with variance  $\widetilde{\sigma}_{v_0}^2\in[0,\infty)$, where $\widetilde{\sigma}_{v_0}^2$ is defined in \eqref{eq:def_sigma_tilde}. If $\widetilde{\sigma}_{v_0}^2>0$,  the sequence of processes 
	\[ [0,\infty) \ni t\mapsto \frac{\munder^{\xi,u_0,M}(nt) - v_0nt}{\sqrt{n\widetilde{\sigma}_{v_0}^2}},\quad n\in\N, \]
	converges as $n\to\infty$ in $\p$-distribution to a standard Brownian motion in the Skorohod space  $D([0,\infty))$.  
\end{theorem}
We underline that in the above theorems the case $\sigma_v^2=0$ (resp.\  $\widetilde{\sigma}_{v_0}^2=0$) is allowed and leads to a degenerate limit of the corresponding sequences. This can be excluded, e.g.\ if the finite-dimensional projections of the stochastic process $\xi=(\xi(x))_{x\in\R}$     are associated (see e.g.\ \cite{associated_rv_rao} or \cite{oliveira} for definitions and results). In this case the covariances in \eqref{eq:sigmav} are nonnegative and $\sigma_v^2>0$ follows. \cite[Proposition~2.1]{nolen_random_medium_FKPP} provides an example of a potential, which is generated by an i.i.d.\ sequence of  random variables and thus associated.

\subsection{The non-linear equation}
Coming back to the original equation of interest, it is natural to ask whether we can obtain results for \eqref{eq:KPP1}, which are in some sense counterparts to those derived in latter section for \eqref{eq:PAM}. For the solution to \eqref{eq:KPP1} it is known (see e.g.\ \cite[§ 7.6]{Fr-85}) that the first order of the front is linear as well and moves with the same  velocity as the front of \eqref{eq:PAM}. Indeed, by \cite[Theorem 7.6.1]{Fr-85} we have that $\p\text{-a.s.},$
\begin{equation*}
\frac{m(t)}{t} \tend{t} {\infty} v_0.
\end{equation*}
As mentioned in Section~\ref{sec:random_pot}, the next result, which is one of the main results of the paper, states that there is an at most logarithmic distance of the fronts of \eqref{eq:KPP1} and \eqref{eq:PAM}. 
\begin{theorem}
	\label{th: log-distance breakpoint median}
	Let \eqref{eq:VEL} be fulfilled. Then for each $F$ fulfilling \eqref{eq:standard_condi} there exists a constant $\Cl{constant theorem log dist}>0$ such that the following holds:  For all $M>0$, $\varepsilon\in(0,1)$, $u_0\in\mathcal{I}_{\text{\emph{PAM}}}$ and $w_0\in\mathcal{I}_{\text{\emph{F-KPP}}},$ there exists a non-random  $C=C(\e,M,u_0,w_0)>0$ and a $\p\text{-a.s.}$ finite  random time  $\mathcal{T}=\mathcal{T}(\xi,\e,M,u_0,w_0)\geq 0$, such that 
	\begin{equation}
	\label{eq: log-distance breakp median}
	-C\leq \munder^{\xi,u_0,M}(t) - m^{\xi,F,w_0,\e}(t) \leq C_1\ln t + C\quad\forall t\geq \mathcal{T}.
	\end{equation}
\end{theorem} 
Moreover for $w_0=u_0$ and $M\leq \varepsilon$, the left inequality in \eqref{eq: log-distance breakp median} is zero for all $t\geq0$.

Furthermore, combining Theorem~\ref{th:InvFrontPAM} and Theorem~\ref{th: log-distance breakpoint median}, we can deduce an invariance principle for the front of \eqref{eq:KPP1} as well. 
\begin{corollary}
	\label{cor: functional clt for m}
	Let \eqref{eq:VEL} be fulfilled, $F$ fulfill \eqref{eq:standard_condi},  $w_0\in\mathcal{I}_{\text{\emph{F-KPP}}}$ and  $\e\in(0,1)$. Then as $n \to \infty,$ the sequence $n^{-1/2}\big( m^{\xi,F,w_0,\varepsilon}(n)-v_0n \big),$ $n \in \N,$  converges in $\p$-distribution to a centered Gaussian random variable with variance  $\widetilde{\sigma}_{v_0}^2\in[0,\infty)$, where $\widetilde{\sigma}_{v_0}^2$ is defined in \eqref{eq:def_sigma_tilde}. If $\widetilde{\sigma}_{v_0}^2>0$, the sequence of processes 
	\[ [0,\infty) \ni t\mapsto \frac{m^{\xi,F,w_0,\varepsilon}(nt) - v_0nt}{\sqrt{n\widetilde{\sigma}_{v_0}^2}},\quad n\in\N, \]
	converges as $n\to\infty$ in $\p$-distribution to a standard  Brownian motion in the Skorohod space  $D([0,\infty))$.
\end{corollary}

\subsection{Discussion and previous results} \label{sec:discussion}

As already alluded to above, in the homogeneous case \eqref{eq:homKPP1} of
constant potential the front has been well-understood by now. This is indeed the case to a much
 wider extent than illustrated in the Introduction, see e.g.\
  \cite{Bo-16} and references therein for further details. Also the heterogeneous setting \eqref{eq:KPP1} of random potential and the properties of its solution have been investigated. Specifically, under fairly general assumptions, the existence and characterization of the propagation speed (i.e., the linear order $\lim_{t \to \infty} m(t)/t$ of the position of the front) have been derived by Freidlin and G\"artner, see e.g.\ \cite{GaFr-79} as well as \cite[Chapter VII]{Fr-85}, using large deviation principles. Incidentally, the Feynman-Kac formula (see Proposition \ref{prop:many-to-few} below), which characterizes the solution to the linearization \eqref{eq:PAM}, has also played an important role in the derivation.

Second order corrections for the position have been investigated by Nolen \cite{nolen_random_medium_FKPP}. Also making the detour along \eqref{eq:PAM}, he examines \eqref{eq:KPP1} for a potential $\xi$ under similar assumptions, but requires the (random) initial conditions to satisfy
\begin{equation}
\label{eq:ODE}
\mathcal{C}_1(\xi) g^{\xi,\gamma}(x) \leq w_0(x,\xi) \leq \mathcal{C}_2(\xi)g^{\xi,\gamma}(x)\quad\forall\ x>0.
\end{equation}
Here, $\mathcal{C}_1(\xi)$ and $\mathcal{C}_2(\xi)$ are positive random variables, and $g=g^{\xi,\gamma}$ is a solution to 
\[ g''(x) + (\xi(x)-\gamma)g(x)=0,\ x>0 \]
for $\gamma>\overline \gamma$, for a certain $\overline \gamma >0.$ 
For technical reasons, he additionally  requires $\gamma<\gamma^*$, i.e., the initial condition $w_0(x)$ must not decay too fast as $x$ tends to infinity. The technical assumption~\eqref{eq:ODE} thus entails being in the {\em supercritical regime} which corresponds to waves that move faster than the minimal speed. As his main result, Nolen obtains a central limit theorem for $m(t)$ in this case also, see \cite[Theorem~1.4]{nolen_random_medium_FKPP}. In this lingo, our set of initial conditions corresponds to the {\em critical regime}, and
Corollary \ref{cor: functional clt for m} (for the critical regime) also suggests that the randomness of the functional central limit theorem is already coming from the environment, and not necessarily due to the random initial condition. 

Furthermore, in \cite{No-11a} a corresponding invariance principle for the front has been derived in the case where the non-linearity in \eqref{eq:KPP1} is either ignition typ or bistable.

In \cite{HaNoRoRy-12}, on the other hand,  for the case of periodic instead of random $\xi,$ the authors have investigated the respective logarithmic correction, corresponding to \eqref{eq:homNonLinFront} in the homogeneous setting; here, the authors have been able to characterize the constant in front of the logarithmic correction as a certain minimizer.

In our main Theorem \ref{th: log-distance breakpoint median}, we establish a corresponding logarithmic upper bound for the difference \eqref{eq:log_delay_hom}, also in the setting of a random environment. However, our methods are currently too coarse for identifying a sharp  prefactor. Nevertheless, it is a natural question whether the logarithmic upper bound we derive captures the correct order at least. And indeed, in some sense, a partial positive answer to this questions is provided in the companion article \cite{CeDrSc-20}. There, the authors show that even in the  setting of \eqref{eq:KPP1} described above, for some potential satisfying our standing assumptions, $w_0=\mathds{1}_{(-\infty,0]},$ and all $\varepsilon\in(0,1/2),$ there exists a  sequence $(t_n)_{n\in\N}=(t_n(\varepsilon))_{n\in\N}$ such that $t_n\to\infty$ and 
	\begin{equation}
	\label{eq:transition_area_grows}
	 m_{t_n}^\varepsilon - \inf\{ x\geq 0: w(t_n,x)\leq 1-\varepsilon\} \in \Omega(\ln n);
	\end{equation}
in particular, the solution does not exhibit a uniformly bounded transition front.

\bigskip
As explained above, there is a profound connection between the PDEs we consider and branching Brownian motion, which is more involved that in the setting of constant $\xi.$ Indeed,
related results for branching random walk in random environment (BRWRE) have been recently derived in \cite{drewitz_cerny}. In this source the authors analyze the distribution of the maximal particle of a branching random walk in random environment, which itself is closely related to a discrete-space version of \eqref{eq:KPP1}. In particular, a corresponding logarithmic upper bound on the distance of the expected position of the maximal particle and the median of the distribution of the maximal particle is given. Furthermore, an invariance principle for the median of the position of the maximal particle of BRWRE is derived. It is therefore no surprise that, on the one hand, principal techniques we employ in this paper are generalizations and adaptations of respective discrete space analogues from \cite{drewitz_cerny}. On the other hand, we work under weaker independence assumptions and less assumptions on the non-linearity which is assumed to be in \eqref{eq:standard_condi}. 
What is more, we also obtain a more profound understanding of the scope of the methods employed, see in particular Section \ref{sec:vel}. 
In combination with the current article, somewhat of a completion of the picture is provided by \cite{CeDrSc-20}, by establishing a result with a similar flavor of Bramson's alluded to above.

\bigskip

Recall that directly before Theorem \ref{th:InvPAM} we have introduced our assumption \eqref{eq:VEL}, which reads $v_0 >v_c;$ from a technical point of view, this will be necessary for our change of measure argument to work. It is not hard to show that for a rich class of potentials, this condition is satisfied indeed, cf.\ Section \ref{sec:vel}. What is more, however, in Proposition \ref{le:v_c>v_0} below we also obtain a more profound understanding of the scope of the methods employed. I.e., there exist potentials $\xi$ fulfilling \eqref{eq:SASS}, but such that $v_0<v_c$ holds true. In this regime our methods do not apply, and it is an interesting and open question to obtain a more profound understand of this situation as well.

\subsection{Notational conventions}
\label{sec:convention}
We will frequently use sums of real-indexed quantities $A_x$, $x\in\R$. In this case, we write
\begin{align*}
\sum_{i=1}^x A_{i} &:= \sum_{i=1}^{\lfloor x\rfloor}A_{i} + A_x,\quad x\in[0,\infty)\setminus\N_0,
\end{align*}
where $\sum_{i=1}^0:=0$. This notion remains consistent if we also allow for additive constants $b\in\R,$ i.e. 
\begin{align*}
\sum_{i=1}^x (A_{i} + b) = \sum_{i=1}^{\lfloor x\rfloor} A_{i} +A_x+ b\cdot x = \sum_{i=1}^x A_{i} + \sum_{i=1}^x b,\quad x\in[0,\infty).
\end{align*}
Finally, we set
\begin{align*}
\sum_{i=x+1}^y A_{i} &:=
\begin{cases}
\sum_{i=1}^y A_{i} -\sum_{i=1}^x A_{i},& x\leq y, \\
\sum_{i=1}^x A_{i} -\sum_{i=1}^y A_{i}, & x>y,
\end{cases}
\quad x,y\in[0,\infty).
\end{align*}
A prime example is the quantity $A_{x}= \ln E_x\big[ e^{\int_0^{H_{\lceil x\rceil-1}} (\xi(B_s)-\es) \diff s} \big]$, where $H_y:=\inf\{t\geq0:B_t=y\}$. Indeed, by the strong Markov property we have $\ln E_x\big[ e^{\int_0^{H_{0}} (\xi(B_s)-\es) \diff s} \big]=\sum_{i=1}^{\lfloor x\rfloor}A_i + A_x$ for all $x\in[0,\infty)\setminus \N_0$. 

Furthermore, we will often use positive finite constants $c_1,c_2,\ldots$ in the proofs. This numbering is consistent within any of the proofs, and it is reset after each proof. On the
other hand, $C_1,C_2,\ldots$ will be used to denote positive finite constants that are fixed throughout
the article, and they will oftentimes depend on each other. 
Other constants like $c,C, \varepsilon,\delta$ etc.\ in the proofs  are used to compare certain quantities and are also reset after each proof.  

\bigskip

{\bf Acknowledgments:} We would like to thank Jiří Černý for useful discussions and remarks on a preliminary version of this paper.

\section{Some technical tools}
In this section we will introduce some further tools that will be helpful in the proof of the main results.

\subsection{Connection to branching processes}

We define a \emph{branching Brownian motion in random environment} (BBMRE) as follows: Conditionally on the realization of $\xi$ and for fixed $x\in\mathbb{R}$, consider an initial particle starting at a point $x$ and moving as a standard Brownian motion $(B_t)_{t\geq0}$ on $\mathbb{R}$. While at site $y$, the particle dies at rate $\xi(y)$. More precisely, for an exponentially distributed random variable $S$ with parameter one, independent of everything else, the first particle dies at time $\inf\big\{t\geq0: S<\int_0^t\xi(B_s)\diff s \big\}$. 
When the initial particle dies, it gives birth to $k$ new particles with probability $p_k$, $k\in\N$ (see \eqref{eq:p_k} below for the precise assumptions on the $p_k$). The new particle(s) start their evolution at the site where their parent particle had died, and they evolve independently of everything else and according to the same stochastic behavior as their parent. Note that on the one hand we assume $p_0=0$, so genealogies do not die out at any finite time. On the other hand, we allow $p_1>0,$ i.e., it is possible to die and give birth to one descendant. This implies that a particle at site $y$ branches into more than one particle with rate $\xi(y)(1-p_1)$. We denote the corresponding probability measure by $\prob_x^\xi$ and write $\Eprob_{x}^\xi$ for the respective expectation.
By $N(t)$ we denote the set of particles alive at time $t$. For $\nu\in N(t)$ we let $(X_s^\nu)_{s\in[0,t]}$ be the spatial trajectory of the genealogy of ancestral particles of $\nu$ up to time $t$. For a Borel set $A\subset\R$ and $y \in \R$ we define
\begin{align} \label{eq:NtA}
N(t,A)&:= \big\{ \nu\in N(t): X_t^\nu \in A \big\}\quad \text{ and } \quad N^\leq (t,y) := |N(t,(-\infty,y])|,
\end{align}
i.e., the set of particles which are in $A$ at time $t$, and the number of particles to the left of $y$ at time $t.$

Let us now introduce a set of functions, formulated in terms of the  probability generating function of the sequence $(p_k)_{k\in\N}$ as above. More precisely, assume $(p_k)_{k\in\N}$ and $F=F^{(p_k)_{k\in\N}}$ to fulfill
\begin{equation}
\label{eq:p_k}
\tag{PROB}
\begin{split}
p_k\in[0,1]\ \forall k\in\mathbb{N},&\quad \sum_{k=1}^\infty p_k=1,\quad \sum_{k=1}^\infty kp_k  = 2,\quad m_2 := \sum_{k=1}^\infty k^2 p_k<\infty; \\
F(u)&=1-u - \sum_{k=1}^\infty p_k(1-u)^k,\ u\in[0,1].
\end{split}
\renewcommand*{\theHequation}{notag11.\theequation}
\end{equation}
Note that such $F$ always fulfills \eqref{eq:standard_condi}, is smooth and  strictly concave on $(0,1)$. 

A fundamental result is the connection of the solution to \eqref{eq:KPP1} and the expected number of particles of the branching process. This is sometimes referred to as the McKean representation. The respective result in the homogeneous setting has been proven in \cite{branch_markov_I} and \cite{mckean} also.

\begin{proposition}
	\label{prop:mckean1}
	Let $w_0\in\mathcal{I}_{\text{F-KPP}}$, $\xi$   be H\"{o}lder continuous and $F$ fulfill \eqref{eq:p_k}. Then $\p$-a.s.\ 
	the solution to \eqref{eq:KPP1} is given by
	\begin{equation}
	\label{eq:McKean}
	\tag{McKean}
	w(t,x) = 1- \Eprobth_x^\xi \Big[ \prod_{\nu\in N(t)} \big(1-w_0(X_t^\nu)\big) \Big],\quad (t,x)\in [0,\infty) \times\R.
	\renewcommand*{\theHequation}{notag12.\theequation}
	\end{equation}
\end{proposition}
For the proof see Section~\ref{sec:mckean_proof}. 
\begin{remark} \label{rem:McKeanHeaviside}
	We will frequently use the application of the latter result to  functions $w_0=\mathds{1}_{(-\infty,0]}$,  resulting in
	\begin{align*}
	w(t,x) &= 1- \Eprob_x^\xi \Big[ \prod_{\nu\in N(t)} \mathds{1}_{(0,\infty)}(X_t^\nu) \Big] = 1- \prob_x^\xi \big( X_t^\nu > 0\ \forall\nu\in N(t) \big)
	=\prob_x^\xi \big( N^\leq (t,0)\geq1 \big). 
	\end{align*}
\end{remark}

Furthermore, we will need the so-called \emph{many-to-few} lemma, which breaks down the moments of the branching process to a functional of the Brownian paths. For our purposes, it suffices to state it up to the second moments (\emph{many-to-one} and \emph{many-to-two formula}). 
\begin{proposition}
	\label{prop:many-to-few}
	Let $(p_k)_{k\in\N}$ fulfill \eqref{eq:p_k} and let  $\varphi_1$, $\varphi_2: \, [0,\infty) \to [-\infty,\infty]$ be c\`{a}dl\`{a}g functions satisfying $\varphi_1\leq\varphi_2$. Then the first and second moments of the number of particles in $N(t)$ whose genealogy stays between $\varphi_1$ and $\varphi_2$ in the time interval $[0,t]$ are given by
	\begin{align}
	\Eprobth_x^\xi&\Big[ \big| \nu\in N(t): \varphi_1(s)\leq X_s^\nu \leq \varphi_2(s)\ \forall s\in[0,t] \big| \Big] \notag \\
	&= E_x\Big[ \exp\Big\{ \int_0^t \xi(B_r)\diff r \Big\};\varphi_1(s)\leq B_s\leq \varphi_2(s)\ \forall s\in[0,t] \Big] \quad  \label{eq:f-k-1}\tag{FK-1}
	\renewcommand*{\theHequation}{notag13.\theequation}
	\end{align}
	and
	\begin{align}
	\Eprobth_x^\xi\Big[& \big| \nu\in N(t): \varphi_1(s)\leq X_s^\nu \leq \varphi_2(s)\ \forall s\in[0,t] \big|^2 \Big] \notag \\
	&= E_x\Big[ \exp\Big\{ \int_0^t \xi(B_r)\diff r \Big\};\varphi_1(s)\leq B_s\leq \varphi_2(s)\ \forall s\in[0,t] \Big] \notag \\
	&\qquad + (m_2-2)\int_0^t E_x\Big[ \exp\Big\{ \int_0^s\xi(B_r)\diff r \Big\}\, \xi(B_s) \mathds{1}_{ \{\varphi_1(r)\leq B_r\leq \varphi_2(r)\ \forall 0\leq r\leq s \} } \Big. \label{eq:f-k-2}\tag{FK-2} \\
	&\qquad\quad\times\Big. \Big( E_{y}\big[\exp\big\{ \int_0^{t-s}\xi(B_r)\diff r \big\}\, \mathds{1}_{ \{ \varphi_1(r+s)\leq B_r\leq \varphi_2(r+s)\ \forall 0\leq r\leq t-s\}}\big] \Big)^2_{|_{y =B_s}} \Big] \diff s,\notag
	\renewcommand*{\theHequation}{notag14.\theequation}
	\end{align}
	respectively.
\end{proposition}
The proof of the first identity follows from \cite[Section 4.1]{harris_robert_multiple_spines}. 
The second identity can be shown  by using \cite[Lemma 1]{harris_robert_multiple_spines}, conditioning on the first splitting of the so-called ‘‘spines'', similarly to the proof of \cite[(2.6)]{guen} for $n=2$ there. Indeed, one has to consider branching Brownian motion instead of branching random walk and replace binary branching by general branching. Then the expectation of the quantity in display \cite[(2.15)]{guen} turns into the second summand in \eqref{eq:f-k-2}, because the first splitting rate of the two ‘‘spine  particles''  at site $y$ is $(m_2-2)\xi(y)$.  

As a consequence of Proposition~\ref{prop:many-to-few}, the solution to \eqref{eq:PAM} can be expressed by a functional of the branching process, i.e.\ we have that the solution to \eqref{eq:PAM} is given by 
\begin{align*}
u(t,x)&=E_x\Big[ \exp\Big\{ \int_0^{t} \xi(B_s)\diff s \Big\}\, u_0(B_t) \Big]=\Eprob_x^\xi\Big[ \sum_{\nu\in N(t)} u_0(X_t^\nu) \Big]. 
\end{align*}
As a special case for  $u_0=\mathds{1}_{(-\infty,0]}$, this turns into
\begin{align}
u(t,x)=\Eprob_x^\xi\big[ N^\leq(t,0) \big] &=E_x\Big[  \exp\Big\{ \int_0^{t} \xi(B_s)\diff s \Big\}; B_t\leq 0 \Big].
\label{eq:F-K}
\end{align}

\subsection{Change of measure}
The first tool is a change of measure which makes it typical for the Brownian motion in the Feynman-Kac formula started at $tv,$ some $v \ne 0,$ to be close to the origin at time $t.$
For this purpose, let $(\xi(x))_{x\in\mathbb{R}}$ be as in \eqref{eq:POT} and define the shifted potential
\begin{equation*}
\zeta:=\xi-\es.
\end{equation*}
Then $\p$-a.s., 
\begin{equation} \label{eq:zetaBd}
    \zeta(x)\in[-(\es-\ei),0] \quad \forall x\in\mathbb{R}.
\end{equation} 
We will oftentimes write 
\begin{equation} \label{eq:hit}
H_y:=\inf\big\{t\geq0:B_t=y\big\}, \quad y\in\mathbb{R}, \quad \text{ and } \quad \tau_i:=H_{i-1}-H_{i}, \quad i \in \Z,
\end{equation}
for the first hitting times and their pairwise differences.
The above shift of $\xi$ gives rise to a change of measure 
which will play a crucial role in the following. For $x,y \in \R$ as well as $\eta \le  0$ define the probability measures $P_{x,y}^{\zeta,\eta}$ via
\begin{align}
P_{x,y}^{\zeta,\eta}(A)&:=\frac{1}{Z_{x,y}^{\zeta,\eta}}E_x\Big[ \exp\Big\{\int_{0}^{H_{x-y}}\big(\zeta(B_s)+\eta\big)\diff s\Big\};A \Big], \quad A\in\sigma\big(B_{t\wedge H_{x-y}}:t\geq 0\big),\label{eq:def_P_x_zeta}
\end{align}
with normalizing constant
\begin{align*}
Z_{x,y}^{\zeta,\eta}&:=E_x\Big[ \exp\Big\{\int_{0}^{H_{x-y}}\big(\zeta(B_s)+\eta\big)\diff s\Big\} \Big] \in (0,\infty),
\end{align*}
where $P_x$ and $E_x$, $x\in\R$, are defined below \eqref{eq:feynman_kac}. 
For $A\in\sigma\big(B_{t\wedge H_{x-y}}:t\geq 0\big)$, using the strong Markov property at time $H_{x-y}$, we infer $P_{x,y}^{\zeta,\eta}(A)=P_{x,y'}^{\zeta,\eta}(A)$ for all $y'\geq y$. Thus, by Kolmogorov's extension theorem (see e.g.\ \cite[Theorem 2.4.3]{tao_lectures}),
\begin{equation}
 \label{eq:P_x}
\big(P_{x,y}^{\zeta,\eta}\big)_{y\geq0} \text{ can be extended to a unique probability measure }P_x^{\zeta,\eta} \text{ on }\sigma\big( B_t:t\geq 0 \big).
\end{equation}
 We write $E_x^{\zeta,\eta}$ for the corresponding expectation and introduce the logarithmic moment generating functions
\begin{align}
L_{x}^{\zeta}(\eta)&:=\ln E_{x}\Big[ \exp\Big\{ \int_0^{H_{\lceil x \rceil -1}}\big( \zeta(B_s)+\eta \big)\diff s \Big\} \Big],\quad x>0,\label{eq:LogMG_x}\\
\overline{L}_{x}^\zeta(\eta)&:=\frac{1}{x}\sum_{i=1}^{x} L_{i}^{\zeta}(\eta) = \frac{1}{x} \ln E_x\Big[ \exp\Big\{ \int_0^{H_0}(\zeta(B_s)+\eta)\diff s \Big\} \Big],\quad x>0,\notag
\end{align}
where we recall the notation introduced in Section \ref{sec:convention}, and
where the last equality is due to the Markov property. In addition, set
\begin{equation}
\label{eq:expectedLogMG}
L(\eta):=\E\big[ L_{1}^{\zeta}(\eta) \big].
\end{equation} 
Due to \eqref{eq:zetaBd},
for any $\eta\le  0$
the quantities above are well-defined, and it is easy to check that in this case and under \eqref{eq:POT}, the expressions defined in \eqref{eq:LogMG_x}--\eqref{eq:expectedLogMG} are finite. We have the following useful properties.
\begin{lemma}\label{le:expectedLogMGfct}
	\begin{enumerate}
		\item  The function $(-\infty,0)\ni\eta\mapsto L(\eta)$ is infinitely  differentiable  and its derivative $L'(\eta)$ is positive and monotonically strictly increasing. 
		\item
		We have $\p\text{-a.s.}$ that
		\begin{equation}
		\label{eq:LLNEmpLogMG}
		\lim_{x\to\infty} \overline{L}_{x}^\zeta(\eta) = L(\eta)\quad\text{for all }\eta\leq  0.
		\end{equation} 		
		\item $L'(\eta)\downarrow 0$ as $\eta\downarrow -\infty$ 
		\item For every $v>v_c:=\frac{1}{L'(0-)}$ 
		(where $\frac{1}{+\infty}:=0$), which we call \emph{critical velocity}, there exists a 
		\begin{equation}
		\label{eq:legTraExpectedLogMGfct_1}
		\text{ unique solution $\overline{\eta}(v)<0$ to the equation }
		L'( \overline{\eta}(v) ) = \frac{1}{v}.
		\end{equation}
		$\overline{\eta}(v)$ can be characterized as the unique maximizer to $(-\infty,0]\ni \eta\mapsto \frac{\eta}{v} - L(\eta) $, i.e. 
		\begin{equation}\label{eq:legTraExpectedLogMGfct_2}
		  \sup_{\eta\leq 0} \Big( \frac{\eta}{v} - L(\eta) \Big)= \frac{\overline{\eta}(v)}{v} - L\big( \overline{\eta}(v) \big).
		\end{equation} 
		The function $(v_c,\infty)\ni v\mapsto \overline{\eta}(v)$ is continuously differentiable and strictly decreasing. 
	\end{enumerate}
\end{lemma}

\begin{proof}[Proof of Lemma \ref{le:expectedLogMGfct}]
	\begin{enumerate}	 
		\item This follows from Lemma~\ref{le:log mom fct and derivatives}.
		\item
		By \cite[Theorem 7.5.1]{Fr-85}, for every $\eta\leq 0$  we get  $\p$-a.s.\ that
		\( \lim_{x\to\infty} \overline{L}^\zeta_x(\eta) = \E\big[  L_1^\zeta(\eta)\given \mathcal{F}^\zeta_\text{inv} \big],\)
		where $ \mathcal{F}^\zeta_\text{inv}$ is the $\sigma$-algebra of all $\p$-invariant sets. Due to our standing assumptions, the family $\zeta(x),$ $x \in \R,$ is mixing and thus ergodic. Thus, $ \mathcal{F}^\zeta_\text{inv}$ is $\p$-trivial, i.e., $\E\big[ L_1^\zeta(\eta)\given \mathcal{F}^\zeta_\text{inv} \big]=L(\eta)$. By continuity of the functions $\overline{L}^\zeta_x$ and $L$, the statement follows.
		\item  We note that $L$ is strictly increasing and strictly convex on $(-\infty, 0)$ by (a) and 
		\[ L(\eta)\geq\E\Big[ \ln E_1\big[ e^{-(\es-\ei-\eta)H_0} \big] \Big]= -\sqrt{2(\es-\ei-\eta)}\quad\text{for all }\eta\leq0, \]
		where the equality is due to \cite[(2.0.1), p.~204]{handbook_brownian_motion}. 
		Thus, we infer that its derivative $L'(\eta)$ tends to $0$ as $\eta\to-\infty$.
		\item Using that $1/v_c>1/v$ (where $\frac{1}{0}:=+\infty$) and the fact that $L'$ is strictly increasing with $L'(\eta)\downarrow0$ for $\eta\downarrow-\infty$, we can find a unique $\overline{\eta}(v)<0$  such that $L'(\overline{\eta}(v))=1/v$, giving \eqref{eq:legTraExpectedLogMGfct_1}. On the other hand, \eqref{eq:legTraExpectedLogMGfct_2} is a direct consequence of \eqref{eq:legTraExpectedLogMGfct_1} and standard properties of the Legendre transformations of strictly convex functions. 
		Because $L'$ is strictly increasing and smooth on $(-\infty,0)$, it has a strictly increasing inverse function $(L')^{-1}$, which is differentiable on $(0,1/v_c)$. By \eqref{eq:legTraExpectedLogMGfct_1}, for $v>v_c$ we have $\overline{\eta}(v)=(L')^{-1}(1/v)$ and thus by the formula for the derivative of the inverse function we get
		\[\overline{\eta}'(v)=-\frac{1}{v^2}\cdot \frac{1}{L''(\overline{\eta}(v))}.\]
		Because the right-hand side of the latter display is continuous in $v$ and negative, we conclude.
	\end{enumerate}
\end{proof}
We use the standard notation \( L^*: \, \R \to (-\infty, \infty]\) to denote the Legendre transformation
\begin{align*}
  v & \mapsto \sup_{\eta\leq 0} \big( \eta v - L(\eta) \big)
\end{align*}
of $L$.
 Lemma~\ref{le:expectedLogMGfct} entails that
\begin{align}
\label{eq:expec_Leg_LMG}
L^*(1/v) = \frac{\overline{\eta}(v)}{v} - L\big( \overline{\eta}(v) \big).
\end{align}
In the next part of this section, we are interested in a suitable \emph{tilting parameter} $\eta_x^\zeta(v)$ 
such that  
\begin{equation}\label{eq:probTilting}
E_x^{\zeta,\eta_x^{\zeta}(v)}\big[ H_0 \big] = \frac{x}{v},\quad x>0,\ v>0,
\end{equation}
holds true (setting $\eta_x^\zeta(v):=0$ if no such parameter exists). 
For  $\eta_x^\zeta(v)$ fulfilling \eqref{eq:probTilting} we observe that under
$P_x^{\zeta,\eta_x^{\zeta}(v)}$, 
the Brownian motion is tilted to have time-averaged velocity $v$ until it reaches the origin. In the remaining part of this section, see Lemma \ref{le:concEtaEmpLT},  we will show that for suitable $v$ and $x$ large enough, a tilting parameter as postulated in \eqref{eq:probTilting} actually exists. Furthermore, we will show that the random parameter $\eta_x^\zeta(v)$  concentrates around the deterministic  $\overline{\eta}(v)$ defined in \eqref{eq:legTraExpectedLogMGfct_1}. The last result is a perturbation estimate for $\eta_x^\zeta(v)$ in $x,$ cf.\ Lemma \ref{le:PertEta_n}.

\subsection{Concentration inequalities}
We have the following result regarding the existence (or, more precisely, negativity) and concentration properties of the postulated parameter $\eta_x^\zeta(v)$.
\begin{lemma}\label{le:concEtaEmpLT}
\begin{enumerate}
\item \label{item:exist}
For every $v>v_c$ there exists a finite random variable $\mathcal{N}=\mathcal{N}(v)$ such that for all $x\geq\mathcal{N}$ the solution $\eta_x^\zeta(v)<0$ to \eqref{eq:probTilting} exists. 
\item \label{item:conc}
For each $q\in\mathbb{N}$ and each compact interval $V\subset(v_c,\infty)$, there exists $\Cl{const:concEtaEmpLT}:=\Cr{const:concEtaEmpLT}(V,q)\in(0,\infty)$ such that
	\begin{equation}
	\label{eq:concEtaEmpLT_cont}
	\p\Big( \sup_{v\in V}\sup_{x\in[n,n+1)} |  \eta_x^\zeta(v) - \overline{\eta}(v)| \geq \Cr{const:concEtaEmpLT}\sqrt{\frac{\ln n}{n}} \Big) \leq \Cr{const:concEtaEmpLT} n^{-q}\qquad \text{for all }n\in\mathbb{N}.
	\end{equation}
\end{enumerate}
\end{lemma}
\begin{proof}
We recall that due to 
Lemma~\ref{le:log mom fct and derivatives}, the tilting parameter $\eta_x^\zeta(v)$ can alternatively be characterized as the unique solution $\eta_x^\zeta(v) \in (-\infty,0)$ to  
\begin{equation}
\label{eq:EmpLogMGDerivTilt}
\big(\overline{L}^\zeta_x\big)'(\eta_x^\zeta(v)) = \frac{1}{v},
\end{equation}
if the solution exists, and $\eta_x^\zeta(v)=0$ otherwise. 
We start with noting that Part \ref{item:exist} directly follows from Part \ref{item:conc}. Indeed, let $A_n$, $n\in\N$, be the event in the probability on the left-hand side of \eqref{eq:concEtaEmpLT_cont}. Then $\sum_n \p(A_n)<\infty$ for $q\geq 2$. By the first Borel-Cantelli lemma, $\p$-a.s.\ 
 only finitely many of the $A_n$ occur. In combination with   
	the fact that $\overline \eta(v) <0$, cf.\ \eqref{eq:legTraExpectedLogMGfct_1}, this implies that  $\p$-a.s., the value of $\eta_x^\zeta(v)$ can only vanish for  $x>0$ small enough. In particular, we deduce the existence of a $\p$-a.s.\ finite random variable $\mathcal N$ as postulated. \\
	Hence, it remains to show \eqref{eq:concEtaEmpLT_cont}.
For this purpose, in the following lemma we investigate the fluctuations of the  functions through which the parameters
 $\eta_x^\zeta(v)$ and $\overline{\eta}(v)$ are implicitly defined; we will then infer the desired bounds on the fluctuations of the parameters themselves through perturbation estimates for these functions.
\begin{lemma}\label{le:concEmpLT}
	For every compact interval $\triangle\subset(-\infty,0)$ and each $q \in \N,$ there exists a constant $\Cl{const:concEmpLT}=\Cr{const:concEmpLT}(\triangle,q)\in(0,\infty)$ such that
	\begin{equation}
	\label{eq:concEmpLT_cont}
	\p\Big( \sup_{\eta\in\triangle}\sup_{x\in[n,n+1)} \Big| \big( \overline{L}_x^\zeta\big)'(\eta) - L'(\eta) \Big| \geq \Cr{const:concEmpLT}\sqrt{\frac{\ln n}{n}} \Big) \leq \Cr{const:concEmpLT} n^{-q}\qquad \text{for all }n\in\mathbb{N}.
	\end{equation}
\end{lemma}
In order not to hinder the flow of reading, we
 postpone the proof of this auxiliary result to the end of the proof of Lemma~\ref{le:concEtaEmpLT}  and finish the proof of Lemma~\ref{le:concEtaEmpLT} (b) first.  Let $q\in\N$ and $V\subset(v_c,\infty)$ be a compact interval. By Lemma~\ref{le:log mom fct and derivatives}, for each compact $\triangle\subset(-\infty,0)$ we have   $\p$-a.s.,
\begin{align}
    \label{eq:secQDerBd}
    \begin{split}
     	\Cr{constDerLMG}^{-1} \leq \inf_{\eta \in \triangle} L''(\eta) &\leq \sup_{\eta \in \triangle} L''(\eta) \leq 	\Cr{constDerLMG},  \\
     	\Cr{constDerLMG}^{-1} \leq \inf_{x\geq1}\,\inf_{\eta\in\triangle} \big(\overline{L}_x^\zeta\big)''(\eta) &\leq \sup_{x\geq1}\,\sup_{\eta\in\triangle} \big(\overline{L}_x^\zeta\big)''(\eta) \leq 	\Cr{constDerLMG}. 
    \end{split}
\end{align}
Therefore, and because the function $V\ni v\mapsto\overline{\eta}(v)$ is strictly decreasing  by Lemma~\ref{le:expectedLogMGfct}, it is possible to find $N=N(V)\in\mathbb{N}$ and a compact interval $\triangle=\triangle(N,V) \subset (-\infty,0),$ where for notational convenience we write
\begin{equation} \label{eq:triangleEta}
V= [v_*, v^*] \quad \text{ and } \quad\triangle = [\undereta,\overeta],
\end{equation}
 such that, using standard calculus for sets,
\begin{equation*}
    \big(\overline{\eta}(V)- \Cr{const:concEmpLT}	\Cr{constDerLMG}\sqrt{\ln n /n}\big)\cup\big( \overline{\eta}(V)+ \Cr{const:concEmpLT}	\Cr{constDerLMG}\sqrt{\ln n /n} \big) \subset \triangle\quad\text{ for all }n\geq N. 
\end{equation*}
Let $n\geq N$ and assume that the complement of the event on the left-hand side of \eqref{eq:concEmpLT_cont},
\begin{equation} \label{eq:LpertBd}
 \sup_{\eta\in\triangle}\, \sup_{x\in[n,n+1)} \big| \big(\overline{L}_x^\zeta\big)'(\eta) - L'(\eta) \big|< \Cr{const:concEmpLT}\sqrt{\frac{\ln n}{n}},
 \end{equation}
occurs.
On this event, for all $v\in V$ and all $x\in[n,n+1),$
\begin{equation} \label{eq:LbarPertEst}
 \big(\overline{L}_x^\zeta\big)'\big(\overline{\eta}(v)-\Cr{const:concEmpLT}	\Cr{constDerLMG}\sqrt{\ln n/ n}\big) \leq \frac{1}{v} \leq \big(\overline{L}_x^\zeta\big)'\big(\overline{\eta}(v)+\Cr{const:concEmpLT}	\Cr{constDerLMG}\sqrt{\ln n/n}\big)
 \end{equation}
and thus, due to the strict monotonicity of $(\overline{L}_x^\zeta)'$ implied by \eqref{eq:secQDerBd}, there exists a unique $\eta_x^\zeta(v)\in \triangle$ such that $\big(\overline{L}_x^\zeta\big)'(\eta_x^\zeta(v))=1/v$.  Due to \eqref{eq:LbarPertEst}, still assuming \eqref{eq:LpertBd}, we have
\[ \sup_{v\in V}\, \sup_{x\in[n,n+1)} | \overline{\eta}(v) - \eta_x^\zeta(v) |\leq  \Cr{const:concEmpLT}	\Cr{constDerLMG}\sqrt{\frac{\ln n}{n}}. \]
Thus, for $n\geq N$, choosing $\Cr{const:concEtaEmpLT}>2\Cr{const:concEmpLT}	\Cr{constDerLMG}$, the probability in \eqref{eq:concEtaEmpLT_cont} is bounded by the right-hand side of \eqref{eq:concEmpLT_cont}, which finishes the proof.
\end{proof}
It remains to prove Lemma~\ref{le:concEmpLT}.
\begin{proof}[Proof of Lemma~\ref{le:concEmpLT}]
	Applying the strong Markov property, we get 
	\[ x\big( \overline{L}_x^{\zeta} \big)'(\eta) = E_x^{\zeta,\eta}(H_0) =   E_{x}^{\zeta,\eta}\big[ H_{\lfloor x\rfloor} \big]  + E_{\lfloor x\rfloor}^{\zeta,\eta}\big[ H_0 \big] =   E_{x}^{\zeta,\eta}\big[ H_{\lfloor x\rfloor} \big] + \lfloor x\rfloor \big( \overline{L}_{\lfloor x\rfloor}^{\zeta} \big)'(\eta).  \]
	Furthermore,
	$0\leq E_{x}^{\zeta,\eta}\big[ H_{\lfloor x\rfloor} \big]=\big(L_x^\zeta\big)'(\eta)\leq \Cr{constDerLMG}$ by \eqref{eq:L_x'} and Lemma~\ref{le:log mom fct and derivatives} b), and thus also $0\leq \big( \overline{L}_x^\zeta\big)'(\eta)=\frac{1}{x}\sum_{i=1}^x\big(L_i^\zeta\big)'(\eta) \leq \Cr{constDerLMG}$  for all $x\geq 1$ and all $\eta\in\triangle$, $\p$-a.s. As a consequence, we get that for all $x\geq1,$
	\[ \big| \big( \overline{L}_x^{\zeta} \big)'(\eta) - L'(\eta) \big|\leq \big| \big( \overline{L}_{\lfloor x\rfloor}^{\zeta} \big)'(\eta) - L'(\eta) \big| + \frac{2\Cr{constDerLMG}}{\lfloor x\rfloor}. \]
	It is therefore enough to prove 
	\begin{equation}
	    \label{eq:concEmpLT_discrete}
	    \p\Big( \sup_{\eta\in\triangle} \Big| \big( \overline{L}_n^\zeta \big)'(\eta)  - L'(\eta) \Big| \geq \Cr{const:concEmpLT}\sqrt{\frac{\ln n}{n}} \Big) \leq \Cr{const:concEmpLT} n^{-q}\qquad \text{for all }n\in\mathbb{N}.
	\end{equation}
	For each $\eta\in\triangle$, $((L_i^\zeta)'(\eta)-L'(\eta))_{i\in\Z}$ is a family of stationary, centered and bounded random variables. Furthermore, they fulfill the exponential mixing condition \eqref{eq:expMix3} due to Lemma~\ref{le:expMixLi}. Since $\sigma((L_i^\zeta)'(\eta):i\geq k)\subset\sigma(\xi(x):x\geq k-1)$ and \eqref{eq:MIX}, setting $Y_i:=(L_i^\zeta)'(\eta)-L'(\eta)$, the left-hand side in \eqref{eq:rioDepExp} is bounded by some constant $c_2>0$, uniformly for every $i$. Then setting $m_i:=c_1$,  condition \eqref{eq:rioDepExp} is fulfilled and we can apply the Hoeffding-type inequality from Corollary~\ref{cor:rioDepExp} to get $c_2>0$ such that 
	\[ \p\Big(  \Big| \big( \overline{L}_n^\zeta\big)'(\eta)  - L'(\eta) \Big| \geq c_2\sqrt{\frac{\ln n}{n}} \Big) \leq c_2n^{-q-1}\qquad \text{for all }\eta\in\triangle\text{ and  all }n\in\mathbb{N}. \]
	Let $\triangle_n := (\triangle \cap \frac1n \Z)\cup \{\undereta,\overeta\}$, recalling the notation of \eqref{eq:triangleEta}. Because $|\triangle\cap\frac1n \Z|\leq n\cdot \text{diam}(\triangle)+1$, taking advantage of the previous display we infer
	\begin{align*}
	\p\Big( \sup_{\eta\in\triangle_n} &\Big| \big( \overline{L}_n^\zeta\big)'(\eta) - L'(\eta) \Big| \geq c_2\sqrt{\frac{\ln n}{n}} \Big) \\
	&\leq |\triangle_n|\cdot \sup_{\eta\in\triangle_n} \p\Big(  \Big| \big( \overline{L}_n^\zeta \big)'(\eta) - L'(\eta) \Big| \geq c_2\sqrt{\frac{\ln n}{n}} \Big)  \leq c_3(\triangle) n^{-q}, \qquad \text{for all }n\in\mathbb{N}.
	\end{align*}
	By Lemma~\ref{le:log mom fct and derivatives} b), we have $\p$-a.s.\  $\sup_n\sup_{\eta\in\triangle}\big| \big(\overline{L}_n^\zeta(\eta)\big)''\big|\vee \big|L''(\eta)\big|\leq \Cr{constDerLMG}$. Thus, the mean value theorem entails that
	\begin{align*}
	\sup_{\eta\in\triangle} \big| \big( \overline{L}_n^\zeta\big)'(\eta)  - L'(\eta) \big| &\leq \sup_{\eta\in\triangle_n} \big| \big( \overline{L}_n^\zeta\big)'(\eta)  - L'(\eta) \big| + \frac{2\Cr{constDerLMG}}{n},
	\end{align*}
and thus we find $\Cr{const:concEmpLT}>0$ such that  \eqref{eq:concEmpLT_discrete} and hence \eqref{eq:concEmpLT_cont} hold true.
\end{proof}

In what comes below, many results will implicitly depend on the choice of compact intervals $V$ and $\triangle,$ which have already occurred before.
Thus, in order to avoid ambiguity and due to assumption \eqref{eq:VEL}, we will now 
\begin{equation}
\label{eq:defVTriangle}
\begin{gathered}
\text{fix arbitrary compact intervals $V\subset (v_c,\infty)$ and $\triangle=\triangle(V)\subset(-\infty,0)$ such that } \\
\text{$v_0\in\text{int}(V)$ and $\overline{\eta}(V)\subset\text{int}(\triangle)$}.
\end{gathered}
\end{equation}
Furthermore, due to Lemma~\ref{le:concEtaEmpLT},  there exists a $\p$-a.s.\  finite random variable $\Cl[index_prob]{constHn}=\Cr{constHn}(\xi,\Cr{const:concEtaEmpLT}(V,2))$ such that  
\begin{equation}
\label{eq:H_n}
\mathcal{H}_x:=\mathcal{H}_x(V) := \big\{ \eta_x^\zeta(v) \in \triangle \text{ for all }v\in V \big\} \quad\text{ occurs for all }x\geq \Cr{constHn}.
\end{equation}
We write
\begin{equation*}
\big( \overline{L}_x^\zeta \big)^*\Big(\frac{1}{v}\Big) = \sup_{\eta<0}\Big( \frac{\eta}{v} - \overline{L}^\zeta_x(\eta) \Big) = \frac{\eta_x^\zeta(v)}{v} - \overline{L}^\zeta_x(\eta_x^\zeta(v)),\quad x\geq1,
\end{equation*}
for  Legendre transformation
of the weighted averages.
We also recall to set $\eta_x^\zeta(v)=0$, if there is no solution $\eta_x^\zeta(v)\in\triangle$ to \eqref{eq:EmpLogMGDerivTilt}; note that this can only happen on $\mathcal{H}_x^c$. 

In order to show an  invariance principle for the Legendre transformation $( \overline{L}_x^\zeta)^*$ in the following section, 
 we now derive a perturbation result on the tilting parameter $\eta_x^\zeta(v)$ in $x.$

\begin{lemma}
	\label{le:PertEta_n}
	There exists a constant $\Cl{const_eta_pert}>0$ such that  $\p$-a.s., for all $x \in (0,\infty)$ large enough, uniformly in $v\in V$ and $0\leq h\leq x$,
	\begin{equation}
	\label{eq:PertEta_n}
	\big| \eta_{x}^\zeta(v) - \eta_{x+h}^\zeta(v)   | \leq \Cr{const_eta_pert}\frac{h}{x}.
	\end{equation}
\end{lemma}
\begin{proof}
	By Lemma~\ref{le:concEtaEmpLT} we can choose $x$ large enough such that  $\eta_y^\zeta(v)\in\triangle$ for all $y\geq x$ and all $v\in V$. For $h=0$, the statement is obvious. For $0<h\leq x$, it suffices to show that there exists $c_1>0$ such that 
	\begin{equation}
	\label{eq:eta_nPert}
	\sup_{\eta\in\triangle} \big| \big( \overline{L}_{x+h}^\zeta  \big)'(\eta) - \big( \overline{L}_{x}^\zeta  \big)'(\eta) \big| \leq c_1\frac{h}{x}.
	\end{equation}
	Indeed, using \eqref{eq:EmpLogMGDerivTilt} we can write 
	\begin{align*}
	\big( \overline{L}_{x+h}^\zeta  \big)'(\eta_{x+h}^\zeta(v)) - \big( \overline{L}_{x}^\zeta \big)'(\eta_{x+h}^\zeta(v)) &= \big( \overline{L}_{x}^\zeta  \big)'(\eta_{x}^\zeta(v)) - \big( \overline{L}_{x}^\zeta \big)'(\eta_{x+h}^\zeta(v))  \\
	&= \big( \overline{L}_{x}^\zeta \big)''(\widetilde{\eta})(\eta_{x}^\zeta(v)-\eta_{x+h}^\zeta(v) )
	\end{align*}
	for some $\widetilde{\eta}\in\triangle$ between $\eta_{x}^\zeta(v)$ and $\eta_{x+h}^\zeta(v)$. By the second display in  \eqref{eq:secQDerBd} we know that   $\p$-a.s.\  $\inf
	_{\eta\in\triangle,x\geq 1}(\overline{L}_x^\zeta\big)''(\eta)\geq \Cr{constDerLMG}^{-1}$. Using this, inequality \eqref{eq:PertEta_n} is a direct consequence of \eqref{eq:eta_nPert} with $\Cr{const_eta_pert}:=c_1\Cr{constDerLMG}$.
	 To prove \eqref{eq:eta_nPert}, recall that for all $\eta\in\triangle$, $x\geq 1,$ and $0< h\leq x,$ by the strong Markov property applied at time $H_x,$
	\begin{align*}
	\big( \overline{L}_{x+h}^\zeta  \big)'(\eta) - \big( \overline{L}_{x}^\zeta  \big)'(\eta)  &= \frac{1}{x+h} \big( E_{x+h}^{\zeta,\eta}\big[ H_x \big] + E_x^{\zeta,\eta}[H_0]\big)- \frac{1}{x} E_x^{\zeta,\eta}[H_0] \\
	&=-\frac{h}{x+h}\big( \overline{L}_{x}^\zeta  \big)'(\eta) + \frac{h}{x+h}\frac1h E_{x+h}^{\zeta,\eta}[H_x].
	\end{align*}
	
	Finally, recall that by  Lemma~\ref{le:log mom fct and derivatives} there exists $\Cr{constDerLMG}=\Cr{constDerLMG}(\triangle)>0$ such that $\p$-a.s.\ we have $\sup_{\eta\in\triangle,x>0}\big|\big( \overline{L}_x^\zeta \big)'(\eta)\big|\leq \Cr{constDerLMG}$. By exactly the same argument used for the proof of the latter inequality (see proof of \eqref{le:log mom fct and derivatives}), one can show that also $\sup_{\eta\in\triangle,x,h>0}\big| \frac1h E_{x+h}^{\zeta,\eta}[H_x]\big|\leq \Cr{constDerLMG}$ holds $\p$-a.s.\ with the same constant $\Cr{constDerLMG}$.  \eqref{eq:eta_nPert} now follows choosing $c_1:=2 \Cr{constDerLMG}$. 
\end{proof}	

\section{Large deviations and perturbation results for the PAM}
The main objective of this section is to establish certain exact large deviation results for hitting times of Brownian motion under the tilted measures introduced above, and then to apply these in order to obtain perturbation results.
For this purpose let
\begin{align}\label{eq:sigmav}
\begin{split}
    V_x^{\zeta,v}(\eta) &:= \frac{\eta}{v} - L_x^\zeta(\eta),\quad x\in\mathbb{R},\\
    \sigma^2_v &:= \text{Var}_\p(V_1^{\zeta,v}(\overline{\eta}(v)) + 2\sum_{i=2}^\infty \text{Cov}_\p\big (V_1^{\zeta,v}(\overline{\eta}(v)),V_i^{\zeta,v}(\overline{\eta}(v) \big),\quad \sigma_v := \sqrt{\sigma^2_v}, \quad v\in V. 
    \end{split}
\end{align}
We start with observing that $\sigma_v^2\in[0,\infty)$ for all $v\in V$. Indeed, $(\widetilde{L}_i)_{i\in\N}$, where $\widetilde{L}_i:=L_i^{\zeta}(\overline{\eta}(v))-\E[L_i^{\zeta}(\overline{\eta}(v))]$, is  a sequence of bounded (see Lemma~\ref{le:log mom fct and derivatives}), centered and mixing (see Lemma~\ref{le:expMixLi}) random variables, giving 
\begin{align}
\sum_{i=1}^\infty \big|\text{Cov}_\p\big (V_1^{\zeta,v}(\overline{\eta}(v)),V_i^{\zeta,v}(\overline{\eta}(v)) \big)\big| &=\sum_{i=1}^\infty \big|\E\big [\widetilde{L}_1\widetilde{L}_i \big]\big| = \sum_{i=1}^\infty \big|\E\big [\widetilde{L}_i\E[\widetilde{L}_1\given\F^{i-1}] \big]\big| <\infty, \label{eq:cov_bounded}
\end{align}
where the last inequality is due to uniform boundedness of $\widetilde{L}_i$ in $i$, \eqref{eq:expMix1} and the summability criterion in \eqref{eq:MIX}. Thus $\sigma_v^2<\infty$. Furthermore,  $\sigma_v^2\geq0$ is due to \eqref{eq:cov_bounded} and \cite[Lemma~1.1]{rio}.

We now introduce the process $W^v_x(t)$ of empirical Legendre transformations
\begin{equation}
    W^v_x(t) := {t\sqrt{x}}\Big( ( \overline{L}_{xt}^\zeta )^*(1/v) - L^*(1/v) \Big),\quad t,x>0,\ v\in V, \label{eq:defW}
\end{equation}
and set $W_0^v(t)=W_x^v(0)=0$ for $t,x>0$, $v\in V$, and obtain a first functional Central limit theorem for it. 
\begin{proposition}
\label{prop:InvarLegTrafoEmpLogMG}
For every $v \in V$, $W_n^v(1)$ converges in $\p$-distribution to a centered Gaussian random variable with variance $\sigma_v^2\geq0$. If $\sigma_v^2>0$, the sequence of processes 
\begin{equation*}
    [0,\infty) \ni t\mapsto \frac{1}{\sigma_v}W_n^v(t),\quad n\in\mathbb{N}, 
\end{equation*}
converges in $\p$-distribution to a standard Brownian motion  in the sense of weak convergence of measures on $C([0,\infty))$ with topology induced by the metric $\rho$ from \eqref{eq:def_metric}.
\end{proposition}
\begin{proof}
    It is sufficient to show the claim if 
    $(W^v_n(t))_{t\in[0,\infty)}$ is replaced by $(W_n^v(t) \cdot \mathds{1}_{\mathcal{H}_{nt}})_{t \in [0,\infty)},$ $n \in \N,$ with $\mathcal H_{nt}$ as defined in \eqref{eq:H_n}, since the $\p$-probability of $\mathcal H_{nt}$ tends to $1$ for $n\to\infty$ by Lemma~\ref{le:concEtaEmpLT}. In the notation of \eqref{eq:sigmav}, setting
    \begin{equation}
        \label{eq:defSn}
        S_x^{\zeta,v}(\eta):=\sum_{i=1}^x V_i^{\zeta,v}(\eta),\ x\in\R,
    \end{equation}
    on $\mathcal{H}_{nt}$ we have
    \begin{align*}
        \big( \overline{L}_{nt}^\zeta \big)^*\Big(\frac{1}{v}\Big) &=\frac{\eta_{nt}^\zeta(v)}{v} - \overline{L}^\zeta_{nt}(\eta_{nt}^\zeta(v)) = \frac{1}{{nt}} \sum_{i=1}^{nt} V_i^{\zeta,v}(\eta_{nt}^\zeta(v)) = \frac{1}{{nt}} S_{nt}^{\zeta,v}(\eta_{nt}^\zeta(v)).
    \end{align*}
   Thus, we can rewrite the relevant term as a sum of three differences
    \begin{align}
        \label{eq:summands}
        \begin{split}
        {nt}\Big( ( \overline{L}_{nt}^\zeta )^*(1/v) - L^*(1/v) \Big) &=   \Big( S_{nt}^{\zeta,v}(\eta_{nt}^\zeta(v)) - S_{nt}^{\zeta,v}(\overline{\eta}(v)) \Big)
        + \Big( S_{nt}^{\zeta,v}(\overline{\eta}(v)) - \E\big[ S_{nt}^{\zeta,v}(\overline{\eta}(v)) \big]  \Big)\\
        &\quad + \Big( \E\big[ S_{nt}^{\zeta,v}(\overline{\eta}(v)) \big]  -  {nt} L^*(1/v) \Big),
        \end{split}
    \end{align}
    where we note that the third summand vanishes. Indeed, we have
    \begin{align*}
        \E\big[ S_{nt}^{\zeta,v}(\overline{\eta}(v)) \big] &= nt\Big(\frac{\overline{\eta}(v)}{v} - \E[L_{1}^\zeta(\overline{\eta}(v))] \Big) 
         =ntL^*(1/v),
    \end{align*}
    where the last equality is due to \eqref{eq:legTraExpectedLogMGfct_2} and the definition of the Legendre transform. The proof is completed by the use of Lemmas~\ref{le:S_xLogbounded} and \ref{le:CLT_S_n} below, which show that the second summand of \eqref{eq:summands} exhibits the postulated diffusive behavior whereas the latter summand is negligible in that scaling.
    \end{proof}
    \begin{lemma}
        \label{le:CLT_S_n}
        For every $v\in V$,  $\frac{1}{\sqrt{n}}\Big(S_{nt}^{\zeta,v}(\overline{\eta}(v)) - \E\big[ S_{nt}^{\zeta,v}(\overline{\eta}(v)) \big] \Big)$ converges in $\p$-distribution to a centered Gaussian random variable with variance $\sigma_v^2\geq0$. If $\sigma_v^2>0$, the sequence of processes  
        \begin{equation*}
        [0,\infty)\ni t\mapsto \frac{1}{\sigma_v\sqrt{n}}\Big(S_{nt}^{\zeta,v}(\overline{\eta}(v)) - \E\big[ S_{nt}^{\zeta,v}(\overline{\eta}(v)) \big] \Big),\quad n\in\mathbb{N},
        \end{equation*}
        converges in $\p$-distribution to a standard Brownian motion in the sense of weak convergence of measures on $C([0,\infty))$ with topology induced by the metric $\rho$ from \eqref{eq:def_metric}. 
    \end{lemma}
	\begin{proof} 
		Let $\widetilde{L}_i:=L_i^{\zeta}(\overline{\eta}(v))-\E[L_i^{\zeta}(\overline{\eta}(v))]$, $\widetilde{V}_i:=V_i^{\zeta}(\overline{\eta}(v))-\E[V_i^{\zeta}(\overline{\eta}(v))]$ and $M\in\N$. Further set $\widetilde{L}_i^{(M)}:=\sum_{j=1+(i-1)M}^{iM}\widetilde{L}_j$. Then $(\widetilde{L}_i^{(M)})_{i\in\Z}$ is a sequence of centered, stationary and (by Lemma~\ref{le:log mom fct and derivatives}) bounded random variables. To show the central limit theorem on $C([0,M])$, we will use the method of martingale approximation from \cite{hall_heyde}, which is summarized as a theorem by Nolen in \cite[section~2.3]{nolen_random_medium_FKPP} and turns out to be applicable in our situation. That is, we have to make sure that condition \cite[(2.36)]{nolen_random_medium_FKPP} is fulfilled. Indeed, replacing in \eqref{eq:L_0_mixing_proof} $\F^j$ by $\F_k$ and noting that quantity $A$ in \eqref{eq:L_0_mixing_proof} is $\F_k$-measurable we get 
		\begin{align*}
		\sum_{k=1}^{\infty}|\widetilde{L}_0^{(M)}-\E[\widetilde{L}_0^{(M)}\given \F_k]| &\leq c_1\sum_{k=1}^\infty e^{-k/c_1}<\infty,
		\end{align*}
		giving the convergence of the first series in \cite[(2.36)]{nolen_random_medium_FKPP}. Furthermore, using that $\widetilde{L}_k^{(M)}$ is $\F^{(k-1)M}$-measurable and bounded, also recalling \eqref{eq:MIX}, we get
		\begin{align*}
		\sum_{k=1}^\infty |\E[\widetilde{L}_k^{(M)}\given \F_0]|\leq \sum_{k=1}^\infty \psi(k-1)\E[|\widetilde{L}_0^{(M)}|]<\infty.
		\end{align*} 
		Because the series in \eqref{eq:sigmav} is absolutely convergent, by \cite[Lemma~1.1]{rio} and \cite[(2.37)]{nolen_random_medium_FKPP} we have $\lim_{n\to\infty}\frac1n \E\big[\big(\sum_{k=1}^n\widetilde{L}_k^{(M)}\big)^2\big]=M\cdot \lim_{n\to\infty}\frac1{Mn} \E\big[\big(\sum_{k=1}^{Mn}\widetilde{L}_k\big)^2\big]=M\cdot \sigma_v^2\in[0,\infty)$. Furthermore, if $\sigma_v^2>0$, \cite[Theorem~2.1]{nolen_random_medium_FKPP} entail that the sequence of processes 
		\[ [0,1]\ni t\mapsto X_n^{(M)}(t):=\frac{1}{\sigma_v\sqrt{nM}}\big(\sum_{k=1}^{\lfloor nt\rfloor}\widetilde{L}_{k}^{(M)} + (nt-\lfloor nt\rfloor)\widetilde{L}_{\lfloor nt\rfloor+1}^{(M)} \big),\quad n\in\N, \] 
		converge in $\p$-distribution to a standard Brownian motion $(B_t)_{t\in[0,1]}$ in the sense of weak convergence of measures on $C([0,1])$ with the topology induced by the uniform metric.  
		Then by definition, above  convergence also holds true for $(\widetilde{V}_i)_{i\geq1}$ instead of $(\widetilde{L}_i)_{i\geq1}$. Furthermore, we have the uniform bound  
		\[\sup_{t\in[0,M],n\in\N} \Big| S_{nt}^{\zeta,v}(\overline{\eta}(v)) -\big(\sum_{i=1}^{\lfloor nt/M\rfloor M} V_i^{\zeta,v}+ (nt-\lfloor  nt/M\rfloor M) \sum_{i=1+\lfloor  nt/M\rfloor M}^{M+\lfloor  nt/M\rfloor M}V^{\zeta,v}_{i}\big) \Big| \leq c_2\quad \p\text{-a.s.}\]
		Consequently, the sequence $[0,M]\ni t\mapsto \frac{1}{\sigma_v\sqrt{n}}\Big(S_{nt}^{\zeta,v}(\overline{\eta}(v)) - \E\big[ S_{nt}^{\zeta,v}(\overline{\eta}(v)) \big] \Big)$ has the same weak limit as $\big(\sqrt{M}\cdot X_n^{(M)}(t/M)\big)_{t\in[0,M]}$, $n\in\N$,  which converges to $\big(\sqrt{M}\cdot B(t/M)\big)_{t\in[0,M]}$ and the latter process is a standard Brownian motion on $[0,M]$. Because $M\in\N$ was arbitrary, \cite[first Theorem~2.4 on page 15]{whitt69}  gives weak convergence on $C([0,\infty))$. 
\end{proof}
     To show that the first summand in \eqref{eq:summands} is asymptotically negligible, we use the following result. 
\begin{lemma}
	\label{le:S_xLogbounded}
	There exists a constant $\Cl{const_concSn}\in(0,\infty)$ such that  for every $v\in V$ and $M>0$, 
	\begin{equation*}
	\limsup_{n\to\infty}\frac{1}{\ln n}\sup_{0\leq t\leq M}\Big| S_{nt}^{\zeta,v}(\eta_{nt}^\zeta(v)) - S_{nt}^{\zeta,v}(\overline{\eta}(v)) \Big| \leq \Cr{const_concSn}\quad \p\text{-a.s.}
	\end{equation*}
\end{lemma}     
  \begin{proof}
	There exists a  $\p$-a.s.\  finite time $\Cr{constHn}(\omega)$, defined before \eqref{eq:H_n}, such that  for all $x\geq \Cr{constHn}$ and all $v\in V$ we have $\eta_x^\zeta(v)\in \triangle$. Furthermore, by Lemma~\ref{le:log mom fct and derivatives}, 
	$S_{x}^{\zeta,v}$ is infinitely differentiable on $(-\infty,0)$, so for all $x\geq \Cr{constHn}$ there exists $\widetilde{\eta}_x^\zeta(v) \in [\overline{\eta}(v) \wedge \eta_{x}^\zeta(v), \overline{\eta}(v) \vee \eta_{x}^\zeta(v)]$ such that  
	\begin{align*}
	S_{x}^{\zeta,v}(\overline{\eta}(v)) &= S_{x}^{\zeta,v}(\eta_{x}^\zeta(v)) + \big(S_{x}^{\zeta,v}\big)'(\eta_{x}^\zeta(v)) \big( \overline{\eta}(v) - \eta_{x}^\zeta(v)\big) + \frac{\big(S_{x}^{\zeta,v}\big)''(\widetilde{\eta}_x^\zeta(v))}{2} \big( \overline{\eta}(v) - \eta_{x}^\zeta(v) \big)^2.
	\end{align*}
	Due to \eqref{eq:EmpLogMGDerivTilt},  $(S_x^{\zeta,v})'(\eta_x^{\zeta}(v))=0$ and by  Lemma~\ref{le:log mom fct and derivatives} we have $\sup_{\eta\in\triangle}\sup_{x\geq1}\big|(S_x^{\zeta,v})''(\eta)\big|/x\leq c_1$. 
	By \eqref{eq:concEtaEmpLT_cont} and the first Borel-Cantelli lemma, there exists a finite random variable $\mathcal{N}_2\geq\Cr{constHn}$ such that for $x\geq\mathcal{N}_2$ the complementary event on the left-hand side of \eqref{eq:concEtaEmpLT_cont} occurs, hence
	\[  \sup_{x\geq\Cr{constHn}} \big( \eta_{x}^\zeta(v) - \overline{\eta}(v) \big)^2 \frac{x}{\ln x} \leq \Cr{const:concEtaEmpLT}^2(V,2) \]
	and thus 
	\begin{align}
	\sup_{x\geq\Cr{constHn}} \frac{1}{\ln x}\big| S_{x}^{\zeta,v}(\eta_{x}^\zeta(v)) - S_{x}^{\zeta,v}(\overline{\eta}(v))\big|   &\leq 
	\Cr{const_concSn}
	\label{eq:conc_S_x}
	\end{align}
	with $\Cr{const_concSn}:=\frac{c_1\Cr{const:concEtaEmpLT}^2(V,2)}{2}$. Finally, we have
	\begin{align*}
	\sup_{0\leq  t\leq M}\Big| S_{nt}^{\zeta,v}(\eta_{nt}^\zeta(v)) - S_{nt}^{\zeta,v}(\overline{\eta}(v)) \Big| &\leq  \sup_{0\leq x\leq \Cr{constHn}}  \Big| S_{x}^{\zeta,v}(\eta_{x}^\zeta(v)) - S_{x}^{\zeta,v}(\overline{\eta}(v)) \Big| \\
	&\qquad +\sup_{\Cr{constHn}/n\leq t\leq M}  \Big| S_{nt}^{\zeta,v}(\eta_{nt}^\zeta(v)) - S_{nt}^{\zeta,v}(\overline{\eta}(v))\Big|\\
	&\leq 2\Cr{constHn}c_2 + \Cr{const_concSn}\ln M+  \Cr{const_concSn}\ln n,
	\end{align*}
	where in the last inequality we used that $\p$-a.s., every summand in the definition of $S_n^{\zeta,\eta}$ is uniformly bounded by $c_2$.  The $\p$-a.s.\  finiteness of $\Cr{constHn}$ gives the claim.
\end{proof}

As a by-product of the proof above we get an approximation result of $W_x^v$ be a centered stationary sequence. 
\begin{corollary}
		\label{cor:invarLegTrafo}
		For every $v\in V$ and all $t$ such that  $vt\geq\Cr{constHn}$ we have 		
		\begin{equation*}
		\Big| \sqrt{vt} W_{vt}^v(1) -  \sum_{i=1}^{vt}\big( L(\overline{\eta}(v)) - L_i^{\zeta}(\overline{\eta}(v))   \big)  \Big| \leq \Cr{const_concSn}\ln v + \Cr{const_concSn}\ln t.
		\end{equation*}
\end{corollary}
	\begin{proof}
		By the definition of $W^v_x(t)$ and $S_x^{\zeta,v}(\eta)$ from \eqref{eq:defW} and \eqref{eq:defSn}, as well as the definition in  \eqref{eq:expec_Leg_LMG} for the corresponding  Legendre transformations,  we have
		\begin{align*}
		 \sqrt{vt} W_{vt}^v(1) &= vt\Big( \big( \overline{L}_{vt}^\zeta \big)^*(1/v) - L^*(1/v) \Big) =vt\Big( \frac{\eta_{vt}^\zeta(v)}{v} - \overline{L}_{vt}\big(  \eta_{vt}^\zeta(v) \big) - \frac{\overline{\eta}(v)}{v} + L(\overline{\eta}(v)) \Big) \\
		&=S_{vt}^{\zeta,v}\big( \eta_{vt}^\zeta(v) \big) - \sum_{i=1}^{vt} \big( \frac{\overline{\eta}(v)}{v} - L_i^\zeta(\overline{\eta}(v)) \big) + \sum_{i=1}^{vt}\big( L(\overline{\eta}(v)) - L_i^{\zeta}(\overline{\eta}(v))   \big) \\
		&=S_{vt}^{\zeta,v}\big( \eta_{vt}^\zeta(v) \big)  - S_{vt}^{\zeta,v}\big( \overline{\eta}(v) \big)  + \sum_{i=1}^{vt}\big( L(\overline{\eta}(v)) - L_i^{\zeta}(\overline{\eta}(v))   \big). 
		\end{align*}
		Then we can conclude using \eqref{eq:conc_S_x}. 
	\end{proof}
    
\subsection{An exact large deviation result for auxiliary processes}

For $x\geq 0$ and $v>0$ we introduce
\begin{align*}
    Y_v^\approx(x) &:= E_x\Big[ e^{\int_0^{H_0}\zeta(B_s)\diff s}; H_0\in\Big[\frac{x}{v} - K,\frac{x}{v}\Big] \Big], \\
    Y_v^>(x) &:= E_x\Big[ e^{\int_0^{H_0}\zeta(B_s)\diff s}; H_0<\frac{x}{v} - K \Big], \quad \text{and}\\
    Y_v(x) &:= E_x\Big[ e^{\int_0^{H_0}\zeta(B_s)\diff s}; H_0\leq \frac{x}{v} \Big] = Y_v^\approx(x) + Y_v^>(x), 
\end{align*}
where $K>0$ is a constant, defined in \eqref{eq:defK} below. For $v\in V$ and $x\geq 1$ we define
\begin{align*}
    \sigma_x^\zeta(v) :=
    \begin{cases}
    \big|\eta_x^\zeta(v)\big|\sqrt{\text{Var}_x^{\zeta,\eta_x^\zeta(v)}(H_0)},& \text{ on }\mathcal{H}_x,\\
    \sup\limits_{\eta\in\triangle}\left|\eta\right|\sqrt{\text{Var}_x^{\zeta,\sup_{\eta\in\triangle}}(H_0)},& \text{ on }\mathcal{H}_x^c.
    \end{cases}
\end{align*}
 Furthermore, by Lemma~\ref{le:log mom fct and derivatives}, there exists some $C>1$ such that  $\text{Var}_x^{\zeta,\eta}(H_0)=x\big(\overline{L}_x^\zeta\big)''(\eta)\in[x \Cr{constDerLMG}^{-1},x \Cr{constDerLMG}]$ for all $x\geq1$. Thus, there is some constant $\Cl{constSigman}=\Cr{constSigman}(\triangle)>1$ such that 
\begin{equation}
    \label{eq:sigmanIneq}
    \Cr{constSigman}^{-1}\sqrt{x}\leq \sigma_x^\zeta(v) \leq \Cr{constSigman}\sqrt{x}\quad\text{ for all }v\in V,\ x\geq 1,\ \text{on }\mathcal{H}_x.
\end{equation}
We now prove the following result.
\begin{proposition}
    \label{prop:ExactLD}
    Let  $V$ be as in \eqref{eq:defVTriangle}, $\sigma_v$ defined by \eqref{eq:sigmav} and $W^v_x(t)$ as in \eqref{eq:defW}, $v\in V$, and $K>0$ be such that \eqref{eq:defK} holds. Then there exists a constant $\Cl{constY_v}\in(1,\infty)$, 
		such that for all $v\in V$ and all $x\geq 1$, on $\mathcal{H}_x$ we have
		\begin{equation}
		\label{eq: prop: Y_v^bla vergleichbar}
		\sigma^\zeta_x(v) Y_v^>(x)\exp\left\{  xL^*(1/v) + \sqrt{x}W^v_x(1) \right\} \in \left[ \Cr{constY_v}^{-1},\Cr{constY_v} \right].
		\end{equation}
	  	Furthermore,  for all $v\in V$ and all $x\geq 1$, on $\mathcal{H}_{x}$ we have 
		\begin{equation}
		\label{eq:Y_v^approxStab}
		\frac{Y_{v}^\approx(x)}{Y_v^>(x)}\in [\Cr{constY_v}^{-1},\Cr{constY_v}], 
		\end{equation}
		and the sequence $n^{-1/2}\big( \ln Y^\approx_v(n)-nL^*(1/v) \big)$, $n\in\N$, converges to a centered Gaussian random variable with variance $\sigma_v^2\in[0,\infty)$, where $\sigma_v^2$ is defined in \eqref{eq:sigmav} and if $\sigma_v^2>0$,  the sequence  of processes
		\begin{align}
		    [0,\infty)\ni t&\mapsto \frac{1}{\sigma_v\sqrt{n}} \big(\ln Y_v^\approx( nt) +  nt L^*(1/v)\big),\quad n\in\N, \label{eq:InvY_v_approx}   
		\end{align}
		converges in $\p$-distribution to standard Brownian motion in $C([0,\infty))$.
	\end{proposition}
		
	\begin{proof}
	    We start with proving \eqref{eq: prop: Y_v^bla vergleichbar} and for this purpose let $x\geq 1$ such that $\mathcal{H}_{\lfloor x\rfloor}$ occurs.  Write $\eta:=\eta_{x}^\zeta(v)$ and $\sigma:=\sigma_x^\zeta(v),$ and recall the notation introduced in \eqref{eq:hit}, i.e. $\tau_i=H_i-H_{i-1}$, $i=1,\ldots,\lceil x\rceil -1$, and set $\tau_x:= H_{\lceil x\rceil - 1} - H_x,$ $x \in \R \backslash \Z$, (which would be consistent with the definition in \eqref{eq:hit} for $x$ integer) to define $\widehat{\tau}_{i}:={\widehat{\tau}}^{(x)}_{i}:=\tau_{i} - E_x^{\zeta,\eta}\left[ \tau_{i}\right]$. Then $\sum_{i=1}^x E_x^{\zeta,\eta}\left[ \tau_{i} \right]=E_x^{\zeta,\eta}\left[ H_{0} \right]=\frac{x}{v}$. We now rewrite
		\begin{align}
		Y_{v}^\approx(x)&=E_x\Big[ e^{ \int_{0}^{H_{0}}(\zeta(B_s)+\eta)\diff s } \exp\Big\{ -\eta\sum_{i=1}^{x}\widehat{\tau}_{i} \Big\}; \, \sum_{i=1}^{x}\widehat{\tau}_{i}\in\left[ -K,0 \right] \Big] \exp\left\{ -x\frac{\eta}{v} \right\}\notag\\
		&=E_x^{\zeta,\eta}\Big[  \exp\Big\{ -\sigma \frac{\eta}{\sigma}\sum_{i=1}^x\widehat{\tau}_{i} \Big\}; \,  \frac{\eta}{\sigma}\sum_{i=1}^x\widehat{\tau}_{i}\in\Big[ 0,-\frac{K\eta}{\sigma} \Big] \Big]   \exp\left\{ -x\left( \frac{\eta}{v} - \overline{L}_{x}^\zeta(\eta) \right) \right\}. \label{Y_vapprox herleitung}
		\end{align}
		Analogously, we get
		\begin{align*}
		Y_v^>(x)&=E_x^{\zeta,\eta} \Big[ \exp\Big\{ -\sigma \frac{\eta}{\sigma}\sum_{i=1}^x\widehat{\tau}_{i}\Big\}; \,  \frac{\eta}{\sigma}\sum_{i=1}^x\widehat{\tau}_{i} >-\frac{K\eta}{\sigma}  \Big] \exp\left\{ -x\left( \frac{\eta}{v} - \overline{L}_{x}^\zeta(\eta) \right) \right\}.
		\end{align*}
		We define $\mu_x^{\zeta,v}$ as the distribution of $\frac{\eta}{\sigma}\sum_{i=1}^x\widehat{\tau}_i$ under $P_x^{\zeta,\eta}$. Then 
		\begin{align}
		\label{eq:lctl_mu_n1}
		Y_v^\approx(x)&=e^{-x( \frac{\eta}{v} - \overline{L}_{x}^\zeta(\eta) )}\int_0^{\frac{-K\eta}{\sigma}} e^{-\sigma y}\, \diff\mu_x^{\zeta,v}(y)
		\end{align}
		and  
		\begin{align}
		\label{eq:lctl_mu_n2}
		Y_v^>(x) &=e^{-x ( \frac{\eta}{v} - \overline{L}_{x}^\zeta(\eta) )}\int_{\frac{-K\eta}{\sigma}}^{\infty} e^{-\sigma y}\, \diff\mu_x^{\zeta,v}(y).
		\end{align}
		Using Lemma~\ref{le:LCLTHitting} below, the integrals on the right-hand side of \eqref{eq:lctl_mu_n1} and \eqref{eq:lctl_mu_n2},  multiplied by $\sigma$, are bounded from below and above by positive constants. Display \eqref{eq: prop: Y_v^bla vergleichbar} now follows by the definition of $W_x^v,$ and \eqref{eq:Y_v^approxStab} is direct a consequence of \eqref{eq:lctl_mu_n1}-\eqref{eq: LCLTmun2}. The last two statements, are a consequence of \eqref{eq: prop: Y_v^bla vergleichbar}, \eqref{eq:Y_v^approxStab}, $W_{nt}^v(1)=\frac{1}{\sqrt{t}} W_n^v(t)$ and Proposition~\ref{prop:InvarLegTrafoEmpLogMG}.
	\end{proof}
	To complete the previous proof, it remains to prove the following.
	\begin{lemma}\label{le:LCLTHitting}
		Under the conditions of Proposition~\ref{prop:ExactLD},
        there exists a constant $\Cl{constCLTHitting}>1$ such that  for all $v\in V$ and $x\geq 1$, on $\mathcal{H}_x,$
		\begin{equation}
		\label{eq: LCLTmun1}
		\sigma_x^\zeta(v) \int_0^{{-K\eta_x^\zeta(v)}/{\sigma_x^\zeta(v)}} e^{-\sigma_x^\zeta(v) y} \diff\mu_x^{\zeta,v}(y) \in [\Cr{constCLTHitting}^{-1},\Cr{constCLTHitting}]
		\end{equation}
		and 
		\begin{equation}
		\label{eq: LCLTmun2}
		\sigma_x^\zeta(v) \int_{{-K\eta_x^\zeta(v)}/{\sigma_x^\zeta(v)}}^{\infty} e^{-\sigma_x^\zeta(v) y} \diff\mu_x^{\zeta,v}(y) \in [\Cr{constCLTHitting}^{-1},\Cr{constCLTHitting}],
		\end{equation}
		with $\mu_x^{\zeta,v}$ as in the proof of Proposition~\ref{prop:ExactLD}.
	\end{lemma}

	\begin{proof}
		We write $n:=\lceil x\rceil$ and recall that under $P_x^{\zeta,\eta}$, the sequence $\Big(\sqrt{n}\frac{\eta_x^\zeta(v)}{\sigma_x^\zeta(v)}\widehat{\tau}_i\Big)_{i=1,\ldots,\lfloor x\rfloor,x}$ is a sequence of independent,  centered random variables. Thus, on $\mathcal{H}_x$ we obtain
		\begin{align*}
		\frac{1}{n} \sum_{i=1}^x\Var_x^{\zeta,\eta}\bigg(\sqrt{n}\frac{\eta_x^\zeta(v)}{\sigma_x^\zeta(v)}\widehat{\tau}_i\bigg) &= \bigg(\frac{\eta_x^\zeta(v)}{\sigma_x^\zeta(v)}\bigg)^2\text{Var}_x^{\zeta,\eta}\bigg(  \sum_{i=1}^x\widehat{\tau}_i \bigg) = \bigg(\frac{\eta_x^\zeta(v)}{\sigma_x^\zeta(v)}\bigg)^2 \text{Var}_x^{\zeta,\eta}(H_0) =1.
		\end{align*}
		Additionally, the $\widehat{\tau}_i$'s have uniform exponential moments. Thus, the conditions of \cite[Theorem 13.3]{rao_normal_approx} are fulfilled and an application of \cite[(13.43)]{rao_normal_approx} yields
		\begin{equation*}
		\sup_{\mathcal{C}}\big| \mu_{x}^{\zeta,v}(\mathcal{C})-\Phi(\mathcal{C}) \big| \leq c_1{n}^{-1/2},
		\end{equation*}
		where the supremum is taken over all Borel-measurable convex subsets of $\R$, $\Phi$ denotes the standard Gaussian measure on $\mathbb{R}$ and $c_1$ only depends on the uniform bound of the exponential moments of the $\widehat{\tau}_i$'s. Without loss of generality, we will assume $c_1>4$. Then, due to  \eqref{eq:sigmanIneq}, by  denoting $\mathcal{C}:=\big[0,-K\eta_{x}^\zeta(v)/{\sigma_x^\zeta(v)}\big]$ we can choose $K>0$ large enough, 
		so that
		\begin{equation}
		    \label{eq:defK}
		    \Phi(\mathcal{C})\geq2c_1^{-1}n^{-1/2}\quad\text{for all } n\in\N \text{ and }v\in V.
		\end{equation}
		We thus get
		\begin{align}
		c_1^{-1}n^{-1/2} &\leq \Phi\left(\mathcal{C} \right) - \big| \mu_x^{\zeta,v}(\mathcal{C})-\Phi(\mathcal{C}) \big| \leq\mu_x^{\zeta,v}\left( \mathcal{C} \right) 	 
		\leq \big| \mu_x^{\zeta,v}(\mathcal{C})-\Phi(\mathcal{C}) \big| + \Phi\left(\mathcal{C}\right) \leq c_2(K)\cdot  n^{-1/2} \label{eq:mu_nBounded}.
		\end{align}
		Because the integrand in \eqref{eq: LCLTmun1} is bounded away from $0$ and infinity on the respective interval of integration (uniformly in $n \in \N$), \eqref{eq: LCLTmun1} is a direct consequence of \eqref{eq:sigmanIneq} and \eqref{eq:mu_nBounded}. For \eqref{eq: LCLTmun2}, we split the integral into a sum:
		\begin{align*}
		&\int_{{-K\eta_x^\zeta(v)}/{\sigma_x^\zeta(v)}}^{\infty} e^{-\sigma_x^\zeta(v) y} \diff\mu_x^{\zeta,v}(y) = \sum_{k=1}^\infty \int_{-k K\eta_x^\zeta(v)/\sigma_x^\zeta(v)}^{-(k+1) K\eta_x^\zeta(v)/{\sigma_x^\zeta(v)}} e^{-\sigma_x^\zeta(v) y} \diff\mu_x^{\zeta,v}(y) \\
		&\leq \sum_{k=1}^\infty \mu_x^{\zeta,v}\left( -\frac{K\eta_x^\zeta(v)}{{\sigma_x^\zeta(v)}}\left[k,k+1  \right] \right)e^{ -k\cdot K|\eta_x^\zeta(v)| } \leq c_2n^{-1/2} \sum_{k=1}^\infty e^{-k\cdot K|\overeta|} \leq \Cr{constCLTHitting} n^{-1/2},
		\end{align*}
		where we recall the notation from \eqref{eq:triangleEta}.
		The lower bound in \eqref{eq: LCLTmun2} can be obtained by noting that
		\begin{align*}
		   \int_{{-K\eta_x^\zeta(v)}/{\sigma_x^\zeta(v)}}^{\infty} e^{-\sigma_x^\zeta(v) y} \diff\mu_x^{\zeta,v}(y) &\geq \int_{{-K\eta_x^\zeta(v)}/{\sigma_x^\zeta(v)}}^{{-2K\eta_x^\zeta(v)}/{\sigma_x^\zeta(v)}} e^{-\sigma_x^\zeta(v) y} \diff\mu_x^{\zeta,v}(y) \\
		   &\geq e^{2K\eta_x^\zeta(v)}\mu_x^{\zeta,v}\big( [{-K\eta_x^\zeta(v)}/{\sigma_x^\zeta(v)}, {-2K\eta_x^\zeta(v)}/{\sigma_x^\zeta(v)}]\big).
		\end{align*} 
		Analogously to \eqref{eq:mu_nBounded}, choosing $\Cr{constCLTHitting}$ large enough, the last expression is bounded from below by $\Cr{constCLTHitting}^{-1}n^{-1/2}$. Combining this with \eqref{eq:sigmanIneq}, we finally arrive at \eqref{eq: LCLTmun2}.
		\end{proof}
		
	We are now in the position to prove the following result.  
	\begin{lemma}
	\label{le:Y_vapprox}
	Under the conditions of Proposition~\ref{prop:ExactLD}, for all $\delta>0$  there exists a constant $\Cl{constY_vapprox}=\Cr{constY_vapprox}(\delta)\in(1,\infty)$ such that for all $v\in V$, $t>0$, on $\mathcal{H}_{vt}$ we have
	\begin{align}\label{eq:comparisons}
	\Cr{constY_vapprox}^{-1}Y_v^\approx(vt)& \leq E_{vt}\Big[ e^{\int_{0}^{t}\zeta(B_s)\diff s}; B_{t}\in [-\delta,0]\Big] 
	\leq E_{vt}\Big[ e^{\int_{0}^{t}\zeta(B_s)\diff s}; B_{t}\leq  0\Big] \leq \Cr{constY_vapprox} Y_v^\approx( vt). 
	\end{align}
\end{lemma}
\begin{proof}
	The second inequality is obvious.
	Since $\{B_{t}\leq 0\}\subset\{ H_0\leq t \}$ and $\zeta\leq0$, we get $E_{vt}\big[ e^{\int_0^{t}\zeta(B_s)\diff s}; B_{t}\leq 0 \big]\leq Y_v(vt)\leq (1+\Cr{constY_v})Y_v^\approx(vt)$ by Proposition~\ref{prop:ExactLD} 
	and thus the last inequality in \eqref{eq:comparisons} is obtained. Therefore, it remains to show the first inequality. For this purpose, define the function $p(s):=E_0\big[ e^{\int_0^s\zeta(B_r)\diff r};B_s\in[-\delta,0] \big]$ which is bounded from below by $c_1(K,\delta)>0$ for all $s\in[0,K]$. Using the strong Markov property at $H_0$, we finally get
	\begin{align*}
	Y_v^\approx(vt)&= E_{vt}\left[ e^{\int_0^{H_0}\zeta(B_r)\diff r}; H_0\in\left[ t-K,t \right] \right] \\
	&\leq c_1(K,\delta)^{-1} E_{vt}\left[ e^{\int_0^{H_0}\zeta(B_r)\diff r}  p(t-s)_{|s=H_0}; H_0\in\left[ t-K,t \right] \right] \\
	&\leq  c_1(K,\delta)^{-1} E_{vt}\left[ e^{\int_0^{H_0}\zeta(B_r)\diff r}p(t-s)_{|s=H_0}\right]=c_1(K,\delta)^{-1}E_{vt}\left[ e^{\int_0^{t}\zeta(B_r)\diff r};B_{t}\in[-\delta,0] \right].
	\end{align*}
	and the claim follows by choosing $\Cr{constY_vapprox}:=c_1(K,\delta)\vee (1+\Cr{constY_v})$.
\end{proof}
Plugging the relation $\xi(x)=\zeta(x)+\es$, $x\in\R$, into Lemma \ref{le:Y_vapprox} immediately supplies us with the following corollary.
\begin{corollary}
	\label{cor: Y_v approx exp part}
	Let $\Cr{constY_vapprox}$ be as in Lemma~\ref{le:Y_vapprox}. Then for all $v\in V$, $t>0$, on $\mathcal{H}_{vt}$ we have
	\begin{align*}
	\Cr{constY_vapprox}^{-1}e^{\emph{\es}\cdot t}Y_v^{\approx}(vt)  &\leq E_{vt}\left[ e^{ \int_0^{t} \xi(B_s)\diff s };B_{t}\leq 0 \right] \leq \Cr{constY_vapprox} e^{\emph{\es}\cdot t}Y_v^{\approx}(vt).
	\end{align*}
\end{corollary}
Using the Feynman-Kac formula \eqref{eq:feynman_kac} we also get the following result. Recall that $u^{u_0}$ denoted the solution to \eqref{eq:PAM} with initial condition $u_0\in\mathcal{I}_{\text{PAM}}$. 
\begin{corollary}
	\label{cor:comp_heavi}
	Let $\Cr{constY_vapprox}=\Cr{constY_vapprox}(\delta)$ be as in Lemma~\ref{le:Y_vapprox}. Then for all $v\in V$, $t>0$, on $\mathcal{H}_{vt}$ we have
	\begin{equation*}
	u^{\mathds{1}_{[-\delta,0]}}(t,vt) \leq u^{\mathds{1}_{(-\infty,0]}}(t,vt) \leq \Cr{constY_vapprox}^2\cdot u^{\mathds{1}_{[-\delta,0]}}(t,vt). 
	\end{equation*}
\end{corollary}	

The previous results are fundamental for proving perturbation statements in the next section, which themselves will allow us to analyze path probabilities of the branching process.

\subsection{Proof of Theorem~\ref{th:InvPAM}}
Now we are ready to prove our first main result.	

\begin{proof}[Proof of Theorem~\ref{th:InvPAM}]
	We first assume $\sigma_v^2>0$ and consider the case $u_0=\mathds{1}_{(-\infty,0]}$ to show the second part of the claim, i.e.\ that the sequence of processes 
	\begin{equation}
	\label{eq:log_PAM_weak_conv}
	[0,\infty)\ni t\mapsto \frac{1}{ \sqrt{nv\sigma_v^2}} \big(\ln u(nt,vnt) - nt \Lambda(v) \big),\quad n\in\N, 
	\end{equation}
	converges in $\p$-distribution to standard Brownian motion. Because $[0,\infty) \ni t\mapsto \ln u(t,vt)$ might be discontinuous \emph{only} in $0$, we have to make clear what we mean by above convergence.  In fact, we show the invariance principle for a sequence of auxiliary processes $(X_n^v(t))_{t\geq0}$, $n\in\N$, where for every $n\in\N$ and $t\geq\frac{1}{n}$ $X_n^v(t)$ is the same as  in \eqref{eq:log_PAM_weak_conv}, whereas for $t\in[0,1/n]$ the term $\ln u(nt,vnt)$ in \eqref{eq:log_PAM_weak_conv} is replaced by $(1-nt)\ln u(0,0) + nt\ln(1,v)$, making $(X_n^v(t))_{t\geq 0}$ continuous.  Because the difference of the processes in \eqref{eq:log_PAM_weak_conv} and $(X_n^v (t))_{t\geq0}$ converges uniformly  to zero as $n\to\infty$, convergence of the processes in \eqref{eq:log_PAM_weak_conv} to a standard Brownian motion is defined as the convergence of the processes  $(X_n^v (t))_{t\geq0}$, $n\in\N$, to a standard Brownian motion in $C([0,\infty))$ with topology induced by the metric $\rho$ from \eqref{eq:def_metric}.

	By Proposition~\ref{prop:ExactLD} and Corollary~\ref{cor: Y_v approx exp part}, on $\mathcal{H}_{nvt}$ (recall the notation from \eqref{eq:H_n}) we have 
	\begin{align}
	\label{eq:u_sandwich}
	\begin{split}
	-\ln \Cr{constY_v} &\leq \ln \sigma_{ nvt }^\zeta(v) + \ln Y_v^\approx( nvt) +  nvtL^*(1/v) + \sqrt{nvt} W^v_{ vnt}(1) \leq \ln \Cr{constY_v}, 	\quad \text{ and}\\
	-\ln \Cr{constY_vapprox}  &\leq \ln u(nt,vnt) -  {\es\cdot nt} - \ln Y_v^{\approx}( vnt)  \leq \ln \Cr{constY_vapprox}.
	\end{split}
	\end{align}

	Recall that a sequence of processes $t\mapsto A_n(t)$, $n\in\N$, converges in $\p$-distribution to standard Brownian motion if and only if for all $\sigma>0$ the sequence $t\mapsto \sigma^{-1} A_n(\sigma^2 t)$, $n\in\N$, converges in $\p$-distribution to a standard Brownian motion. Applying this to \eqref{eq:InvY_v_approx}, the sequence of processes
	\begin{equation}
	\label{eq:InvY_v_proof_Th1}
	 [0,\infty)\ni t\mapsto \frac{1}{\sqrt{nv\sigma_v^2}} \big(\ln Y_v^\approx( vnt) +  vnt L^*(1/v)\big),\quad n\in\N,
	\end{equation}
	converges in $\p$-distribution to a standard Brownian motion. Further, by the  second line in \eqref{eq:u_sandwich}, 
	\[ -\ln \Cr{constY_vapprox}  \leq \big(\ln u(nt,vnt) - nt(\es-vL^*(1/v)) \big) -  \big( \ln Y_v^\approx(vnt) + vntL^*(1/v) \big)  \leq \ln \Cr{constY_vapprox} \]
	holds. Consequently,  if we can prove that 
	\begin{equation}
	\label{eq:Lyapunov_Legendre}
	\Lambda(v)=\es -v L^*(1/v)\quad\forall v\in V.
	\end{equation}
	the claim follows from \eqref{eq:InvY_v_proof_Th1}.  To prove \eqref{eq:Lyapunov_Legendre}, we set $n=1$ in \eqref{eq:u_sandwich} and note that $\frac{W_{vt}^v(1)}{\sqrt{t}}\tend{t}{\infty}0$ $\p$-a.s.\ for all $v\in V$, because $W_n(1)$ converges in $\p$-distribution to a centered normally distributed random variable by Proposition~\ref{prop:InvarLegTrafoEmpLogMG}.  Using \eqref{eq:lyapunov_alt}, \eqref{eq:u_sandwich} and \eqref{eq:sigmanIneq}, we get \eqref{eq:Lyapunov_Legendre}. 
	
	It remains to show the claim for arbitrary $u_0\in\mathcal{I}_{\text{PAM}}$. Recall that there exist	
	$\delta'\in(0,1)$ and $C'>1$, such that $\delta' \mathds{1}_{[-\delta',0]}(x) \leq  u_0(x)\leq C' \mathds{1}_{(-\infty,0]}(x)$ for all  $x\in\R$. Therefore, using Corollary~\ref{cor:comp_heavi} we have 
	\begin{equation}
	\label{eq:PAM_sandwich}
	\delta'\Cr{constY_vapprox}^{-2} u^{\mathds{1}_{(-\infty,0]}}(t,vt) \leq u^{{u_0}}(t,vt) \leq C'u^{\mathds{1}_{(-\infty,0]}}(t,vt).
	\end{equation}
	where we used that the solution to \eqref{eq:PAM} is linear in its initial condition. 
	Thus, the convergence of \eqref{eq:log_PAM_weak_conv} for arbitrary initial condition $u_0\in\mathcal{I}_{\text{PAM}}$ follows from the one with initial condition $\mathds{1}_{(-\infty,0]}$. This gives the second part of Theorem~\ref{th:InvPAM}. 
	
	It remains to show that $(nv)^{-1/2}\big( \ln u(n,vn)-n\Lambda(v) \big)$ converges in $\p$-distribution to a Gaussian random variable. For $u_0=\mathds{1}_{(-\infty,0]},$ this is a direct consequence of 
	\eqref{eq:u_sandwich} for $t=1$, \eqref{eq:Lyapunov_Legendre} and the second part of Proposition~\ref{prop:ExactLD}. For general $u_0,$ the claim follows from \eqref{eq:PAM_sandwich}.
\end{proof}	

In view of Corollary~\ref{cor:comp_heavi}, Proposition~\ref{prop:lyapunov} and \eqref{eq:Lyapunov_Legendre}, the Lyapunov exponent $\Lambda$, defined in \eqref{eq:lyapunov}, determines the exponential decay (growth, resp.) for solutions to \eqref{eq:PAM} for \emph{arbitrary} initial conditions $u_0\in\mathcal{I}_{\text{PAM}}$ (and not only for those with compact support). 
\begin{corollary} \label{cor:lyapunov_alt}
	For all $v\geq0$ and all $u_0\in\mathcal{I}_{\text{PAM}}$ we have that $\p$-a.s.,
	\begin{equation}
		\label{eq:lyapunov_alt}
		\Lambda (v) = \lim_{t\to\infty} \frac{1}{t} \ln u^{u_0}(t,vt).
	\end{equation}
Furthermore, $\Lambda$ is linear on $[0,v_c]$ and strictly concave on $(v_c,\infty),$ and the convergence in \eqref{eq:lyapunov_alt} holds uniformly on any compact interval $K\subset[0,\infty)$. 
\end{corollary}
\begin{proof}
	Note that by Proposition~\ref{prop:lyapunov}, $\delta'\mathds{1}_{[-\delta',0]}\leq u_0\leq C'\mathds{1}_{ (-\infty,0] }$ and Corollary~\ref{cor:comp_heavi} we have that \eqref{eq:lyapunov_alt} holds for all $v\in V$ and all compact $V\subset(v_c,\infty),$ so \eqref{eq:lyapunov_alt} is true for all $v>v_c$. The strict concavity of $\Lambda$ on $(v_c,\infty)$ follows from the strict convexity of $L^*(1/v)$, which in turn follows from the strict convexity of $L$ and standard properties of the Legendre transformation. If $v_c=0$, the proof is complete due to $\lim_{t \to \infty} \frac{1}{t}\ln u^{\delta'\mathds{1}_{[-\delta',0]}}=\Lambda(0)=\es$ by Proposition~\ref{prop:lyapunov} and $u^{u_0}\leq e^{\es t}$ for all $u_0\in\mathcal{I}_{\text{PAM}}$. Thus, let us assume $v_c>0$ from now on.  First observe that  $L^*(1/v)$ tends to $L^*(1/v_c)=-L(0)$ as $v\downarrow v_c$. Indeed, due to Lemma~\ref{le:expectedLogMGfct} (d), for each $v>v_c$ there exists a unique $\overline{\eta}(v)\in(-\infty,0),$ characterized via $L'\big(\overline{\eta}(v)\big)=\frac{1}{v}$, such that  
	\( L^*(1/v)=\frac{\overline{\eta}(v)}{v}-L(\overline{\eta}(v)). \)
	Furthermore, $(v_c,\infty)\ni v \mapsto \overline{\eta}(v)$ is continuously differentiable and strictly decreasing, bounded from above by $0.$ In addition, $(-\infty,0)\ni \eta \mapsto L'(\eta)$ is smooth and  strictly monotone and tends to $L'(0-)$ as $\eta\uparrow 0.$ As a consequence, we get that $\eta(v)\uparrow 0$ as $v\downarrow v_c$ and thus $L^*(1/v) \to  L^*(1/v_c)$ as $v\downarrow v_c$. Therefore, we deduce that $\Lambda(v)=\es - vL^*(1/v)$ for all $v\in[v_c,\infty)$. 	
	 Furthermore, for all $u_0\in\mathcal{I}_{\text{PAM}}$, due to \eqref{eq:POT}
	 \begin{equation}
	 	\label{eq:Lambda_lin_bound}
	 	\begin{split}
	 		u^{u_0}(t,vt)&\leq C'e^{\es t} E_{vt}\big[ e^{\int_0^t \zeta(B_s)\diff s};B_t\leq 0 \big] \leq C'e^{\es t} E_{vt}\big[ e^{\int_0^{H_0} \zeta(B_s)\diff s};H_0\leq t \big] \\
	 		&\leq C'\exp\big\{ t\big( \es + v\overline{L}_{vt}(0) \big) \big\}.
	 	\end{split}
	 \end{equation}
	Using Lemma~\ref{le:expectedLogMGfct} (b), taking logarithms, and dividing by $t$, we get that the (due to Proposition~\ref{prop:lyapunov}) concave function $\Lambda$ is bounded from above by the linear function $[0,\infty)\ni v\mapsto \es -vL(0)=\es + vL^*(1/v_c)$  and coincides with this function at $v=0$ as well as $v=v_c,$ and hence on the whole interval $[0,v_c]$.  Using \eqref{eq:Lambda_lin_bound} again, we infer \eqref{eq:lyapunov_alt} for all $v \geq 0$ and $u_0\in\mathcal{I}_{\text{PAM}}$. 

	To show that the convergence is uniform on every compact interval $K\subset [0,\infty)$, for $\varepsilon>0$ arbitrary we consider $\varepsilon\Z:=\{ k\varepsilon:k\in \Z\},$ and for $y\in\R$ set  $\epsnet{y}:=\sup\{ x\in\varepsilon\Z:\, x\leq y \}$. Then  the convergence is uniform on $ K\cap \varepsilon\Z.$ 
	A fortiori, for $t$ large enough,  uniformly in $y\in K$,
	\[
	u(t,t\epsnet{y}) \geq e^{ t( \Lambda(\epsnet{y})-\varepsilon ) }. 
	\]
 Lemma~\ref{le:harnack} then entails that 
 	\begin{equation}
 		\label{eq:harnack}
 		\inf_{z \in [ t\epsnet{y}-1,t\epsnet{y}+1 ] }
 		u(t+1,z) \geq \frac{1}{\Cr{Harnack}}  u(t,t\epsnet{y}) \geq \frac{1}{\Cr{Harnack}} e^{t ( \Lambda(\epsnet{y})-\varepsilon )}. 
 	\end{equation}
	Furthermore, using $0\leq y-\epsnet{y}\leq \varepsilon$, we have 
	\begin{equation}
		\label{eq:Gaussian_est}
		\begin{split}
			P_{yt}(B_{\varepsilon t}\in[t \epsnet{y}-1,t\epsnet{y}+1 ]) &\geq \sqrt{\frac{2}{\pi \varepsilon t}} \inf_{x\in [t\epsnet{y}-1,t\epsnet{y}+1] }e^{-\frac{(x-yt)^2}{2\varepsilon t}} \\
			&\geq \sqrt{\frac{2}{\pi}}\cdot\exp\Big\{  - \frac{(\varepsilon t +1)^2}{2\varepsilon t} - \frac{\ln(\varepsilon t)}{2} \Big\}.
		\end{split}
	\end{equation}
	Using the Feynman-Kac formula in the equality, the Markov property at time $\varepsilon t,$ and \eqref{eq:POT} in the first inequality, we infer that  
	\begin{align}		\label{eq:Harnack_uniform}
	\begin{split}
		u(t+1+\varepsilon t&,yt) = E_{yt}\big[ e^{\int_0^{t+1+\varepsilon t}\xi(B_s)\diff s} u_0(B_{t+1+\varepsilon t}) \big] \\
		&\geq e^{\ei \varepsilon t} \cdot P_{yt}(B_{\varepsilon t}\in[t \epsnet{y}-1,t\epsnet{y}+1 ]) \cdot \inf_{z \in [ t\epsnet{y}-1,t\epsnet{y}+1 ] }u(t+1,z) \\
		&\geq c_1 \exp\Big\{ t \Big( \Lambda(\epsnet{y})+(\ei-3/2)\varepsilon  - \frac{1}{t} - \frac{1}{2\varepsilon t^2}   -   \frac{\ln(\varepsilon t)}{2t}  \Big) \Big\}, 
		\end{split}
	\end{align}
	where in the last inequality we used \eqref{eq:harnack} and \eqref{eq:Gaussian_est}. 
	Setting $\te:=t+1+\varepsilon t$ and $\y:=\frac{\te}{t}y$, we get  
	\begin{align*}
		\label{eq:Lyap_lower_bound}
		 \frac{1}{\te} \ln u(\te,\te y) - \Lambda(y)  &= \frac{t}{\te} \Big( \frac{1}{t} \ln u(\te,t\y) - \Lambda(\epsnet{\y}) \Big) + \frac{t}{\te} \Lambda(\epsnet{\y}) - \Lambda(y). 
	\end{align*}
Since 	\eqref{eq:Harnack_uniform} holds uniformly for all $y\in K$, we infer that
\begin{equation*}
	\inf_{y \in K} \Big( \frac{1}{\te}\ln u(\te,\te y) - \Lambda(y)\Big) \geq \frac{t}{\te}\Big(\frac{\ln c_1}{t} +(\ei-3/2)\varepsilon  - \frac{1}{t} - \frac{1}{2\varepsilon t^2}  -   \frac{\ln(\varepsilon t)}{2t}    \Big)  -  \sup_{y\in K} \Big| \frac{t}{\te} \Lambda(\epsnet{\y}) - \Lambda(y) \Big|.
\end{equation*}
Since $\Lambda$ concave and finite, it is uniformly continuous on compact intervals. As a consequence, since
$\varepsilon>0$ was chosen arbitrarily, 
we deduce the lower bound
\begin{equation}
	\label{eq:lower_bound_lyapunov}
	\liminf_{t\to\infty} \inf_{y \in K}\Big(\frac{1}{t}\ln u(t,ty) -  \Lambda(y)\Big) \geq 0.
\end{equation}
To derive the matching upper bound, we assume that the convergence does not hold uniformly on $K$. Then, due to \eqref{eq:lower_bound_lyapunov}, there exist $\alpha>0$ and  sequences $(t_n)_{n\in\N}\subset [0,\infty)$ and $(y_n)_{n\in\N}\subset K$ such that $t_n\to\infty$ and
	\begin{equation}
		\label{eq:Lambda_uniform_contradict}
		\frac{1}{t_n}\ln  u(t_n,t_ny_n) - \Lambda(y_n)  \geq \alpha\quad\forall n\in\N.
	\end{equation} 
	Retreating to a suitable subsequence, we can assume $y_{n}\tend{n}{\infty} y\in K.$ 
	For $n$ such that $|y_n-y|\leq\varepsilon$, we have, similarly to \eqref{eq:Gaussian_est}, 
	\begin{align*}
		P_{y(t_n+1+\varepsilon t_n)}\big( B_{\varepsilon t_n}\in[ t_ny_n-1,t_ny_n+1 ] \big) &\geq \sqrt{\frac{2}{\pi\varepsilon  t_n} } \inf_{x\in[t_ny_n -1,t_ny_n+1]} e^{-\frac{(x-y(t_n+1+\varepsilon t_n))^2}{2\varepsilon t_n}}  \\
		&\geq \sqrt{\frac{2}{\pi\varepsilon  t_n} } \exp\Big\{- \frac{(1+y)^2}{2\varepsilon t_n}(t_n\varepsilon +1 )^2 \Big\}.
	\end{align*}
	Therefore, taking advantage of \eqref{eq:harnack} again and using  an argument as in the derivation of \eqref{eq:Harnack_uniform}, we infer that
	\begin{align}
		u(t_{n}+1+\varepsilon t_{n}, (t_{n}+1+\varepsilon t_{n})y) &\geq c_1\cdot  u(t_{n},t_{n}y_{n}) \cdot  \sqrt{\frac{1}{\varepsilon  t_n} } \exp\Big\{\varepsilon\ei t_n- \frac{(1+y)^2}{2\varepsilon t_n}(t_n\varepsilon +1 )^2 \Big\} \label{eq:Harnack2}
	\end{align}
for all $n$ such that $|y_n-y|\leq\varepsilon$.  
Now recall that $y_n\tend{n}{\infty} y$, that $\Lambda$ is continuous, as well as $\frac{1}{t}\ln u(t,ty)\tend{t}{\infty}\Lambda(y).$
Therefore, first taking logarithms, dividing by $t_n$, and then taking $n \to \infty,$ the left-hand side in \eqref{eq:Harnack2}  converges to $(1+\varepsilon)\Lambda(y).$ Contrarily, by \eqref{eq:Lambda_uniform_contradict}, the limit of the right-hand side is bounded from below by $ \Lambda(y)+\alpha +\varepsilon\ei - \frac{(1+y)^2}{2}\varepsilon$. 
Choosing $\varepsilon>0$ small enough,  
this leads to a  contradiction. As a consequence, we deduce the uniform convergence on $K$. 
\end{proof}

\subsection{Time perturbation}
In the next step we prove perturbation results, i.e., time and space perturbation Lemmas~\ref{le:time_perturb} and \ref{le:space_perturb}. These statements will be useful when comparing the expected number of particles which are slightly slower or faster than the ones with given velocity. As usual,  $u=u^{u_0}$ denotes the solution to \eqref{eq:PAM} with initial condition $u_0\in\mathcal{I}_{\text{PAM}}$.

\begin{lemma}
	\label{le:time_perturb}
	\begin{enumerate}
		\item Let $u_0\in\mathcal{I}_{\text{PAM}}$ and let $\e: \, (0,\infty) \to (0,\infty)$ be a function such that  $\e(t)\to0$ and $t\e(t)\to\infty$ as $t\to\infty$. Then there exists $\Cl{const_time_pert_genau}=\Cr{const_time_pert_genau}((\e(t))_{t\geq0},u_0)$ such that  $\p$-a.s.,  for all $t$ large enough,
		\begin{equation}
		\label{eq:time_pert_genau}
		\sup_{(v,h)\in\mathcal{E}_t} \Big| \ln  \Big( \frac{u^{u_0}(t+h,vt)}{u^{u_0}(t,vt)} \Big) - h (\text{\emph{\texttt{es}}}-\overline{\eta}(v)) \Big| \leq \Cr{const_time_pert_genau} + \Cr{const_time_pert_genau} |h|\Big( \sqrt{\frac{\ln t}{t}} + \frac{|h|}{t} \Big), 
		\end{equation}
		where $\mathcal{E}_t:=\left\{ (v,h):\ v\in V,\ |h|\leq t\e(t),\frac{vt}{t+h}\in V \right\}$.
		\item For all $\e>0$ and $u_0\in\mathcal{I}_{\text{PAM}}$ there exists a constant $\Cl{const_time_perturbation}>1$ and a $\p$-a.s.\  finite random variable $\Cl[index_prob_cont_time]{T_index_timepert}$ such that for all $t\geq\Cr{T_index_timepert}$, uniformly in $0\leq h\leq t^{1-\e}$, $v\in V$ and $\frac{vt}{t+h}\in V,$
		\begin{equation}
		\label{eq:time_pert_ungenau}
		\Cr{const_time_perturbation}^{-1}e^{h/\Cr{const_time_perturbation}} u^{u_0}(t,vt)  \leq  u^{u_0}(t+h,vt)  \leq \Cr{const_time_perturbation}e^{\Cr{const_time_perturbation}h} u^{u_0}(t,vt).
		\end{equation}
	\end{enumerate}
\end{lemma}

\begin{proof}
$(a)$	Note that it suffices to show the claim for $u_0=\mathds{1}_{(-\infty,0]}$. Indeed, 
for all $u_0\in\mathcal{I}_{\text{PAM}}$ we have $\delta'\mathds{1}_{[-\delta',0]}\leq u_0\leq C'\mathds{1}_{(-\infty,0]}$. Using Corollary~\ref{cor:comp_heavi}, we infer that for all
$u_0\in\mathcal{I}_{\text{PAM}}$, all $v\in V,$ and all $t$ large enough
\begin{equation*}
	\delta'\Cr{constY_vapprox}^{-2} u^{\mathds{1}_{(-\infty,0]}}(t,vt) \leq u^{{u_0}}(t,vt) \leq C'u^{\mathds{1}_{(-\infty,0]}}(t,vt).
\end{equation*}
where $u^{u_0}$ denotes the solution to \eqref{eq:PAM} with initial
condition $u_0$. 

by the same argument as at the end of the proof of Theorem~\ref{th:InvPAM}, the solutions to \eqref{eq:PAM} for different initial conditions $u_0\in\mathcal{I}_{\text{PAM}}$ differ at most by a multiplicative constant.
	Let $t$ be large enough such that $\mathcal{H}_{vt}$ occurs for all $v\in V$, which is possible by \eqref{eq:H_n}.
	Letting $(v,h)\in \mathcal{E}_t$ and writing $v':=\frac{vt}{t+h}\in V,$
	we infer that
	\begin{align}
	\frac{u^{\mathds{1}_{(-\infty,0]}}(t+h,vt)}{u^{\mathds{1}_{(-\infty,0]}}(t,vt)} &= \frac{E_{vt}\left[ e^{\int_0^{t+h}\xi(B_s)\diff s}; B_{t+h}\leq 0 \right]}{E_{vt}\left[ e^{\int_0^{t}\xi(B_s)\diff s}; B_{t}\leq 0 \right]} =e^{\es\cdot h}\frac{E_{vt}\left[ e^{\int_0^{t+h}\zeta(B_s)\diff s}; B_{t+h}\leq 0 \right]}{E_{vt}\left[ e^{\int_0^{t}\zeta(B_s)\diff s}; B_{t}\leq 0 \right]}
	\label{eq:frac_time_pert}. 
	\end{align}
	Using Lemma~\ref{le:Y_vapprox}, on $\mathcal{H}_{vt},$ the last fraction divided by
	 $\frac{Y_{v'}^\approx(vt)}{Y_{v}^\approx(vt)}$ is bounded away from $0$ and infinity for all $t$ large enough.
As in the derivation of \eqref{Y_vapprox herleitung}, the term $\frac{Y_{v'}^\approx(vt)}{Y_{v}^\approx(vt)}$ can be written as
	\begin{align} \label{eq:prodApproxRel}
	&\frac{  E_{vt}^{\zeta,\eta_{vt}^\zeta(v')}\Big[  \exp\Big\{ -\eta_{vt}^\zeta(v')\sum_{i=1}^{vt}\widehat{\tau}_{i} \Big\}; \sum_{i=1}^{vt}\widehat{\tau}_{i}\in\Big[ -K,0 \Big] \Big]     }{  E_{vt}^{\zeta,\eta_{vt}^\zeta(v)}\Big[  \exp\Big\{ -\eta_{vt}^\zeta(v)\sum_{i=1}^{vt}\widetilde{\tau}_{i} \Big\}; \sum_{i=1}^{vt}\widetilde{\tau}_{i}\in\Big[ -K,0 \Big] \Big]     } \times \frac{\exp\Big\{ -vt\Big( \frac{\eta_{vt}^\zeta(v')}{v'} - \overline{L}_{vt}^\zeta(\eta_{vt}^\zeta(v')) \Big) \Big\}}{\exp\Big\{ -vt\Big( \frac{\eta_{vt}^\zeta(v)}{v} - \overline{L}_{vt}^\zeta(\eta_{vt}^\zeta(v)) \Big) \Big\}},
	\end{align}
	where $\widehat{\tau}_i=\tau_i - E_{vt}^{\zeta,\eta_{vt}^\zeta(v')}[\tau_i]$ and $\widetilde{\tau_i}:=\tau_i - E_{vt}^{\zeta,\eta_{vt}^\zeta(v)}[\tau_i]$. But now, since $v'\in V,$ as in the proof of Proposition~\ref{prop:ExactLD}, the first fraction of the previous display is bounded from below and above by positive constants, for all $t$ large enough.
	Indeed, setting $x=vt$ in \eqref{eq:lctl_mu_n1}, the denominator in \eqref{eq:prodApproxRel} equals the integral in  \eqref{eq:lctl_mu_n1} 
	and, replacing $v$ by $v'$ in \eqref{eq:lctl_mu_n1}, the  numerator equals the integral in \eqref{eq:lctl_mu_n1}.
	The claim then follows due to   Lemma~\ref{le:LCLTHitting} and using \eqref{eq:sigmanIneq}. Therefore, taking logarithms in \eqref{eq:frac_time_pert} and recalling the definition  of $S_{vt}^{\zeta,v}(\eta)$ in \eqref{eq:defSn}, according to the previous considerations it suffices to show that the logarithm of the second fraction in \eqref{eq:prodApproxRel} plus $\overline{\eta}(v)\cdot h$, i.e.
	\begin{align}
	&\big(S_{vt}^{\zeta,v}(\eta_{vt}^{\zeta}(v)) - S_{vt}^{\zeta,v}(\eta_{vt}^{\zeta}(v'))   \big)+ \big(  S_{vt}^{\zeta,v}(\eta_{vt}^{\zeta}(v')) - S_{vt}^{\zeta,v'}(\eta_{vt}^{\zeta}(v'))\big) + \overline{\eta}(v)\cdot h,\label{eq: time perturbation sums}
	\end{align}
satisfies  bound on the right-hand side of  \eqref{eq:time_pert_genau}, uniformly in $(v,h)\in\mathcal{E}_t$. 
	Recall that  $\frac{1}{v'}=\frac{1}{v}\left( 1+\frac{h}{t} \right)$, thus the second summand in \eqref{eq: time perturbation sums} is $-h\cdot \eta_{vt}^\zeta(v')$. The triangular inequality entails
	\begin{align}
	|\eta_{vt}^\zeta(v') - \overline{\eta}(v)| &\leq |\eta_{vt}^\zeta(v')-\overline{\eta}(v')|+|\overline{\eta}(v') - \overline{\eta}(v)| \label{eq: conv speed eta_n(v')-eta(v)},
	\end{align}
	and so by Lemma~\ref{le:concEtaEmpLT}, uniformly  for $v'\in V$ and $t$ large enough, the first term on the right-hand side of \eqref{eq: conv speed eta_n(v')-eta(v)} can be upper bounded by $\Cr{const:concEtaEmpLT}\sqrt{\frac{\ln vt}{vt}}$, $\p$-a.s.\ 
	Furthermore, by Lemma~\ref{le:expectedLogMGfct} d) we know that $\overline{\eta}$ is continuously differentiable and strictly decreasing, having  uniform positive bounds of the derivative on every bounded subinterval of $(v_c,\infty)$.
	Hence, $\overline{\eta}(\cdot)$ is Lipschitz continuous on $V$  and we therefore get that the second summand in \eqref{eq: conv speed eta_n(v')-eta(v)} can be upper bounded by $c_1|v-v'|=c_1v\frac{|h|}{|t+h|}=c_1v\frac{|h|}{t}\cdot\frac{t}{|t+h|}\leq c_2 \frac{|h|}{t}$, uniformly for all $v,v'\in V$ and all $t$ large enough, where the last inequality is due to $|h|/t\leq \varepsilon(t)\to 0$. Therefore, the absolute value of the sum of the second and third summand in \eqref{eq: time perturbation sums} 
	is upper bounded by $\Cr{const_time_pert_genau}|h|\big( \sqrt{\frac{\ln t}{t}} + \frac{|h|}{t} \big)$ with $\Cr{const_time_pert_genau}:=c_2\vee \Cr{const:concEtaEmpLT}$. 
	
	It remains to show that the first summand in \eqref{eq: time perturbation sums} tends to $0$ as $t$ tends to $\infty$. 
	We write
	\begin{align*}
	S_{vt}^{\zeta,v}(\eta_{vt}^{\zeta}(v')) &= S_{vt}^{\zeta,v}(\eta_{vt}^{\zeta}(v)) + \big( \eta_{vt}^{\zeta}(v') - \eta_{vt}^{\zeta}(v) \big) \big(S_{vt}^{\zeta,v}\big)'(\eta_{vt}^{\zeta}(v))  + \frac{1}{2}\big( \eta_{vt}^{\zeta}(v') - \eta_{vt}^{\zeta}(v) \big)^2 \big(S_{vt}^{\zeta,v}\big)''(\widetilde{\eta})
	\end{align*}
	for some $\widetilde{\eta} \in [\eta_{vt}^{\zeta}(v') \wedge \eta_{vt}^{\zeta}(v), \eta_{vt}^{\zeta}(v') \vee \eta_{vt}^{\zeta}(v)]$. Recall that $S_{vt}^{\zeta,v}(\eta)=vt\big(\frac{\eta}{v}-\overline{L}_{vt}^\zeta(\eta)\big)$ and by definition $\big(\overline{L}_{vt}^\zeta\big)'\big(\eta_{vt}^\zeta(v)\big) = \frac{1}{v}$. Thus,  
	\(	\big(S_{vt}^{\zeta,v}\big)'(\eta_{vt}^{\zeta}(v))=0.\)
	Furthermore, $\big(S_{vt}^{\zeta,v}\big)''(\eta) = -{vt}\big( \overline{L}_{vt}^\zeta \big)''(\eta)$ and the function $\big( \overline{L}_{vt}^\zeta \big)''$ is uniformly bounded away from $0$ and infinity by Lemma~\ref{le:log mom fct and derivatives} on $V.$ Thus, by the characterizing equation \eqref{eq:EmpLogMGDerivTilt} and the implicit function theorem, on $\mathcal{H}_{vt}$, the function $\eta_{vt}^\zeta(\cdot)$ is  differentiable with uniformly bounded first derivative, i.e.
	\begin{equation}
	\label{eq:conc_eta_vt}
	\big| \eta_{vt}^{\zeta}(v') - \eta_{vt}^{\zeta}(v)  \big| \leq c_3|v'-v| \leq c_3v\frac{|h|}{t(1-\e(t))}.
	\end{equation} 
	Thus, on $\mathcal{H}_{vt}$, the first summand in \eqref{eq: time perturbation sums} can be bounded by $c_4 \cdot t \cdot \frac{h^2}{t^2}= c_4\cdot \frac{h^2}{t},$ 
	uniformly in $(v,h)\in\mathcal{E}_t$, for all $t$ large enough.	This implies $a).$	

$(b)$	The first part of the  proof  is similar to that of $a)$; indeed, dividing by $u^{u_0}(t,vt)$ and taking logarithms in \eqref{eq:time_pert_ungenau}, one arrives at \eqref{eq: time perturbation sums} again. The second part then consists of showing that \eqref{eq: time perturbation sums} is  lower and upper bounded by two strictly increasing linear functions for all $0< h\leq t^{1-\varepsilon}$. Following the same computations as in the proof of $a),$ for $v\in V$ and $h>0$, such that $v'\in V$, we have up to some additive constant, which is independent of $h$, that 
	\begin{align*}
	\ln \frac{u^{u_0}(t+h,vt)}{u^{u_0}(t,vt)} &\asymp c_t\frac{vt}{2}(v-v')^2 -\eta_{vt}^\zeta(v')\cdot h + \es\cdot h,
	\end{align*} 
	where $c_t$ is a function which for $t$ large enough is positive and bounded away from $0$ and infinity. Because $\eta_{vt}^\zeta(v')<0$, the latter expression is bounded from below by $h/\Cr{const_time_perturbation}$ and bounded from above by $\Cr{const_time_perturbation}\cdot h$ for our choice of parameters, hence we can conclude. 
\end{proof}	

	\subsection{Space perturbation}
	While in the previous section we have been investigating the effects of time perturbations of $u$ and related quantities, here we will consider space perturbations.
	As before, let $u^{u_0}$ denotes the solution to \eqref{eq:PAM} with initial condition $u_0\in\mathcal{I}_{\text{PAM}}$. 
    
	\begin{lemma}
		\label{le:space_perturb}
		Let $\e(t)$ be a positive function such that $\e(t)\to0$ and
		$\frac{t\e(t)}{\ln t}\to\infty$  as $t\to\infty$. Then for all $\e>0$ there
		exists $C(\e)>0$ such that $\p$-a.s., for all $u_0\in\mathcal{I}_{\text{PAM}}$ we have
		\begin{enumerate}
			\item[(a)]
			\begin{equation}
			\label{eq:space_pert_genau}
			\hspace{-10mm}\limsup_{t\to\infty}\ \sup\left\{  \left|
			\frac{1}{h}\ln\left( \frac{u(t,vt+h)}{u(t,vt)} \right) - L\big(
			\overline{\eta}(v) \big)  \right|: (v,h)\in\mathcal{E}_t\right\} \leq \e,
			\end{equation}
			where $\mathcal{E}_t:=\left\{ (v,h):\ v,v+\frac{h}{t}\in V,\ C(\e)\ln
			t\leq |h|\leq t\e(t)  \right\}$.
			
			\item[(b)] Let $\e(t)$ be a positive function such that $\varepsilon(t)\to0$. Then there exists a constant $\Cl{const_space_perturbation}<\infty$
			and a $\p$-a.s.~finite random variable
			$\Cl[index_prob_cont_time]{T_index_spacepert}$ such that  for all
			$t\geq\Cr{T_index_spacepert}$, uniformly in $0\leq h\leq t\varepsilon(t)$, $v\in V$, $v+\frac{h}{t}\in V$ and $u_0\in\mathcal{I}_{\text{PAM}}$ we have
			\begin{equation}
			\label{eq:space_pert_ungenau}
			\Cr{const_space_perturbation}^{-1}e^{-\Cr{const_space_perturbation}h}\cdot
			u(t,vt) \leq u(t,vt+h) \leq
			\Cr{const_space_perturbation}e^{-h/\Cr{const_space_perturbation}}\cdot u(t,vt).
			\end{equation}
		\end{enumerate}
	\end{lemma}
The proof of Lemma~\ref{le:space_perturb} can be found in the companion article \cite{CeDrSc-20}; indeed, it is central to the results of \cite{CeDrSc-20} and we hence prefer not to duplicate it here due to space constraints. 

\subsection{Approximation results}
In this section we mainly show how moment generating functions can be used in order to approximate quantities related to the solution to \eqref{eq:PAM} and BBMRE. As usual, for $u_0\in\mathcal{I}_{\text{PAM}}$ let $u^{u_0}$ be the solution to \eqref{eq:PAM} with initial condition $u_0$. 
\begin{lemma}
	\label{le:ApproxPAMbyLMG}
	There exists a constant $\Cl{const approx to pam with speed v0}>0$ and a $\p$-a.s.\  finite random variable  $\Cl[index_prob_cont_time]{random index approx to pam spped v0}$ such that  for all $u_0\in\mathcal{I}_{\text{PAM}}$ and  $t\geq\Cr{random index approx to pam spped v0},$
	\begin{equation}
	\label{eq: cor approx to pam speed v0}
	\Big| \ln u^{u_0}(t,v_0t) - \sum_{i=1}^{v_0t}\big( L_i^\zeta(\overline{\eta}(v_0)) - L(\overline{\eta}(v_0)) \big) \Big| \leq \Cr{const approx to pam with speed v0} \ln t.
	\end{equation}
\end{lemma}
\begin{proof}
	By \eqref{eq:Lyapunov_Legendre} and $\Lambda(v_0)=0$, we have $L^*(1/v_0)=\frac{\es}{v_0}$. Also, by \eqref{eq:PAM_initial} and monotonicity of the solution to \eqref{eq:PAM} in its initial condition, we have $u^{\delta'\mathds{1}_{[-\delta',0]}} \leq u^{u_0}\leq u^{ C'\mathds{1}_{ (-\infty,0]  } }$.  Thus, on the one hand, by the many-to-few lemma (Proposition \ref{prop:many-to-few}) and Lemma~\ref{le:Y_vapprox}, for all $u_0\in\mathcal{I}_{\text{PAM}}$ and $t$ such that $v_0t\geq\Cr{constHn}$, where $\Cr{constHn}$ was defined in \eqref{eq:H_n},  we have 
	\begin{equation*}
	\left| \ln  u^{u_0}(t,v_0t) - \left( \ln Y_{v_0}^\approx(v_0t) + v_0tL^*(1/v_0) \right) \right| \leq \ln (\Cr{constY_vapprox}/\delta').
	\end{equation*}
	On the other hand, due to Proposition~\ref{prop:ExactLD} and Corollary~\ref{cor:invarLegTrafo},  there exists a finite random variable $\mathcal{N}$, such that for all $t\geq \mathcal{N}$ we have 
	\begin{equation*}
	\Big| \ln Y_{v_0}^\approx({v_0t}) + v_0tL^*(1/v_0) +  \ln\sigma_{v_0t}^\zeta(v_0) - \sum_{i=1}^{v_0t}\big( L_i^{\zeta}(\overline{\eta}(v_0)) - L(\overline{\eta}(v_0)) \big) \Big| \leq \ln \Cr{constY_v} + \Cr{const_concSn}\ln v_0 +  \Cr{const_concSn} \ln t.
	\end{equation*}
	Finally, by \eqref{eq:sigmanIneq}, $|\ln \sigma_{v_0t}^\zeta(v_0) - \frac{1}{2}\ln t|\leq \ln \Cr{constSigman} + \frac{1}{2}|\ln v_0|$ for all $t$ such that  $v_0 t\geq\Cr{constHn}$.
	Combining this with the previous two display, inequality \eqref{eq: cor approx to pam speed v0} follows with $\Cr{random index approx to pam spped v0}:=(\mathcal{N}\vee \Cr{constHn})/v_0$ and $\Cr{const approx to pam with speed v0}$ suitable. 
\end{proof}

We introduce the so-called \emph{breakpoint inverse}
\begin{equation}
\label{eq:defT_n1}
T_x^{u_0,a}:=\inf\big\{ t\geq0: u^{u_0}(t,x)\geq a  \big\},\quad x\in\R,\ a \in [0,\infty),\ u_0\in\mathcal{I}_{\text{PAM}},
\end{equation}
and abbreviate
\begin{equation}
\label{eq:defT_n}
T_x^{(a)}:=T_x^{\mathds{1}_{(-\infty,0]},a}.
\end{equation} 
Next, we state an important approximation result for $T_x^{u_0,a}$, $x\geq 0$, in terms of the centered logarithmic moment generating functions.
\begin{lemma}
	\label{le:ApproxT_n}
	There exists a constant $\Cl{const approx of Tn}<\infty$ and a $\p$-a.s.\  finite random variable $\Cl[constant]{const prob approx of Tn} = \Cr{const prob approx of Tn}(a,u_0)$, $a>0,u_0\in\mathcal{I}_{\text{PAM}}$, such that  for all $x\geq 1,$ 
	\begin{equation}
	\label{eq:ApproxT_n}
	\Big| T_x^{u_0,a} - \frac{1}{v_0 L(\overline{\eta}(v_0))}\sum_{i=1}^{x} L_i^{\zeta}(\overline{\eta}(v_0)) \Big| \leq \Cr{const prob approx of Tn} +  \Cr{const approx of Tn}\ln x.
	\end{equation}
	Additionally, for each $u_0\in \mathcal{I}_{\text{PAM}}$ and  $a>0,$
	\begin{equation}
	\label{eq:LimitTn/n}
	\lim_{x\to\infty} \frac{T_x^{u_0,a}}{x}=\frac{1}{v_0}\quad \p\text{-a.s}.
	\end{equation}
\end{lemma}
\begin{proof} We set $t=x/v_0$ and let 
	\[ h_t:=\frac{1}{v_0 L( \overline{\eta}(v_0) )}\sum_{i=1}^{v_0t}\big(  L( \overline{\eta}(v_0) - L_{i}^\zeta(\overline{\eta}(v_0))  \big). \]
	We first note that due to Lemma~\ref{le:expMixLi}, the family
	$ (L(\overline{\eta}(v_0))-L_i^\zeta(\overline{\eta}(v_0)))_{i\in\Z}$ satisfies the assumptions of Corollary~\ref{cor:rioDepExp} with all $m_i$ equal to some large enough finite constant and thus $\sum_n \p(|h_n|\geq C\sqrt{n\ln n})<\infty$ for some $C>1$ large enough. The first Borel-Cantelli lemma then readily supplies us with $|h_n|<C\sqrt{n\ln n}$ $\p$-a.s.\  for all $n$ large enough. To control non-integer $t$, we recall
	\begin{align*}
	h_{t}-h_{\lfloor t\rfloor} &= \frac{1}{v_0L(\overline{\eta}(v_0))} \big( v_0(t-\lfloor t\rfloor)L(\overline{\eta}(v_0)) - \ln E_{v_0 t}\big[ e^{\int_0^{H_{v_0\lfloor t\rfloor}}(\zeta(B_s)+\overline{\eta}(v_0))\diff s} \big] \big)
	\end{align*}
	and hence $\p$-a.s.\ that  $|h_t-h_{\lfloor t\rfloor}|\leq  1 + \frac{\sqrt{2(\es-\ei-\overline{\eta}(v_0))}}{|L(\overline{\eta}(v_0))|}$ by \eqref{eq:POT} and \cite[(2.0.1), p.~204]{handbook_brownian_motion}, thus giving 
	\begin{equation} \label{eq:htBd}
	|h_t|<c_1\sqrt{t\ln t}\ \p\text{-a.s.\  for all $t$ large enough.}
	\end{equation}
	
	To show the desired inequality, we note that \eqref{eq:ApproxT_n} is equivalent to 
	\begin{equation}
	\label{eq:ApproxT_n2}
	\big| T_{v_0t}^{u_0,a} -  (t - h_t )  \big| \leq \Cr{const prob approx of Tn} +  \Cr{const approx of Tn}\ln (v_0t).
	\end{equation}
	For proving the latter, observe that it is sufficient to show that we can choose   $\Cr{const approx of Tn}>0$ as well as a $\p$-a.s.\ finite random variable $\mathcal{T},$ such that 
	\begin{align}
	\label{eq:T_x_proof}
	\begin{split}
	u^{ \delta' \mathds{1}_{ [-\delta',0] } }(t-h_t+\Cr{const approx of Tn}\ln t,v_0t) &\geq a \quad \text{ and }\quad
	u^{ \mathds{1}_{ (-\infty,0] } }(t-h_t-\Cr{const approx of Tn}\ln t,v_0t) < \frac{a}{2C'},  \, \forall t\geq\mathcal{T}.
	\end{split}
	\end{align}
	with $\delta',C'$ from \eqref{eq:PAM_initial}. Indeed, due to \eqref{eq:PAM_initial}, the first inequality in \eqref{eq:T_x_proof} implies $T_{v_0t}^{u_0,a} \leq  T_{v_0t}^{ \delta' \mathds{1}_{ [-\delta',0] } ,a} \leq t-h_t+\Cr{const approx of Tn}\ln t$ for $t\geq\mathcal{T}$. To use the second inequality, first note that  $T_{v_0t}^{u_0,a}\geq T_{v_0t}^{ C'\mathds{1}_{(-\infty,0]}, a }= T_{v_0t}^{ \mathds{1}_{(-\infty,0]}, a/C' }$. Then,  using  $u^{\mathds{1}_{(-\infty,0]}}(s,x)=\Eprob_{x}^\xi\big[ N^\leq(s,0) \big]$ and Lemma~\ref{le:MomPAMIncrease}, \eqref{eq:T_x_proof} implies $T_{v_0t}^{ \mathds{1}_{(-\infty,0]}, a/C' }\geq  t-h_t-\Cr{const approx of Tn}\ln t$ for all $t\geq\mathcal{T}$. For $t<\mathcal{T}$ on the other hand, we can use that $\p$-a.s., the family $T_{v_0t}^{u_0,a},$ $t<\mathcal{T},$  as well as the $L_i^\zeta(\overline{\eta}(v_0))$ are uniformly bounded, allowing us to upper bound the remaining cases of \eqref{eq:ApproxT_n2} by some finite random variable $\Cr{const prob approx of Tn}$. 
	
	Thus, in order to show \eqref{eq:T_x_proof}, note that  for $\alpha\in\R$ and uniformly in $u_0\in\mathcal{I}_{\text{PAM}}$ we have that
	\begin{align*}
	&\ln u^{u_0}(t-h_t+\alpha\ln t,v_0t) \\ & = \ln \Big( \frac{u^{u_0}(t-h_t+\alpha\ln t,v_0t)}{u^{u_0}(t,v_0t)} \Big) +  \sum_{i=1}^{v_0t}\big( L_i^\zeta(\overline{\eta}(v_0)) - L(\overline{\eta}(v_0)) \big) + a_t \\
	&=(-h_t + \alpha\ln t)(\es-\overline{\eta}(v_0))  + v_0L(\overline{\eta}(v_0))h_t +   b(\alpha,t)
	=\alpha(\es-\overline{\eta}(v_0))\ln t + b(\alpha,t),
	\end{align*}
	for some error terms $a_t$ and $b(\alpha,t)$ fulfilling $|a_t|\leq \Cr{const approx to pam with speed v0}\ln t$ and
	\[ |b(\alpha,t)| \leq \Cr{const approx to pam with speed v0}\ln t + \Cr{const_time_pert_genau} +  \Cr{const_time_pert_genau} \cdot |\alpha\ln t -h_t|\cdot \Big( \sqrt{\frac{\ln t}{t}} + \frac{| \alpha \ln t -h_t |}{t} \Big) \]
	for all $t$ large enough. 
	Indeed, the first equality is due to Lemma~\ref{le:ApproxPAMbyLMG}, the second due to the time perturbation Lemma~\ref{le:time_perturb}, the last one due to the identity $\es-\overline{\eta}(v_0)= v_0L(\overline{\eta}(v_0)).$ 
	Then due to $|h_t|\leq c_1\sqrt{t\ln t}$ for large $t$ (cf.\ \eqref{eq:htBd}), choosing $\Cr{const approx of Tn}=\alpha>\frac{2\Cr{const_time_pert_genau}\cdot C^2}{\es-\overline{\eta}(v_0)}$ the latter term tends to infinity, supplying us with \eqref{eq:T_x_proof}. 
	
	To complete the proof, equation \eqref{eq:LimitTn/n} is  a direct consequence of \eqref{eq:ApproxT_n} and  \eqref{eq:LLNEmpLogMG}.	
\end{proof}
Recall the definition $\munder^{u_0,a}=\munder^{\xi,u_0,a}$  from  \eqref{eq:def_munder}. 
\begin{corollary}
	\label{cor:breakpoint_limit}
	For all $u_0\in\mathcal{I}_{\text{PAM}}$ and  $a>0$ we have 
	\begin{equation}
	\label{eq:LLN_breakpoint_PAM}
	\frac{\munder^{u_0,a}(t)}{t} \tend{t}{\infty} v_0\quad\p\text{-a.s.}
	\end{equation}
\end{corollary}
\begin{proof}
	For an upper bound, we have $\limsup_{t\to\infty} \frac{\munder^{u_0,a}(t)}{t} = \limsup_{t\to\infty} \frac{\munder^{u_0,a}(t)}{T_{ \munder^{u_0,a}(t) }^{ u_0,a }}\frac{ T_{ \munder^{u_0,a}(t) }^{ u_0,a }  }{ t } \leq v_0$, where the last inequality is due to $T_{ \munder^{u_0,a}(t) }^{ u_0,a } \leq t$ and \eqref{eq:LimitTn/n}. To get a lower bound, we 
	can use the properties of the Lyapunov exponent from Proposition~\ref{prop:lyapunov}, giving $\liminf_{t\to\infty} \frac{\munder^{u_0,a}(t)}{t} \geq v$ for all $v\in(0,v_0)$, and we can conclude.
\end{proof}
\begin{lemma}
	\label{le:T_mtAlmostLinear}
	For every $a>0$ there exists a constant $\Cl{constant lemma T_mt eq t}=\Cr{constant lemma T_mt eq t}(a)>0$ and a $\p$-a.s.\  finite random variable $\Cl[index_prob_cont_time]{time lemma T_mt eq t} = \Cr{time lemma T_mt eq t}(a)$ such that  for all $u_0\in\mathcal{I}_{\text{PAM}}$ and  $t\geq\Cr{time lemma T_mt eq t},$
	\begin{equation}
	\label{eq:T_mtAlmostLinear}
	t - \Cr{constant lemma T_mt eq t} \leq T^{u_0,a}_{\munder^{u_0,a}(t)} \leq t.
	\end{equation}
\end{lemma}
\begin{proof}
	By definition, the inequality $T^{u_0,a}_{\munder^{u_0,a}(t)}\leq t$ follows directly. To show $t-\Cr{constant lemma T_mt eq t}\leq T^{u_0,a}_{\munder^{u_0,a}(t)}$, recall that due to \eqref{eq:LLN_breakpoint_PAM} we can use time perturbation. Indeed, 
	by defining $\Cr{constant lemma T_mt eq t}:=\Cr{const_time_perturbation}\ln\big( \Cr{const_time_perturbation}\cdot 3C'/a \big)$ with $C'$ from \eqref{eq:PAM_initial} and  $\Cr{const_time_perturbation}$ from  Lemma~\ref{le:time_perturb} b), for all $t$ large enough
	\begin{align*}
	u^{\mathds{1}_{(-\infty,0]}}(t- \Cr{constant lemma T_mt eq t},\munder^{u_0,a}(t)) &\leq \Cr{const_time_perturbation} e^{-\Cr{constant lemma T_mt eq t} / \Cr{const_time_perturbation}} u^{u_0}(t,\munder^{u_0,a}(t)) < \frac{a}{2C'}
	\end{align*}
	and thus, recalling $u^{\mathds{1}_{(-\infty,0]}}(s,x)=\Eprob_{x}^\xi\big[ N^\leq(s,0) \big]$ and Lemma~\ref{le:MomPAMIncrease}, we get the lower bound $T^{u_0,a}_{\munder^{u_0,a}(t)} \geq T^{C'\mathds{1}_{(-\infty,0]},a}_{\munder^{u_0,a}(t)} = T^{\mathds{1}_{(-\infty,0]},a/C'}_{\munder^{u_0,a}(t)} \geq t-\Cr{constant lemma T_mt eq t}$ for all $t$ large enough and we can conclude. 
\end{proof}

Recall definition \eqref{eq:defT_n} for $T_x^{(a)}$. 
\begin{corollary}
	\label{cor:T_mt_sandwich}
	There exists $\overline{K} \in (1,\infty)$ such that $\p$-a.s., for all $a>0$ and for all  $x$  large enough
	\begin{equation}
	\label{eq:T_mt_sandwich}
	\sup_{|y|\leq 1} T^{(a)}_{x+y}-\overline{K}\leq T^{(a)}_x\leq \inf_{|y|\leq 1}T^{(a)}_{x+y}+\overline{K}.
	\end{equation}
\end{corollary}
\begin{proof}
	We set $\overline{K}:=1+\Cr{const_time_perturbation}\big(\ln(2\Cr{const_time_perturbation}\Cr{const_space_perturbation})+\Cr{const_space_perturbation}\big)$. Then due to \eqref{eq:LimitTn/n}, $\p$-a.s.\ for all $x$ large enough we have
	\[ \frac{x+y'}{T^{(a)}_{x+y''}\pm\overline{K}}\in V\quad \forall\ y',y''\in [-1, 1]. \]
	This allows us to apply the inequalities \eqref{eq:time_pert_ungenau} and \eqref{eq:space_pert_ungenau} for $u_0=\mathds{1}_{(-\infty,0]}$. Indeed, for all $|y|\leq 1,$
	\begin{align*}
	u(T^{(a)}_{x+y}-\overline{K},x) &\leq \Cr{const_space_perturbation}e^{\Cr{const_space_perturbation}}u(T^{(a)}_{x+y}-\overline{K},x+y)  \leq \Cr{const_space_perturbation}e^{\Cr{const_space_perturbation}}\Cr{const_time_perturbation}e^{-(\overline{K}-1)/\Cr{const_time_perturbation}}u(T^{(a)}_{x+y}-1,x+y) <\frac{a}{2},
	\end{align*}
	where the first inequality is due to \eqref{eq:space_pert_ungenau}, the second one due to  \eqref{eq:time_pert_ungenau} and the last one uses $u(T^{(a)}_{x+y}-1,x+y)<a$. By Lemma~\ref{le:MomPAMIncrease} we get $u(t,x)<a$ for all $t\leq T^{(a)}_{x+y}-\overline{K}$ and thus the left-hand side in \eqref{eq:T_mt_sandwich}. Analogously, first applying \eqref{eq:space_pert_ungenau} and then \eqref{eq:time_pert_ungenau}, we have 
	\begin{align*}
	u(T^{(a)}_{x+y}+\overline{K},x) &\geq \Cr{const_space_perturbation}^{-1}e^{-\Cr{const_space_perturbation}} u(T^{(a)}_{x+y}+\overline{K},x+y) \geq \Cr{const_space_perturbation}^{-1}e^{-\Cr{const_space_perturbation}}\Cr{const_time_perturbation}^{-1}e^{\overline{K}/\Cr{const_time_perturbation}}u(T^{(a)}_{x+y},x+y) \geq a
	\end{align*}
	for all $|y|\leq 1$, giving the right-hand side of \eqref{eq:T_mt_sandwich}.
\end{proof}

\begin{corollary}
	\label{cor:tightnessPAM}
	Let $\munder^{a}(t)=\munder^{\xi,\mathds{1}_{(-\infty,0]},a}(t)$, $a>0$, be defined in \eqref{eq:def_munder}. Then for all $0<\e\leq M$ there exists $C=C(\e,M)$ such that  $\p$-a.s.\  for all $t$ large enough
	\[ 0\leq \munder^{\e}(t) - \munder^{M} (t) \leq C. \]
\end{corollary}
\begin{proof}
	The first inequality is clear. By Corollary~\ref{cor:breakpoint_limit}, we can use the second inequality from Lemma~\ref{le:space_perturb} b) and get the claim by defining $C:=\Cr{const_space_perturbation} \ln\big( \Cr{const_space_perturbation}\cdot M/\varepsilon \big)$ with $\Cr{const_space_perturbation}$ from Lemma~\ref{le:space_perturb} b) to get $u^{\mathds{1}_{(-\infty,0]}}(t,\munder^\varepsilon(t)-C)\geq M$ and thus $\munder^M(t)\geq \munder^\varepsilon(t)-C$ for all $t$ large enough.  
\end{proof}

\subsection{Proof of Theorem~\ref{th:InvFrontPAM}}
Using the preparatory results from the previous sections, it is now possible to obtain an invariance principle for the front of the solution to \eqref{eq:PAM}. Roughly speaking, up to some error which can be controlled by the results from the previous sections, we have $\munder(t)\approx \ln u(t,v_0t)$ and can then use the invariance principle from Theorem~\ref{th:InvPAM} to conclude.

\begin{proof}[Proof of Theorem~\ref{th:InvFrontPAM}]
	Let $u_0\in\mathcal{I}_{\text{PAM}}$, $a>0$ and abbreviate $u=u^{u_0}$ and $\munder:=\munder^{u_0,a}$. We first assume $\sigma_{v_0}^2>0$. Then we have to show that the sequence of processes
	\begin{equation}
	\label{eq:breampoint_inv}
	[0,\infty) \ni t\mapsto \frac{\munder(nt) - v_0nt}{\sqrt{n\widetilde{\sigma}_{v_0}^2}},\quad n\in\N,
	\end{equation} 
	where $\widetilde{\sigma}_{v_0}^2>0$ is given in \eqref{eq:def_sigma_tilde} below, 
	converges in $\p$-distribution to standard Brownian motion in the Skorohod space  $D([0,\infty))$.  Notice that $[0,\infty)\ni t\mapsto\munder(t)$ might not be c\`{a}dl\`{a}g \emph{only} in $0$. To avoid this issue, above convergence is defined as convergence of the sequence of processes in \eqref{eq:breampoint_inv} where we set $\munder(t) \equiv 0$ for $t$ such that $\munder(t)\leq 0$, making it c\`{a}dl\`{a}g. 
	
	Due to the limiting behavior and the continuity of $x\mapsto u(t,x)$ for $t>0$, the value $r(t):=\munder(t)-v_0t$ is  the largest solution to 
	\[ \ln u(t,v_0t+r(t)) = -\ln 2. \]
	We define 
	\[\mathcal{L}(t,h):= \ln \frac{u(t,v_0t+h)}{u(t,v_0t)},\quad t>0,\ h\in\R, \]
	\[ U(t):=-\ln u(t,v_0t)-\ln 2=\mathcal{L}(t,r(t)),\quad t\geq 0,\]
	let 
	\begin{equation*}
	\delta \in \big( 0, |L(\overline{\eta}(v_0))| \big)
	\end{equation*}
	and $\varepsilon(t)$ be a positive function such that $\varepsilon(t)\to0$ and $\varepsilon(t)t^{1/2}\to\infty$. Then by Lemma~\ref{le:space_perturb}, there is $C(\varepsilon)>0$ such that  for $t$ large enough and all $h\in\R$ fulfilling $C(\varepsilon)\ln t\leq |h|\leq \varepsilon(t)t$ and $v_0+\frac{h}{t}\in V$ we have 
	\begin{align}
	-\big(|L(\overline{\eta}(v_0))|+\delta\big)h \leq \mathcal{L}(t,h) \leq -\big(|L(\overline{\eta}(v_0))|-\delta\big)h. \label{eq:h_between_L}
	\end{align}
	Now, multiplying \eqref{eq:ApproxT_n} by $v_0$, replacing $x$ by $\munder(t)$ in \eqref{eq:ApproxT_n} and recalling that $t-\Cr{constant lemma T_mt eq t}\leq T_{\munder(t)}\leq t$ by Lemma~\ref{le:T_mtAlmostLinear}, we get 
	\begin{equation} \label{eq:anotherAppr}
	\Big|(\munder(t)-v_0t) - \frac{1}{L(\overline{\eta}(v_0))}\sum_{i=1}^{\munder(t)}\big( L_i^\zeta(\overline{\eta}(v_0))-L(\overline{\eta}(v_0)) \big)\Big|\leq v_0\Cr{const prob approx of Tn} + \Cr{constant lemma T_mt eq t}  + v_0 \Cr{const approx of Tn}\ln (\munder(t))
	\end{equation}
	for all $t$ large enough. Next, recall that $\frac{\munder(t)}{t}\to v_0$ by Corollary~\ref{cor:breakpoint_limit} and that the standardized sum $\frac{1}{\sqrt{n}}\sum_{i=1}^{nt}\big( L_i^\zeta(\overline{\eta}(v_0))-L(\overline{\eta}(v_0)) \big)$ converges in distribution to a non-degenerate Gaussian random variable by Lemma~\ref{le:CLT_S_n}. As a consequence, in combination with \eqref{eq:anotherAppr}, we infer that $|r(t)|=|\munder(t)-v_0t|\in [C(\varepsilon)\ln t, \varepsilon(t)t]$ with probability tending to $1$ as $t$ tends to infinity.  This and \eqref{eq:h_between_L} implies 
	\begin{equation}
	\label{eq:r(t)_near_U}
	r(t) \in \Big[ \frac{U(t)}{|L(\overline{\eta}(v_0))|\mp\delta}, \frac{U(t)}{|L(\overline{\eta}(v_0))|\pm\delta}\Big],
	\end{equation}
	with probability tending to $1$  as $t$ tends to  $\infty$, where the upper sign is chosen if $U(t)>0$ and the lower sign if $U(t)<0$. If $\sigma_{v_0}^2>0$, due to $\Lambda(v_0)=0$ and Theorem~\ref{th:InvPAM}, the sequence of processes 
		\begin{equation}
		\label{eq:InvUt}
		[0,\infty)\ni t\mapsto \frac{1}{\sqrt{nv_0\sigma_{v_0}^2}}U(nt),\quad n\in\N,
		\end{equation}
		converges in $\p$-distribution to standard Brownian motion. 
		Because \eqref{eq:r(t)_near_U}  holds for all $\delta>0$ small enough,  Theorem~\ref{th:InvFrontPAM} is a direct consequence of the convergence in distribution of \eqref{eq:InvUt} by choosing 
		\begin{align}
		\label{eq:def_sigma_tilde}
		\widetilde{\sigma}_{v_0}&:=\frac{\sqrt{\sigma_{v_0}^2v_0}}{|L(\overline{\eta}(v_0))|}.
		\end{align}
		where $\sigma_{v_0}^2$ is defined in \eqref{eq:sigmav}. This gives the second part of Theorem~\ref{th:InvFrontPAM}. If $\sigma_{v_0}^2=0$, we can proceed analogously and the first part of Theorem~\ref{th:InvFrontPAM} follows from the first part of Theorem~\ref{th:InvPAM} and \eqref{eq:r(t)_near_U}.
\end{proof}

\section{Log-distance of the fronts of the solutions to PAM and F-KPP}

We finally prove our last main result, Theorem~\ref{th: log-distance breakpoint median}. In Sections~\ref{sec:first_mom_lead} and \ref{sec:sec_mom_lead}, we will assume that $u_0=w_0=\mathds{1}_{(-\infty,0]}$. Indeed, using a comparison argument in the proof of Theorem~\ref{th: log-distance breakpoint median}, it will turn out that this is actually sufficient for or purposes. It should also be mentioned here that the tools we employ are inherently probabilistic. As a consequence, and for notational convenience, we will mostly formulate the respective results in terms of the BBMRE in what follows below; the correspondence to the results in PDE terms is immediate from \eqref{eq:F-K} and Remark \ref{rem:McKeanHeaviside}.

In the case $u_0=w_0=\mathds{1}_{(-\infty,0]}$, using Markov's inequality we infer
\[\prob_x^\xi\left(N^\leq (t,0)\geq1\right)\leq \Eprob_x^\xi\left[N^\leq(t,0)\right] \]
and thus $\munder(t)\geq m(t)$ for all $t\geq0$, which establishes the first inequality in \eqref{eq: log-distance breakp median}. The rest of this section will be dedicated to deriving the second inequality in \eqref{eq: log-distance breakp median}, i.e., that the front of the randomized F-KPP equations lags behind the front of the solution of the parabolic Anderson model at most logarithmically.
We introduce some notation and, recalling the notation $X^\nu$ introduced before \eqref{eq:NtA}, start with considering certain ‘‘well-behaved'' particles 
\begin{equation}
\label{eq:defLeadingPart}
\begin{split}
N_{s,u,t}^{\mathcal{L},a} := \big|\big\{ \nu\in N(s): X^\nu_s\leq 0, H^\nu_{k}\geq u-T^{(a)}_{k}-5\chi_1(\munder^{(a)}(t)) \ \forall k\in\{1,\ldots,\lfloor &\munder^{(a)}(t)\rfloor\}\big\}\big|, \\
&a>0,\ s,t,u\geq 0;
\end{split}
\end{equation}
\eqref{eq:defLeadingPart} 
here, $H^\nu_k:=\inf\{ t\geq0: X^\nu_t=k \}$, the random variable $T^{(a)}_k$ has been defined in \eqref{eq:defT_n}, and 
\begin{align}
\chi_b(x)&:=\Cr{const prob approx of Tn}+b(1+\overline{K}+ \Cr{constant lemma T_mt eq t})+\Cr{const approx of Tn}(\ln x\vee 1),\quad  x\in(0,\infty),\ b\in\R, \label{eq:defChi}
\end{align}
where $\Cr{const prob approx of Tn}$ and  $\Cr{const approx of Tn}$ have been defined in Lemma~\ref{le:ApproxT_n}, $\overline{K}$ is taken from Corollary~\ref{cor:T_mt_sandwich}, the constant  $\Cr{constant lemma T_mt eq t}$ from Lemma~\ref{le:T_mtAlmostLinear}. 
 We abbreviate $N_t^\mathcal{L}:=N_{t,t,t}^\mathcal{L}$ and call the particles contributing to  $N_t^\mathcal{L}$ \emph{leading particles at time $t$}. 
Cauchy-Schwarz immediately gives
\begin{equation} \label{eq:CS}
\prob_x^\xi\left( N^\leq(t,0)\geq 1 \right) \geq \prob_x^\xi\left( N_t^{\mathcal{L}}\geq 1 \right) \geq \frac{\Eprob_x^\xi\left[N_t^{\mathcal{L}}\right]^2}{\Eprob_x^\xi\big[(N_t^{\mathcal{L}})^2\big]}.
\end{equation}
The next two sections are dedicated to deriving an upper bound for the denominator and a lower bound for the numerator 
of the right-hand side, both for $x$ in a neighborhood of $\munder(t)$. 

\subsection{First moment of leading particles}
\label{sec:first_mom_lead}
The biggest chunk of this section will consist of  proving the following first moment bound on the number of leading particles.
Recall the notation $\munder^{(a)}(t)$ from \eqref{eq:munder}.

\begin{lemma}
	\label{le: first moment leading particles inequ}
	For all $a>0$ there exists $\gamma_1=\gamma_1(a) \in (0,\infty)$ such that $\p$-a.s., for all $t$ large enough
	\begin{equation*}
	\inf_{x\in[\munder^{(a)}(t)-1,\munder^{(a)}(t)+1]}\Eprobth_{x}^\xi\big[ N_t^{\mathcal{L},a} \big] \geq t^{-\gamma_1}.
	\end{equation*}
\end{lemma}

\begin{proof}
	Let $a>0$. To simplify notation, we will omit the index $a>0$ in the quantities involved and write $N_{s,u,t}^{\mathcal{L}}:=N_{s,u,t}^{\mathcal{L},a}$, $T^{(a)}_x:=T_x$, $\munder^{(a)}(t):=\munder(t)$ from now on. 
	
	Let $A_{u,t}:=\{H_{k}\geq u-T_{k}-5\chi_1(\munder(t)) \ \forall k\in \{1,\ldots,\lfloor m(t)\rfloor\}\}$, 
	let  $K$ be such that \eqref{eq:defK} holds and set $\overline{t}:=T_{\lfloor \munder(t)\rfloor}$. We obtain for all $t$ large enough that 
	\begin{align} \label{eq:firstfirst}
	\begin{split}
	&\inf_{x\in [\munder(t)-1,\munder(t)+1]}\Eprob_{x}^\xi \left[ N_t^{\mathcal{L}} \right]\geq \frac{\inf_{x\in [\munder(t)-1,\munder(t)+1]}\Eprob_{x}^\xi\left[ N_t^{\mathcal{L}} \right]}{2\Eprob_{\lfloor\munder(t)\rfloor}^\xi\left[ N^\leq(\overline{t},0)  \right]} \geq  \frac{c}{2}\frac{\Eprob_{\lfloor m(t)\rfloor}^\xi\left[N_{t-1,t,t+1}^\mathcal{L}\right]}{\Eprob_{\lfloor m(t)\rfloor}^\xi\left[ N^\leq\left(\overline{t},0\right)  \right]} \\
	&=  \frac{c}{2}\frac{E_{\lfloor m(t)\rfloor}\left[ e^{ \int_0^{t-1}\xi(B_s) \diff s }  ; B_{t-1}\leq 0, A_{t+1,t} \right]}{E_{\lfloor m(t)\rfloor}\Big[ e^{ \int_0^{\overline{t}}\xi(B_s)\diff s };B_{\overline{t}}\leq 0 \Big] } \geq c_1 \frac{ E_{\lfloor m(t)\rfloor}\left[ e^{ \int_0^{t-1}\zeta(B_s) \diff s }  ; B_{t-1}\leq 0, A_{t+1,t} \right] }{E_{\lfloor m(t)\rfloor}\Big[ e^{\int_0^{\overline{t}}\zeta(B_s)\diff s}; B_{\overline{t}}\leq 0 \Big]};
	\end{split}
	\end{align}
here, the first inequality follows from the definition of $T_{\lfloor \munder(t)\rfloor}$, the  second inequality is due to Lemma~\ref{le:leadingPart}, the equality follows using Proposition \ref{prop:many-to-few} and the last inequality is due to $\xi=\zeta+\es$, as well as \eqref{eq:T_mt_sandwich} which gives $\overline{t}=T_{\lfloor \munder(t)\rfloor}\leq T_{\munder(t)}+\overline{K}\leq t+\overline{K}$. The numerator can be bounded from below by
\begin{align*}
&E_{\lfloor m(t)\rfloor}\left[ e^{ \int_0^{t-1}\zeta(B_s) \diff s }  ; H_0\in[t-3\overline{K}-\Cr{constant lemma T_mt eq t},t-1], B_{t-1}\leq 0, A_{t+1,t} \right] \\
&\geq E_{\lfloor\munder(t)\rfloor} \Big[ e^{\int_0^{H_0}\zeta(B_s)\diff s} \left.E_0\big[ e^{\int_0^{r}\zeta(B_s)\diff s};B_{t-1-r}\leq 0 \big]\right|_{r=t-1-H_0}; H_0\in[t-3\overline{K}-\Cr{constant lemma T_mt eq t},t-1], A_{t+1,t} \Big] \\ 
&\geq c_2E_{\lfloor\munder(t)\rfloor} \Big[ e^{\int_0^{H_0}\zeta(B_s)\diff s} ;H_0\in[t-3\overline{K}-\Cr{constant lemma T_mt eq t},t-1], A_{t+1,t} \Big],
\end{align*}
where the second inequality is due to $\zeta\geq -(\es-\ei)$ and $P_0(B_s\leq0)\geq1/2$ for all $s\geq0$. Now using the inclusion $\{ B_{\overline{t}}\leq0 \}\subset\{ H_0\leq\overline{t} \}$ in combination with $\zeta\leq0$, we infer $E_{\lfloor \munder(t)\rfloor}\big[ e^{\int_0^{\overline{t}}\zeta(B_s)\diff s};B_{\overline{t}}\leq 0 \big]\leq E_{\lfloor \munder(t)\rfloor}\big[ e^{\int_0^{H_0}\zeta(B_s)\diff s};H_0\leq \overline{t} \big]$. Thus, recalling $\overline{\eta}(v_0)<0$ and  \eqref{eq:def_P_x_zeta}, we can continue to lower bound \eqref{eq:firstfirst} via
\begin{align*}
&\inf_{x\in [\munder(t)-1,\munder(t)+1]}\Eprob_{x}^\xi \left[ N_t^\mathcal{L} \right] \geq c_3\frac{ E_{\lfloor\munder(t)\rfloor}^{\zeta,\overline{\eta}(v_0)} \Big[ e^{-\overline{\eta}(v_0) H_0}; H_0\in[t-3\overline{K}-\Cr{constant lemma T_mt eq t},t-1], A_{t+1,t} \Big] }{ E_{\lfloor \munder(t)\rfloor}^{\zeta,\overline{\eta}(v_0)}\big[ e^{-\overline{\eta}(v_0) H_0};H_0\leq \overline{t} \big] } \\
&\qquad\geq c_4 \frac{ P_{\lfloor\munder(t)\rfloor}^{\zeta,\overline{\eta}(v_0)} \Big( H_0\in[t-3\overline{K}-\Cr{constant lemma T_mt eq t},t-1],\ H_{k}\geq t+1-T_{k}-5\chi_1(\munder(t)), \, \forall  1\leq k\leq\lfloor\munder(t)\rfloor \Big) }{ P_{\lfloor \munder(t)\rfloor}^{\zeta,\overline{\eta}(v_0)}\big( H_0\leq \overline{t} \big) } \\
&\qquad\geq c_4 P_{\lfloor\munder(t)\rfloor}^{\zeta,\overline{\eta}(v_0)} \Big( H_0\in[\overline{t}-2\overline{K},\overline{t}-\overline{K}-1],\ H_{k}\geq \overline{t}-T_{k}-5\chi_0(\munder(t)), \, \forall 1\leq k\leq\lfloor\munder(t)\rfloor \Big),
\end{align*}
where the last inequality is due to $t\geq T_{\munder(t)}\geq \overline{t}-\overline{K}$ and $\overline{t}\geq T_{\munder(t)}-\Cr{constant lemma T_mt eq t}-\overline{K}$ (by \eqref{eq:T_mt_sandwich} and \eqref{eq:T_mtAlmostLinear}). 
Now as we recall that $\frac{\lfloor\munder(t)\rfloor}{t}\to v_0$, abbreviating $\eta=\overline{\eta}(v_0)$, $n:=\lfloor \munder(t)\rfloor$ and thus $\overline{t}=T_n$, 
 we see that in order to finish the proof, it suffices to show that there exists $\gamma \in (0,\infty)$ such that $\p$-a.s.,
  for all $n\in\N$ large enough,
	\begin{equation}
	\label{le: first moment main eq}
	P_{n}^{\zeta,\eta}\big(  H_0\in[T_n-2\overline{K},T_n-\overline{K}-1], H_k\geq T_n-T_k-5\chi_0(n) \; \forall k\in \left\{1,\ldots,n\right\}  \big) \geq n^{-\gamma}.
	\end{equation}
Using the notation
	\begin{align}
	\widehat{H}^{(n)}_k &:= H_k - E_n^{\zeta,\eta}\left[H_k\right] \quad
	\text{ as well as } \quad
	R_k^{(n)}:=T_n - T_k-E_n^{\zeta,\eta}\left[ H_k \right], \label{eq:def_H_hat_and_R}
	\end{align}
	the  probability in \eqref{le: first moment main eq} can be rewritten as 
	\begin{align} \label{eq:nonStatExp}
	P_{n}^{\zeta,\eta}\big( \widehat{H}^{(n)}_0\in[R_0^{(n)}-2\overline{K},R_0^{(n)}-\overline{K}-1], \widehat{H}^{(n)}_k \geq R_k^{(n)} - 5\chi_0(n)  \ \forall k\in \left\{1,\ldots,n\right\}  \big).
	\end{align}
	In order to facilitate computations, we approximate the sequence $(R_k^{(n)})$ by a stationary one, setting
	\begin{align}
	\rho_i &:=\frac{L_i^{\zeta}(\eta)}{v_0 L(\eta)} - (L_i^\zeta)'(\eta)  =\frac{1}{v_0 L(\eta)} \big( L_i^\zeta(\eta) - L(\eta)\big) - \big( E_n^{\zeta,\eta}[\tau_{i-1}] - \E[ E_n^{\zeta,\eta}[\tau_{i-1}]] \big)\label{eq:def_rho_in} \ \text{ and } \ \\
	{\widehat{R}_k^{(n)}}&:=\sum_{i=k+1}^{n} \rho_i,\quad k< n,
	\label{eq:def_R_hat}
	\end{align}
	where $\tau_{i-1}=H_{i-1}-H_i$, and in the equality we used $E_n^{\zeta,\eta}[\tau_{i-1}]=(L_i^\zeta)'(\eta)$ and $\E[ E_n^{\zeta,\eta} [H_k] ]=\frac{n-k}{v_0}$. Applying inequality \eqref{eq:ApproxT_n} from Lemma~\ref{le:ApproxT_n} and using	the identity  $\E[ E_n^{\zeta,\eta} [H_k] ]=\frac{n-k}{v_0},$ we get that $\p$-a.s., 
	\begin{equation}
	\label{eq:approx_R}
	\begin{split}
	\big| R_k^{(n)} - \widehat{R}_k^{(n)} \big| &\leq 2\left(\Cr{const prob approx of Tn} +  \Cr{const approx of Tn}\ln n\right) \quad\text{for all $n\in\N$ and each $k\in\{0,\ldots,n\}.$}
	\end{split}
	\end{equation}
	From now on we will write $\chi:=\chi_0$. Then by \eqref{eq:approx_R}, the probability in \eqref{eq:nonStatExp} can be lower bounded by 
	\begin{align} \label{eq:PHitEst}
	&P_n^{\zeta,\eta}\big( \widehat{H}^{(n)}_0\in[R_0^{(n)}-2\overline{K},R_0^{(n)}-\overline{K}-1]; \widehat{H}^{(n)}_k \geq \widehat{R}_k^{(n)} - 3\chi(n)  \ \forall k\in \left\{1,\ldots,n\right\} \big).
	\end{align}
	Now, for every $n$, enlarging the underlying probability space if necessary, we introduce two  processes $(B_t^{(i,n)})_{t\geq0}$, $i=1,2$, which are independent from everything else and Brownian motions under $P_n$, starting in $n$, and, without further formal definition, we tacitly assume in the following that the tilting of the probability measure $P_n^{\zeta,\eta}$ of our original Brownian motion {\em also applies to} $(B_t^{(i,n)})_{t\geq0}$, $i=1,2,$ in the obvious way. 
	For $i=1,2$, let  $H^{(i,n)}_k:=\inf\{t\geq0:B_t^{(i,n)}=k\}$, $k\in\Z$, be the corresponding hitting times, $\widehat{H}_k^{(i,n)}:={H}_k^{(i,n)}-E_n^{\zeta,\eta}[{H}_k^{(i,n)}]$ and let $\Sigma_n$ be a random variable which, under $P_n$, is uniformly distributed  on $\{1,\ldots,n-1\}$ and independent of everything else. We define 
	\begin{align}
	\beta_k^{(i,n)}&:=\widehat{H}^{(i,n)}_k - \widehat{R}_k^{(n)},\quad k=n-1,n-2,\ldots,\quad i=1,2, \notag \\
	\beta_k^{(n)}&:=
	\begin{cases}
	\beta^{(1,n)}_k,& \Sigma_n\leq k<n,\\
	\beta^{(1,n)}_{\Sigma_n} + \big( \beta^{(2,n)}_k - \beta^{(2,n)}_{\Sigma_n} \big),&  k< \Sigma_n.
	\end{cases}\notag 
	\end{align}
	The $\xi$-adaptedness of the process $(\widehat{R}_k^{(n)})_{k<n}$ implies that the processes $(\beta^{(i,n)}_k)_{k<n}$, $i=1,2$, are $P_n^{\zeta,\eta}$-independent and have the same distribution as $(\beta_k^{(n)})_{k<n}$. We can therefore rewrite \eqref{eq:PHitEst} as 
	\begin{align}
	P_n^{\zeta,\eta}\left( \beta_k^{(n)} \geq - 3\chi(n)  \ \forall k\in \left\{1,\ldots,n\right\}, \beta_0^{(n)} \in I_n \right)\label{eq: prob beta above},
	\end{align}
	where $I_n:=\big[R_0^{(n)}-\widehat{R}_0^{(n)}-2\overline{K},R_0^{(n)}-\widehat{R}_0^{(n)}-\overline{K}-1\big].$ Due to \eqref{eq:approx_R} we have that
	$\p$-a.s.\  for all $n$ large enough, $R_0^{(n)}-\widehat{R}_0^{(n)}-2\overline{K}  \geq -3\chi(n),$ i.e. 
	\begin{equation} \label{eq:InInclusion}
	I_n\subset [-3\chi(n),\infty).
	\end{equation} 
	For each $k\in\{0,\ldots, n\}$ we introduce
	\begin{align*}
	\overline{\beta}^{(1,n)}_k &:=\beta^{(1,n)}_{n-1-k}-\beta^{(1,n)}_{n-1},\qquad
	\overline{\beta}^{(2,n)}_k :=\beta^{(2,n)}_{k} - \beta^{(2,n)}_0,
	\end{align*}
	and note that
	\begin{equation}
	\label{eq: beta alternative Darst. Endwert}
	\beta_0^{(n)}=\overline{\beta}_{n-1-\Sigma_n}^{(1,n)}-\overline{\beta}_{\Sigma_n}^{(2,n)} + \beta_{n-1}^{(1,n)}.
	\end{equation}
	An illustration	of the various processes introduced above is given in  Figure~\ref{fig:betapath} below. 
	Now the key to bound the probability in \eqref{eq: prob beta above} is the following lemma.

	\begin{lemma}
		\label{le:help_result_1st_mom_lead_part}
		\begin{enumerate}
			\item There exists $\gamma'<\infty$ such that  $\p$-a.s.\  for all $n$ large enough, 
			\begin{align*}
			P_n^{\zeta,\eta}\big(  \overline{\beta}^{(1,n)}_k \geq 0 \ \forall 0\leq k\leq n,\ \overline{\beta}^{(1,n)}_n\geq n^{1/4}\big) & \geq n^{-\gamma'}, \quad \text{ and }\\
			P_n^{\zeta,\eta}\big(  \overline{\beta}^{(2,n)}_k \geq 0 \ \forall 0\leq k\leq n,\ \overline{\beta}^{(2,n)}_n\geq n^{1/4}\big) & \geq n^{-\gamma'}. 
			\end{align*}
			\item There exists  $C(\gamma')>0$ such that  $\p$-a.s.\  for all $n$ large enough,
			\begin{align}
			P_n^{\zeta,\eta}\Big( \max_{1\leq k\leq n,i\in\{1,2\}} \big| \beta_k^{(i,n)}-\beta_{k-1}^{(i,n)} \big| \leq C(\gamma')\ln n \Big) 
			&\geq 1- n^{-3\gamma'}. \label{eq: help result first mom b}
			\end{align}
			\item Let $\delta\in(0,1)$. There exists  $c>0$ such that  for all $x\geq 1$ and all $n\in\mathbb{Z},$
			\begin{align*}
			P_n^{\zeta,\eta}\big( \beta_{n-1}^{(1,n)} \in [x,x+\delta] \big) \geq c\delta e^{-x/c}.
			\end{align*}				
		\end{enumerate}
	\end{lemma}
	Before proving Lemma \ref{le:help_result_1st_mom_lead_part} we will finish the current proof in order not to interrupt the Reader. To this end let
	\[ 
	J_n:=\sup \big \{k \in \{1,\ldots, n-1\} \, : \, I_n- \overline{\beta}_{n-k+1}^{(1,n)}+\overline{\beta}_{k}^{(2,n)}\subset [0,2 C(\gamma')\ln n] \big\},
	\]
	where as always $\sup \emptyset:=-\infty$. We have
	\begin{align}
	&\left\{ \beta_k^{(n)} \geq -3\chi(n)\ \forall 0\leq k\leq n-1,\ \beta_0^{(n)}\in I_n \right\} \notag\\
	&\qquad\quad \supset \Big(\big\{ \overline{\beta}^{(1,n)}_k \geq 0 \ \forall 0\leq k\leq n,\ \overline{\beta}^{(1,n)}_n\geq n^{1/4} \big\} \cap \big\{ \max_{1\leq k\leq n} \big| \overline{\beta}_k^{(1,n)}-\overline{\beta}_{k-1}^{(1,n)} \big| \leq C(\gamma')\ln n \big\} \Big.  \notag\\
	&\qquad\qquad \cap \big\{ \overline{\beta}^{(2,n)}_k \geq 0 \ \forall 0\leq k\leq n,\ \overline{\beta}^{(2,n)}_n\geq n^{1/4} \big\} \cap  \big\{ \max_{1\leq k\leq n} \big| \overline{\beta}_k^{(2,n)}-\overline{\beta}_{k-1}^{(2,n)} \big| \leq C(\gamma')\ln n \big\}  \notag \\
	&\qquad\qquad \Big.  \cap\big\{ \beta_{n-1}^{(1,n)} \in I_n - \overline{\beta}_{n-1-\Sigma_n}^{(1,n)} +\overline{\beta}_{\Sigma_n}^{(2,n)}   \big\} \cap \{ \Sigma_n = J_n\} \Big) \label{eq:betaimplication}. 
	\end{align}
	Indeed, due to \eqref{eq: beta alternative Darst. Endwert}, the fifth event on the right-hand side of \eqref{eq:betaimplication} entails that $\beta_0^{(n)}\in I_n$ must hold.
	On the last two events on the right-hand side of \eqref{eq:betaimplication} we have $\beta^{(1,n)}_{n-1}\geq 0$ and thus the first event on the right-hand side of \eqref{eq:betaimplication} implies  that $\beta^{(n)}_k$ is non-negative for $k\geq\Sigma_n$. The third event then implies monotonicity at times $k<\Sigma_n$. Since $I_n\subset [-3\chi(n),\infty)$ due to \eqref{eq:InInclusion}, this gives the first condition on the left-hand side of \eqref{eq:betaimplication}. 
	\begin{figure}[t]
		\centering
		\includegraphics[width=0.6\linewidth]{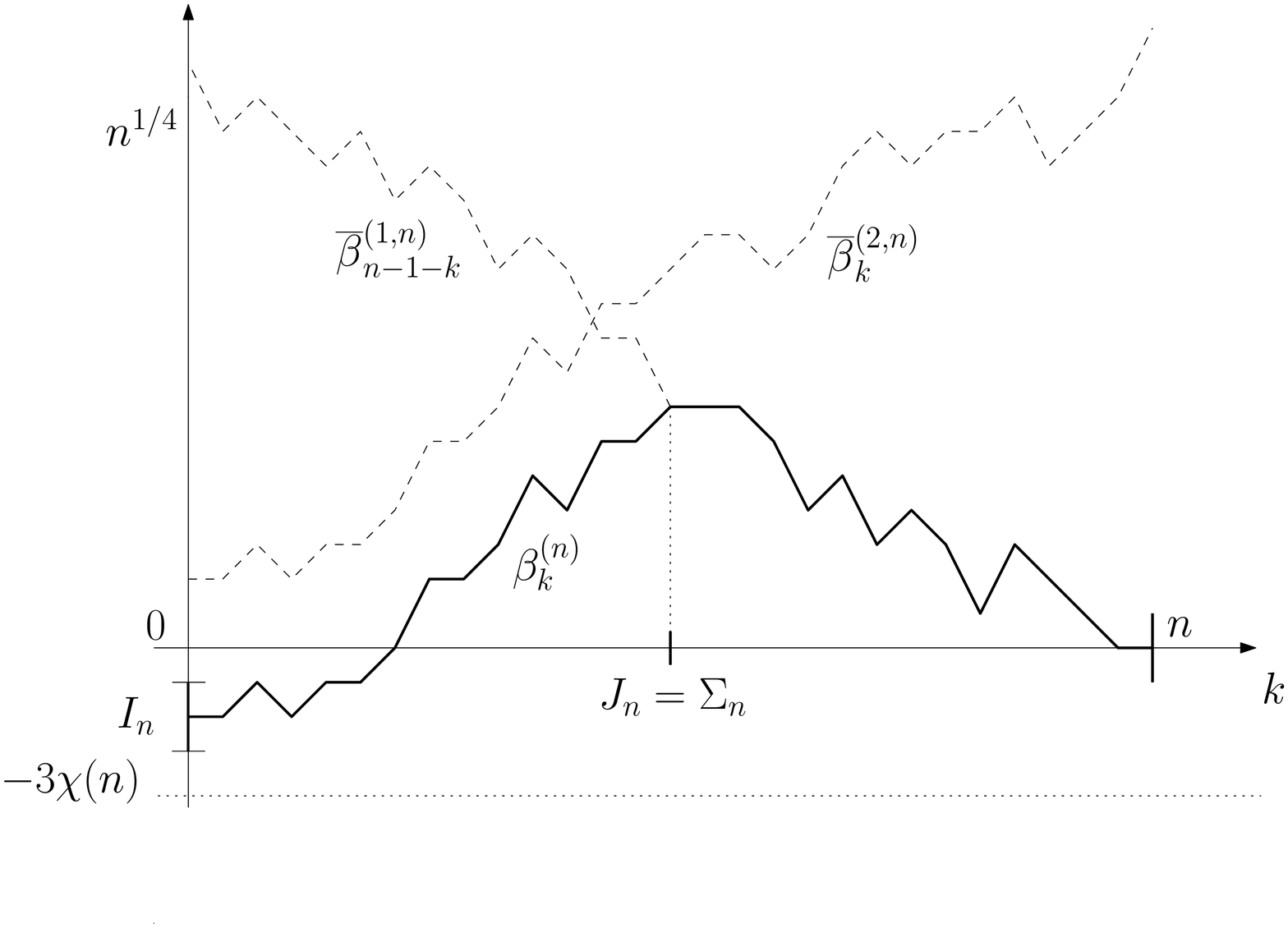}
		\caption{Illustration of \eqref{eq:betaimplication}.}
		\label{fig:betapath}
	\end{figure}
	Now the first and third event on the right hand-side of \eqref{eq:betaimplication} are independent under $P_n^{\zeta,\eta}$ and their probabilities are bounded from below by $n^{-\gamma'}$ due to Lemma~\ref{le:help_result_1st_mom_lead_part} a). Thus, as a consequence of Lemma~\ref{le:help_result_1st_mom_lead_part} b), for $n$ large enough the probability of the first four events is bounded from below by $n^{-2\gamma'}-n^{-3\gamma'}$. 
	Furthermore, the first four events imply that $J_n\in\{1,\ldots,n-1\}$. Thus, due to Lemma~\ref{le:help_result_1st_mom_lead_part}, conditionally on the occurrence of the first four events, the probability of the last two events on the right-hand side in \eqref{eq:betaimplication} can be bounded from below by $cn^{-1}e^{-C(\gamma')\ln n/c}\geq n^{-\gamma''}$ for $n$ large enough. The proof of \eqref{le: first moment main eq} and thus of Lemma~\ref{le: first moment leading particles inequ} is completed by the choice $\gamma_1>2\gamma'+\gamma''$.
\end{proof}

\begin{proof}[{Proof of Lemma~\ref{le:help_result_1st_mom_lead_part}}]
	
$(b)$ Because $H_k^{(1,n)} - H_{k-1}^{(1,n)} \overset{d}{=}\tau_{k-1}$ under $P_n^{\zeta,\eta}$, by recalling $\E\big[E_n^{\zeta,\eta}[\tau_k]\big]=\frac{1}{n}$ and the definition of $\widehat{R}_k^{(n)}$ from \eqref{eq:def_rho_in}, we have $\beta_{k}^{(1,n)}-\beta_{k-1}^{(1,n)}= \tau_{k-1} - \frac{L_{k}^\zeta(\eta)}{v_0L(\eta)}$. Now $L_k^\zeta(\eta)$ is $\p$-a.s.\  bounded by Lemma~\ref{le:log mom fct and derivatives}. Furthermore, for all $\theta$ such that  $|\theta|\leq |\overeta|$ (where $\triangle=[\undereta,\overeta]$), we have
	\[ 0\leq E_n^{\zeta,\eta}\big[ e^{\theta \tau_{k-1}} \big] = E_k^{\zeta,\eta}\big[ e^{\theta \tau_{k-1}} \big] \leq \left( 	E_{k}\big[  e^{(-\es+\ei+\undereta)H_{k-1} } \big] \right)^{-1}=e^{\sqrt{2(\es-\ei+|\eta_{*}|)}}<\infty, \]
	where the last equation is due to \cite[(2.0.1), p.~204]{handbook_brownian_motion}. I.e., $\tau_k$ has uniform exponential moments under $P_n^{\zeta,\eta}$ and thus \eqref{eq: help result first mom b} follows by a union bound in combination with the exponential Chebyshev inequality.
	
	$(c)$	We have $\beta_{n-1}^{(1,n)}=H_{n-1}^{(1,n)} -E_n^{\zeta,\eta}[H_{n-1}^{(1,n)}] - \widehat{R}_{n-1}^{(n)}=H_{n-1}^{(1,n)} - E_n^{\zeta,\eta}[\tau_{n-1}]-
	\rho_n$, thus recalling definition \eqref{eq:def_rho_in},   the event in (c) is equivalent to $\{ H_{n-1}^{(1,n)}\in[x,x+\delta]+\frac{L_n^\zeta(\eta)}{v_0L(\eta)} \}$. Because $\frac{L_n^\zeta(\eta)}{v_0L(\eta)}$ is uniformly bounded and non-negative, it suffices to check that for every $C>0,$ there exists $c>0$ such that  $\inf_{y \in [0, C]}P_n^{\zeta,\eta}(H_{n-1}^{(1,n)}\in [x+y,x+y+\delta])\geq c\delta e^{-x/c}$ for all $x\geq 1$. Indeed, recalling \eqref{eq:def_P_x_zeta}, we can lower bound
	\begin{align*}
	&P_n^{\zeta,\eta}(H_{n-1}^{(1,n)}\in [x+y,x+y+\delta]) \ge
	E_n\big[ e^{-(\es-\ei-\undereta)H_{n-1}}; H_{n-1}^{(1,n)}\in [x+y,x+y+\delta] \big]  \\
	&\quad \geq e^{-(\es-\ei-\undereta)(x+y+\delta)} P_n\big( H_{n-1}^{(1,n)} \in[x+y,x+y+\delta]\big) \geq \frac{\delta  e^{-(\es-\ei-\undereta)(x+y+\delta)}}{\sqrt{2\pi (x+y+\delta)^3}}e^{-\frac{1}{2(x+y)}}, 
	\end{align*}
	where the last inequality is due to \cite[(2.0.2), p.~204]{handbook_brownian_motion}. 
	Now since the latter term can be lower bounded by $c\delta e^{-x/c}$, uniformly in $y\in[0,C],$ the claim follows. 
		
$(a)$  We will prove the second inequality,  and explain the modifications that are necessary to show the first one at the end of the proof.  For later reference it will serve our purposes to exclude some potentially bad behavior of the process  $(\widehat{R}_k^{(n)} - \widehat{R}_{k-1}^{(n)})_k$. To do so,  we take advantage of the next claim, the proof of which we provide after concluding the proof of Lemma~\ref{le:help_result_1st_mom_lead_part}.
\begin{claim}
	\label{cl:rho}
	For each $n\in\Z$, the sequence
	$(\rho_i)_{i\in\Z}$  consists of $\p$-centered and $\p$-stationary random variables, and the family $(\rho_i)_{i\in\Z}$ is bounded $\p$-a.s. In addition, $\rho_i$ is $\F^{i-1}$-adapted and there exists  $\Cl{const_pho}>0$ such that $\p$-a.s.,  for all $k,n\in\Z$, $k<n$, we have  
	\begin{equation}
	\label{eq: exponential decay incremts}
	\big| \E\big[ \rho_{n-k} | \F^n \big] \big| \leq \Cr{const_pho}\cdot \big(  \psi(k/2)+e^{-k/\Cr{const_pho}}\big).
	\end{equation}
	Furthermore, there exists $\overline{\sigma}\in[0,\infty)$ such that  $n^{-1/2}\sum_{l=1}^n \rho_l$ and $n^{-1/2}\sum_{l=1}^n \rho_{-l}$ converge in $\p$-distribution to $\overline{\sigma}X$ as $n\to\infty$, where $X\sim\mathcal{N}(0,1)$ is a standard Normal random variable. 
\end{claim}	
Now due to \eqref{eq: exponential decay incremts}, $(\rho_i)_{i\in\Z}$ fulfills the  conditions of Corollary~\ref{cor:rioDepExp}. As a consequence we deduce that 
for  $k\in\N$ and $x\geq 0$ we have
\(\p\big( \sum_{l=1}^{k}{\rho}_{l} \geq x \big) \leq c_1e^{-\frac{x^2}{c_1k}},\)
which, using stationarity, can be extended to the maximal inequality (e.g.\ by \cite[Theorem 1]{max_bernstein}) 
\begin{align}
\p\Big( \max_{ 0 \leq k\leq y} \big( \widehat{R}_{r+k}^{(n)}-\widehat{R}_r^{(n)} \big) \geq x \Big) &=
\p\Big( \max_{ 0 \leq k\leq y}\sum_{l=0}^{k}{\rho}_{l} \geq x \Big) \leq c_2e^{-\frac{x^2}{c_2y}}\quad\forall r,y\in\Z,\ x\geq 0. \label{eq:max_rhat}
\end{align}
Furthermore, recalling \eqref{eq:def_H_hat_and_R}, \eqref{eq:def_rho_in} and \eqref{eq:def_R_hat}, the increments of the process $(\overline{\beta}^{(2,n)}_k)_k$ can be written as  
\begin{equation*}
\begin{split}
\overline{\beta}^{(2,n)}_k-\overline{\beta}^{(2,n)}_{k-1} &= \big(H_k-H_{k-1} - E_n^{\zeta,\eta}[H_k-H_{k-1}]\big) - \Big( \sum_{i=k+1}^n  \rho_i - \sum_{i=k}^n \rho_i \Big) \\
&=\big(-\tau_{k-1} + (L_k^\zeta)'(\eta)\big) - \Big( - \frac{L_k^\zeta(\eta)}{v_0L(\eta)} + (L_k^\zeta)'(\eta) \Big) =  \frac{L_k^\zeta(\eta)}{v_0L(\eta)} - \tau_{k-1}.
\end{split}
\end{equation*}
$\p$-a.s., by Lemma~\ref{le:log mom fct and derivatives} b), the last fraction in the previous display is positive and uniformly bounded away from zero and infinity, whereas under $P_n^{\zeta,\eta}$, $\tau_{k-1}$ is an absolutely continuous random variable with positive density on $(0,\infty)$. 
	Therefore, for the constant 
	\begin{equation*}
	a:=\frac{1}{4} \sup_{k\in\Z} \essinf_{\xi} \big(\overline{\beta}_k^{(2,n)}-\overline{\beta}^{(2,n)}_{k-1}\big),
	\end{equation*}
	we have $\essinf_{k,n\in\Z:\ k\leq n,\ \xi} P_n^{\zeta,\eta}(\overline{\beta}_k^{(n)}-\overline{\beta}_{k-1}^{(n)}\geq 2a)\geq \delta$ for some universal constant 
	$\delta\in(0,1)$. We now split the environment into $\overline{\xi}(j):=(\xi(l))_{l\geq j}$ and $\underline{\xi}(j):=(\xi(l))_{l<  j}$ and set $t_0=t_{-1}:=0$ as well as $t_i:=2^i$ for $i\geq 1$. Furthermore, we introduce two constants: $\overline{c}>0$, which is defined in \eqref{eq:def_c_overline} below, and $\overline{C}>0$, which is independent of $\overline{c}$ and will be chosen large enough such that the sums in \eqref{eq:sum_finite1} and \eqref{eq:sum_finite2} below are finite.  For $i\geq1$, we define the random variables
	\begin{align} \label{eq:ZiDef}
	\begin{split}
	Z_i^{(n)}
	&:=\underset{\overline{\xi}(t_{i+1})}{\text{ess inf}}\ \underset{x\geq a t_{i-1}^{1/2}}{\inf} P_{n}^{\zeta,\eta}\big( \overline{\beta}^{(2,n)}_{t_i}\geq a t_i^{1/2},\overline{\beta}_k^{(2,n)} \geq t_{i}^{1/4}\ \forall k\in\{t_{i-1},\ldots, t_i \} \ \big|\ \overline{\beta}_{t_{i-1}}^{(2,n)}= x \big) \\
	&=\underset{\overline{\xi}(t_{i+1})}{\text{ess inf}}\  P_{n}^{\zeta,\eta}\big( \overline{\beta}^{(2,n)}_{t_i}\geq a t_i^{1/2},\overline{\beta}_k^{(2,n)} \geq t_{i}^{1/4}\ \forall k\in\{t_{i-1},\ldots, t_i \} \ \big|\ \overline{\beta}_{t_{i-1}}^{(2,n)}= a t_{i-1}^{1/2} \big) ,
	\end{split}
	\end{align}
	where \ 
	$\underset{\overline{\xi}(x)}{\text{ess inf}}$ means taking the essential infimum with respect to $\overline{\xi}(x),$
	and where the second equality is due to the monotonicity of the first probability in \eqref{eq:ZiDef} as a function in $x.$
	Thus, as a random variable, $Z_i^{(n)}$ is measurable with respect to $\F_{t_{i+1}}$. Now since $\overline{\beta}_k^{(2,n)}$ is $\F^k$-measurable, we have that $Z_i^{(n)}$ is $(\F^{t_{i-1}}\cap\F_{t_{i+1}})$-measurable. Setting $i(n):=\ln _2\big( \lfloor \big(\overline{C}\ln n\big)^2\rfloor \big)$, 
	we further define
	\begin{align*}
	Y^{(n)} &:= P_n^{\zeta,\eta} \Big( \overline{\beta}^{(2,n)}_k\geq 0\ \forall \lfloor \overline{C} \ln n \rfloor\leq k\leq t_{i(n)}, \ \overline{\beta}^{(2,n)}_{t_{i(n)}} \geq a \lfloor \overline{C} \ln n \rfloor\  \Big|\ \overline{\beta}^{(2,n)}_{\lfloor \overline{C} \ln n \rfloor} = 2a \lfloor \overline{C} \ln n \rfloor \Big).
	\end{align*}
	Writing $j(n):=\lceil \log_2(n) \rceil$, due to the Markov property of the process $\overline{\beta}^{(2,n)}$ under $P_n^{\zeta,\eta}$, we have $\p$-almost surely that for all $n$ large enough, 
	\begin{align}
	P_n^{\zeta,\eta}\left( \overline{\beta}_n^{(2,n)}\geq n^{1/4},\overline{\beta}_k^{(2,n)}\geq 0\  \forall k\leq n \right)  &\geq \prod_{k=1}^{\lfloor \overline{C} \ln n\rfloor} P_n^{\zeta,\eta}\big(  \overline{\beta}^{(2,n)}_k - \overline{\beta}^{(2,n)}_{k-1}\geq 2a \big) \cdot Y^{(n)}\cdot \prod_{i=i(n)+1}^{j(n)} Z_i^{(n)} \notag \\
	&\geq  \delta^{\lfloor \overline{C}\ln n\rfloor}\cdot  Y^{(n)} \cdot  \exp\Big\{ \sum_{i=i(n)+1}^{j(n)} \ln Z_i^{(n)} \mathds{1}_{B_i^{(n)}} \Big\}, \label{eq:lower_bound}
	\end{align}
	where the event 
	\begin{align*}
	B_i^{(n)} &:= \Big\{ \max_{{r\in[t_{i-1},t_i],\ 0\leq k\leq \frac{5}{2} a t_i^{1/2}/ \overline{c}}} \big( \widehat{R}_{r+k}-\widehat{R}_r \big)< at_{i-1}^{1/2} / 16 \Big\}
	\end{align*}
	occurs $\p$-almost surely for all $i\in[i(n),j(n)]$ and all $n$ large enough. 
	Indeed, by \eqref{eq:max_rhat} we have 
	\begin{align}
	\sum_n \sum_{i=i(n)}^{j(n)} \p\big((B_i^{(n)})^c\big) &\leq \sum_n \sum_{i=i(n)}^{j(n)}t_{i-1} \p\Big( \max_{0\leq k \leq  \frac{5}{2}a t_i^{1/2} / \overline{c}} \big( \widehat{R}_{k}-\widehat{R}_0 \big) \geq at_{i-1}^{1/2} / 16 \Big)\notag \\
	&\leq c_3\sum_n  n\sum_{i=i(n)}^{j(n)} e^{-a^2\overline{c} t_{i-1}^{1/2}/c_3} \leq c_4\sum_n  n\log_2(n) e^{-a^2 \overline{C} \ln n/c_4} <\infty, \label{eq:sum_finite1} 
	\end{align}
	where the last inequality holds true for $\overline{C}$ large enough. Thus, the Borel-Cantelli lemma implies  that $\p$-a.s., for all $n$ large enough the events $B_i^{(n)}$ occur for all $i\in[i(n),j(n)]$.  
	Furthermore, it is possible to show that $\p$-almost surely, for all $n$ large enough we have $Y^{(n)}\geq n^{-\gamma''}$. We postpone a proof of this fact, because it uses the same arguments as the following paragraph and we will describe necessary adaptations afterwards, cf.\ page~\pageref{page:Y_bounded}.  Thus, for the time being it remains to show that there exists $\widetilde{c}>0$ such that  $\p$-almost surely, for all $n$ large enough, 
	\begin{equation}\label{eq: limsup log Z_i bounded}
	\sum_{i=i(n)}^{j(n)}\ln(Z_i^{(n)})\mathds{1}_{B_i^{(n)}}  \geq -\widetilde{c}\cdot j(n).
	\end{equation}
	The second inequality in Lemma~\ref{le:help_result_1st_mom_lead_part} (a) then follows from \eqref{eq:lower_bound} with $\gamma'>\overline{C}\ln(1/\delta) + \gamma''+\widetilde{c}/\ln(2)$. 
	
	In order to establish \eqref{eq: limsup log Z_i bounded}, it is enough to show that there exist $c'',\theta>0$, independent of $\widetilde{c}$, such that for all $i$ large enough, 
	\begin{equation}\label{eq: exp mom of log Z_i exists}
	\sup_n\E\Big[ e^{-\theta \ln(Z_i^{(n)})\mathds{1}_{B_i^{(n)}}} \Big]\leq c''.
	\end{equation}
	Indeed, if \eqref{eq: exp mom of log Z_i exists} holds, setting $\widetilde{Z}_i^{(n)}:=\ln(Z_i^{(n)})\mathds{1}_{B_i^{(n)}}$, by Markov's inequality we have
	\begin{align*}
	\p&\Big(  \sum_{i=i(n)}^{j(n)}\widetilde{Z}_{i}^{(n)} < - \widetilde{c} \cdot j(n) \Big) \leq \p\Big( \sum_{k=0}^3 \sum_{i=\lceil i(n)/4\rceil}^{\lfloor \frac{j(n)}{4}\rfloor -1 }  \widetilde{Z}_{4i+k}^{(n)} < -\widetilde{c}\cdot  j(n)  \Big)  \\
	&\quad \leq \sum_{k=0}^3 \p\Big( \sum_{i=\lceil i(n)/4\rceil}^{\lfloor \frac{j(n)}{4}\rfloor-1}  \widetilde{Z}_{4i+k}^{(n)} <- \widetilde{c}\cdot j(n)/4 \Big)
	\leq 4 e^{-\theta \widetilde{c}\cdot j(n)/4} \max_{k=1,\ldots,4} \E\Big[ e^{-\theta \sum_{i=\lceil i(n)/4\rceil}^{\lfloor \frac{j(n)}{4}\rfloor-1} \widetilde{Z}_{4i+k}^{(n)}} \Big].
	\end{align*}
	We will only estimate the above expectation for the case $k=0;$ the cases $k \in \{1,2,3\}$  can be estimated similarly. Now $\widetilde{Z}_{4i}^{(n)}$  is $\F^{t_{4i-1}}$-measurable, hence, also recalling $t_{4i-1}-t_{4i-2}=2^{4i-2}$, by \eqref{eq:MIX} we have 
	\begin{equation*}
	\E\big[e^{-\theta  \widetilde{Z}_{4i}^{(n)}} | \F_{t_{4i-2}} \big] \leq (1+\psi(2^{4i-2}))\E[e^{-\theta  \widetilde{Z}_{4i}^{(n)}}].
	\end{equation*}
	Since furthermore $\widetilde{Z}_{4(i-1)}^{(n)}$ is $\F^{t_{4i-2}}$-measurable, we obtain via iterated conditioning that 
	\begin{align*}
	\E\Big[ e^{-\theta \sum_{i=\lceil i(n)/4\rceil}^{\lfloor \frac{j(n)}{4}\rfloor}  \widetilde{Z}_{4i}^{(n)}} \Big] &= \E\Big[ \E\big[ \cdots \E\big[\E\big[ e^{-\theta \sum_{i=\lceil i(n)/4\rceil}^{\lfloor \frac{j(n)}{4}\rfloor} \widetilde{Z}_{4i}^{(n)}} \given \F_{t_{4j(n)-2}}\big] \given \F_{t_{4j(n)-6}} \big] \cdots \given \F_{t_{2}} \big] \Big] \\
	& \leq \prod_{i=\lceil i(n)/4\rceil}^{\lfloor \frac{j(n)}{4}\rfloor}(1+\psi(2^{4i-2}))\E[e^{-\theta  \widetilde{Z}_{4i}^{(n)}}] \leq (c_6 c'')^{j(n)},
	\end{align*}
	for some $c_6>0$ and $n$ large enough. Choosing $\widetilde{c}$ large enough, by a Borel-Cantelli argument similar to the proof of Lemma~\ref{le:concEtaEmpLT},  inequality \eqref{eq: limsup log Z_i bounded} would follow.

	Thus, in order to show \eqref{eq: exp mom of log Z_i exists}, note that 
	because 
	\begin{equation}
	\label{eq:equal_distr}
	Z_i^{(n)}=Z_i^{(n)}(\xi(\cdot))=Z_i^{(n-k)}(\xi(\cdot+k))\overset{d}{=}Z_i^{(n-k)}(\xi(\cdot))=Z_i^{(n-k)},
	\end{equation}
	we can drop the supremum in \eqref{eq: exp mom of log Z_i exists}. 
	In the following, we first choose $i$ large enough (and from then on fixed) such that several estimates in the remaining part of the proof hold, and afterwards we adapt $n=n(i)$ to ensure $0\leq i\leq \lceil \log_2(n) \rceil$. For simplicity, we write $Z_i:=Z_i^{(n)}$, $\beta_k:=\overline{\beta}_k^{(2,n)}$, $\widehat{H}_k:=H_k-E_n^{\zeta,\eta}[H_k]$,  $\widehat{R}_k:=\widehat{R}_k^{(n)}$ and  define  
	\begin{align*} 
	\overline{\rho}_k^{(j)}&:=\underset{\overline{\xi}(j)}{\text{ess sup}} \ \rho_k, \quad 
	\overline{R}_k^{(j)}:=\sum_{l=0}^{k} \overline{\rho}_l^{(j)},\quad  0\leq k\leq j.
	\end{align*}
	Furthermore, 
	Thus, $\overline{\rho}_k^{(j)}$ is $\F_j$-measurable and $\overline{R}_{k+l}^{(j)}-\overline{R}_k^{(j)}$ is $(\F^k\cap \F_j)$-measurable for all $l\geq0$. Let $M_R := \esssup \rho_0$ and $L:=at_i^{1/2},$ and note that the latter choice corresponds to diffusive scaling. Then we define 
	\begin{equation}
	\label{eq:r_0_s_0}
	\begin{split}
	r_0&:=t_{i-1} ,\quad m:=\frac{L}{16 M_R}, \quad
	s_0:=\big(\inf\big\{ k\geq r_0:\ \overline{R}_k^{(k+m)} -\overline{R}_{r_0}^{(k+m)}\geq L/8 \big\}- 1\big)  \wedge t_i,
	\end{split}
	\end{equation}
	and for $j\geq 1$ let 
	\begin{align} \label{eq:m}
	\begin{split}
	r_{j} &:= s_{j-1} + \Big \lceil \frac{L}{8 M_R} \Big \rceil,\\
	s_{j} &:= \big(\inf\big\{ k\geq r_{j} : \overline{R}_k^{(k+m)} - \overline{R}_{r_{j}}^{(k+m)} \geq L/8 \big\}-1\big) \wedge \left( r_{j} + (t_i-t_{i-1}) \right).
	\end{split}
	\end{align}
	Heuristically, $s_j$ is the first time after which the process $\overline{R}$ (and thus $\widehat{R}$) increases at least by the amount $L/8$ after time $r_j$. Such large increments of  $\widehat{R}$ are potentially  troublesome, since as a consequence, the process $\beta$ might decrease too much and cause the event in the definition of $Z_i$ to have too small  probability. In order to cater for this inconvenience, we start noting that by definition, $s_j-r_j$ is bounded by $t_i-t_{i-1}$ and   $\F_{s_{j}+m}$-measurable, and $r_{j+1}-(s_j+m)\geq m.$ Thus, by condition \eqref{eq:MIX}, for every non-negative measurable function  $f$    we notice for later reference that
	\begin{equation}
	\label{eq:almostIndep}
	\E[ f(s_j-r_j) \, |\, \F^{r_{j+1}} ] \leq \big(1+\psi(m)\big)\E[f(s_j-r_j)].
	\end{equation} 
	Next, we define 
	\begin{align*}
	\mathcal{G}_j &:= \Big\{ \inf_{r_j\leq k\leq s_j}(\widehat{H}_k-\widehat{H}_{r_j}) \geq -L/8,\ \beta_{s_j}\geq 2L \Big\}, \quad
	\mathcal{G}_j' := \Big\{ \inf_{s_{j}\leq k \leq r_{j+1}}(\widehat{H}_k-\widehat{H}_{s_j})\geq -L/8 \Big\}, \\
	J&:=\inf\big\{ j: s_j-r_j = t_i-t_{i-1} \big\}\wedge \inf\{j:s_j\geq t_i\}, \quad \text{as well as} \quad
	\mathcal{G}:=\bigcap_{j=0}^J\mathcal{G}_j \cap \bigcap_{j=0}^{J-1}\mathcal{G}_j',
	\end{align*}	
	and claim that 
	\begin{equation}
	\label{eq: Zi bounded below}
	Z_i \geq P_n^{\zeta,\eta}\big( \mathcal{G} \ | \ \beta_{t_{i-1}} = a t_{i-1}^{1/2} \big).
	\end{equation}
	Indeed, on $[r_0,s_0]$, the process $\overline{R}$ (and thus also $\widehat{R}$) increases by at most $L/8$, and  the process $\widehat{H}$ decreases by at most $L/8$ on $\mathcal{G}_0$. Moreover, for $j\geq 1$, on $[s_{j-1},r_j]$, the process $\widehat{R}$ increases by at most $L/8$, and $\widehat{H}$ decreases by at most $L/8$ on $\mathcal{G}_j'$. Finally, on $[r_j,s_j]$, the process $\overline{R}$ (and thus $\widehat{R}$) increases by at most $L/8$, and on $\mathcal{G}_j$, $\widehat{H}$ decreases by at most $L/8$ and $\beta_{s_j}\geq 2L$. All in all, conditioning on $\beta_{t_{i-1}}=at_{i-1}^{1/2}=L/\sqrt{2}\geq L/2$, we have  
	$\beta_k\geq L/4\geq t_{i}^{1/4}$ for $k\in[r_0,s_0]$ and $\beta_k\geq L$ for all $k\in[s_0,s_J]$. Since by definition, $s_J\geq t_i$, we get $\beta_{t_i}\geq L=at_i^{1/2}$, implying \eqref{eq: Zi bounded below}.
	
	Furthermore, we can continue to lower bound
	\begin{equation}
	P_n^{\zeta,\eta}\big( \mathcal{G} \ | \ \beta_{t_{i-1}} = a t_{i-1}^{1/2} \big) \ge
	P_n^{\zeta,\eta}\left( \mathcal{G}_0 \ | \ \beta_{r_0}=L/\sqrt{2} \right)\prod_{j=0}^{J-1}P_n^{\zeta,\eta}\left( \mathcal{G}_j' \right)\prod_{j=1}^{J} P_n^{\zeta,\eta}\left( \mathcal{G}_j \ | \ \beta_{r_j}=2L \right). \label{eq:prob_bound_below}
	\end{equation}
	To see this, successively condition on $\beta_{r_j}\geq 2L,$ $j =1,\ldots,J,$ and use the Markov property of the process $\widehat{H}$ as well as the fact that $x\mapsto P_n^{\zeta,\eta}\left( \mathcal{G}_j \, |\, \beta_{r_j} = x \right)$ is increasing. Then use the fact that under $P_n^{\zeta,\eta}$, the event  $\mathcal{G}_j'$ is independent of $\beta_{r_j}$ by the independence of the increments of $\widehat{H}$, $j=0,\ldots, J-1$. 
	
	In order to lower bound \eqref{eq:prob_bound_below}, observe that under $P_n^{\zeta,\eta}$, the sequence $(\widehat{H}_{k+1}-\widehat{H}_{k})_{k\geq r_j}$ consists of independent and centered random variables, whose $P_n^{\zeta,\eta}$-moment generating function is finite in a neighborhood of zero. Thus, the  central limit theorem 
entails that for $i$ large enough we have $P_n^{\zeta,\eta}(\mathcal{G}_j')\geq1/2$ for all relevant choices of $j.$ 
	Moreover, 
	\begin{equation*}
	P_n^{\zeta,\eta}\left( \mathcal{G}_j \ | \ \beta_{r_j}=2L \right)\geq P_n^{\zeta,\eta} \Big( \widehat{H}_{s_j}-\widehat{H}_{r_j}\geq 5L/2, \inf_{r_j\leq k\leq s_j}(\widehat{H}_k - \widehat{H}_{r_j})\geq -L/8 \Big).
	\end{equation*}
	We see that both events are nondecreasing in the (independent) increments of $\widehat{H}$. By Harris' inequality (\cite[Theorem 2.15]{boucheron_conc_inequ}) we get
	\begin{align*}
	P_n^{\zeta,\eta}& \Big( \widehat{H}_{s_j}-\widehat{H}_{r_j}\geq 5L/2, \inf_{r_j\leq k\leq s_j}(\widehat{H}_k - \widehat{H}_{r_j})\geq -L/8 \Big) \\
	&\geq P_n^{\zeta,\eta} \big( \widehat{H}_{s_j}-\widehat{H}_{r_j}\geq 5L/2\big)\cdot P_n^{\zeta,\eta} \Big( \inf_{r_j\leq k\leq s_j}(\widehat{H}_k - \widehat{H}_{r_j})\geq -L/8 \Big).
	\end{align*}
	Recalling $s_j-r_j\leq t_i-t_{i-1}= \frac{L^2}{2a^2}$, a Gaussian scaling yields $P_n^{\zeta,\eta} \Big( \inf_{r_j\leq k\leq s_j}(\widehat{H}_k - \widehat{H}_{r_j})\geq -L/8 \Big)\geq c_7>0$. 
	To bound the first factor, we recall that by \eqref{eq:expMix2} and \eqref{eq:expMix4}, we have $\p$-a.s.\ 
	\[0\leq \overline{\rho}_{r_j+l}^{(r_j+k+m)} -\rho_{r_j+l}  \leq  c_8 (\psi(m/2)+e^{-m/c_8})\quad\text{for all }l\leq k\leq t_i/2.   \]
	Because $m=\frac{L}{16 M_R}$ and $t_i/2 = \frac{L^2}{2a^2}$, due to \eqref{eq:MIX} we finally get  for all $i$ (and thus $L$) large enough (due to  $\psi(x)\cdot x\to0$ ($x\to\infty$), which itself is due to summability of $\psi(k)$), that
	\begin{equation}
	\label{eq:rho_diff}
	\begin{split}
	0&\leq \big(\overline{R}_{r_j+k}^{(r_j+k+m)}-\overline{R}_{r_j}^{(r_j+k+m)}\big) - \big( \widehat{R}_{r_j+k}-\widehat{R}_{r_j} \big) = \sum_{l=1}^k \big(\overline{\rho}_{r_j+l}^{(r_j+k+m)}  - \rho_{r_j+l} \big) \\
	&\leq c_8L^2( \psi(L/16M_R)+e^{-L/c_8})\leq  L/16\quad \text{for all }k\in \Big \{0,\ldots, \frac{L^2}{2a^2} \Big \}. 
	\end{split}
	\end{equation}
	
	By $s_j-r_j\leq t_i-t_{i-1}= \frac{L^2}{2a^2}$ and \eqref{eq:m}, we see that $s_j-r_j\geq L/16$ for all $i$ large enough. 
	Recall that under $P_n^{\zeta,\eta}$, the sequence $\widehat{H}_{s_j}-\widehat{H}_{r_j}$ is a sum of independent centered random variables, whose moment generating function is uniformly bounded in a neighborhood of the origin. Then by \cite[(1)]{hoeffding_reversed}, we can apply \cite[Theorem~4]{hoeffding_reversed} in the following manner: 
	Let  $c'>0$ be as in \cite[Theorem 4]{hoeffding_reversed} and in the notation of the latter theorem we choose $k=s_j-r_j$,  $\alpha=\essinf_{\zeta,k<n} E_n^{\zeta,\eta}[(\widehat{H}_k-\widehat{H}_{k-1})^2]>0$, $M=|\eta|/2$, $u_1=\ldots =u_k=1$ and 
	\begin{equation}
	\label{eq:def_c_overline}
	\overline{c}:=c'\cdot\frac{|\eta|}{2}\cdot \alpha
	\end{equation}
	Then a lower Bernstein-type inequality from \cite[Theorem~4]{hoeffding_reversed}
	gives  that on $B_i^{(n)}$ we have
	\begin{equation}
	\label{eq:Bernstein_low}
	P_n^{\zeta,\eta} \big( \widehat{H}_{s_j}-\widehat{H}_{r_j}\geq 5L/2\big) \geq c_9e^{-\frac{L^2}{c_9 (s_j-r_j)}}.
	\end{equation}
	Note that the condition in $B_i^{(n)}$ makes \cite[Theorem~4]{hoeffding_reversed}  applicable by ensuring 'enough' summands $s_j-r_j$ and is the main reason we have to introduce the sets $B_i^{(n)}$.
	We will write $c=c_7\wedge c_9$ from now on. 	
	Using \eqref{eq: Zi bounded below} in combination with the lower bounds for the factors of \eqref{eq:prob_bound_below} just derived, the term in \eqref{eq: exp mom of log Z_i exists} can be bounded from above by
	\begin{align}
	\E&\big[ e^{-\theta \ln(Z_i)} \big] \leq c_{10}\E\Big[ \exp\Big\{  \theta J\ln(2/c)  +\theta\sum_{j=0}^J \frac{L^2}{c(s_j-r_j)}  \Big\} \Big] \notag \\
	&\quad\leq c_{10} \sum_{k=0}^\infty \E\Big[ \Big(\frac{2}{c}\Big)^{\theta k}   \exp\Big\{ \theta \frac{L^2}{c(s_k-r_k)} \mathds{1}_{s_k-r_k=t_i-t_{i-1}}  +\theta\sum_{j=0}^{k-1} \frac{L^2}{c(s_j-r_j)}\mathds{1}_{s_j-r_j<t_i-t_{i-1}}  \Big\} \Big] \notag \\
	&\quad\leq c_{10} \sum_{k=0}^\infty \Big( \frac{2}{c} \Big)^{\theta k}\big(1+\psi(m)\big)^k e^{\frac{2\theta L^2}{ct_{i-1}}} \cdot  \prod_{j=0}^{k-1}\E\Big[ \exp\Big\{ \frac{\theta L^2}{c(s_j-r_j)} \Big\}\mathds{1}_{s_j-r_j < t_i/2} \Big], \label{eq: sum converge}
	\end{align}
	where we recall $m$ from \eqref{eq:m}, and the last inequality is due to \eqref{eq:almostIndep} in combination with $t_i-t_{i-1}=t_i/2.$
	For the latter expectation we have
	\begin{align}
	&\E\left[ \exp\Big\{ \frac{\theta L^2}{c(s_j-r_j)} \Big\}\mathds{1}_{s_j-r_j < t_{i}/2} \right] = \int_0^\infty \p\Big( e^{ \frac{\theta L^2}{c(s_j-r_j)} }\mathds{1}_{s_j-r_j < t_{i}/2} \geq x \Big) \diff x \notag\\
	&\leq e^{\frac{2\theta L^2}{ct_i}}\p\Big( s_j-r_j<t_i/2 \Big) + \int_{e^{\frac{2\theta L^2}{ct_i}}}^{\infty} \p\Big( e^{ \frac{\theta L^2}{c(s_j-r_j)} } \geq x \Big) \diff x\label{eq: int expectation}.
	\end{align}
Substituting $x=e^{\frac{\theta L^2}{cy}}$, the second summand can be written as
	\begin{align*}
	\int_{0}^{t_i/2} \frac{\theta L^2}{cy^2} e^{\frac{\theta L^2}{cy}}\p\left( s_j-r_j \leq y \right) \diff y.
	\end{align*}
	In order to obtain an upper bound, we start with the probability inside the integral and get
	\begin{align*}
	\p\left( s_j-r_j \leq y \right) &= \p\Big( \max_{ 1 \leq k\leq y} \sum_{l=1}^{k}\overline{\rho}_{r_j+l}^{(r_{j}+k+m)} \geq L/8 \Big) \leq \p\Big( \max_{ 1 \leq k\leq y} \sum_{l=1}^{k}{\rho}_{r_j+l} \geq L/16 \Big)\\
	&\leq c_{11}e^{-\frac{L^2}{c_{11} y}}, \quad \forall\  y \in [0,t_i/2];
	\end{align*}
	here the first inequality is due to \eqref{eq:rho_diff} and the last inequality due to \eqref{eq:max_rhat}. 
	
	Putting these bounds together, the second summand in \eqref{eq: int expectation} can be bounded from above by
	\begin{align*}
	\int_{0}^{t_i/2} \frac{\theta L^2}{cy^2} e^{\frac{ \theta L^2}{cy}} c_{10}e^{-\frac{L^2}{yc_{10}}
	} \diff y  &= c_{11}\int_{0}^{t_i/2} \frac{\theta L^2}{cy^2} e^{\frac{ L^2}{cy}(\theta-c_{12})}  \diff y \leq c_{13}\int_0^{1/2a^2} \frac{\theta}{z^2}e^{\frac{1}{cz}(\theta-c_{12})} \diff z \\
	&\leq c_{14}\theta \int_{0}^\infty e^{\frac{x(\theta-c_{11})}{c}} \diff x.
	\end{align*}
	Now the latter term can be made arbitrarily close to zero by choosing $\theta>0$ small enough. Furthermore, again choosing $\theta>0$ small enough, the first term in \eqref{eq: int expectation} is strictly smaller than one by the central limit theorem from Claim~\ref{cl:rho}. Thus,  $\theta>0$ can be chosen small enough such that for all  $i$ large enough, the sum on the right-hand side in \eqref{eq: sum converge} converges, with a finite upper bound independent of $i.$ 
		
	\label{page:Y_bounded}
	To show $Y^{(n)}\geq n^{-\gamma''}$, we adapt the strategy in the proof of \eqref{eq: limsup log Z_i bounded}, i.e.\ set $L:=2a \overline{C}\ln n $,  $r_0=\lfloor \overline{C}\ln n\rfloor$ and $J:=\inf\{j:s_j\geq \lfloor (\overline{C} \ln n)^2\rfloor  \}$ and keep the other  definitions as in \eqref{eq:r_0_s_0} and \eqref{eq:m}. Then by the same argument below display \eqref{eq: exp mom of log Z_i exists}, $Y^{(n)}\geq n^{-\gamma''}$ for some suitable $\gamma''>0$  follows if $\E[ e ^{-\theta Y^{(n)}} ]\leq c$ for some constant $c>0$, some small $\theta>0$ and all $n$ large enough. 
	But this follows (as in the argument leading to the definition of $B_i^{(n)}$), if the process $(\widehat{R}_k - \widehat{R}_{k-1})_k$ does not decrease too fast, see the Borel-Cantelli argument below display \eqref{eq:Bernstein_low}, which itself is a consequence of
	\begin{align}
	\sum_n (\overline{C} \ln n)^2\cdot \p \big( \sup_{0\leq k\leq L/8\overline{c}} \big(\widehat{R}_k-\widehat{R}_0\big)\geq \frac{L}{8} \big) &\leq c_{15}\sum_n (\overline{C} \ln n)^2 e^{-a\overline{C}\ln n/c_{15}} <\infty. \label{eq:sum_finite2}
	\end{align}
	This completes the proof of the second inequality in  Lemma~\ref{le:help_result_1st_mom_lead_part} a). 
		
	We will now explain how to adapt the latter arguments for the proof of the first inequality in a). We define  $\beta_k=\overline{\beta}^{(1,n)}_{k}-\overline{\beta}^{(1,n)}_{n-1}$. In the definition of  $Z_i^{(n)}$, we have to take the essential infimum over $\overline{\xi}(n-t_{i-2})$ and have to replace the subscripts $k$ of $\beta_k$ by $n-k$, i.e. ‘‘running backwards'' from $n$. Thus,  $Z_i^{(n)}$ is $(\F^{n-t_i}\cap\F_{n-t_{i-2}})$-measurable. It is then enough to consider the case $n=0$ due to the argument in \eqref{eq:equal_distr}. Writing $Z_i:=Z_i^{(0)}$, $\beta_k:=\beta_k^{(0)}$, $\widehat{H}_k:=H_k-E_0^{\zeta,\eta}[H_k]$,  $\widehat{R}_k:=\widehat{R}_k^{(0)}$ and defining 
	\begin{align*} 
	\overline{\rho}_k^{(j)}&:=\underset{\overline{\xi}(j)}{\text{ess sup}} \ \rho_k,\qquad \overline{R}_k^{(j)}:=\sum_{l=k+1}^{0} \overline{\rho}_l^{(j)},\quad  k<0,\ k\in\Z,
	\end{align*}
	we have to adapt the definitions of $r_j$ and $s_j$ by the expressions 
	\begin{align*}
	r_0&:=-t_{i-1}, \qquad 	s_0:=\big(\sup\big\{ k\leq r_0:\ \widehat{R}_k - \widehat{R}_{r_0}\geq L/8 \big\}+ 1\big)  \vee (-t_i), \\
	r_{j} &:= s_{j-1} - \Big \lceil \frac{L}{8 M_R} \Big \rceil, \quad s_{j}':= s_{j-1} -  \frac{L}{16 M_R},\quad j\geq 1,\\
	s_{j} &:= \big(\sup\big\{ k\leq r_{j} : \overline{R}_k^{(s'_{j})} - \overline{R}_{r_{j}}^{(s_{j}')} \geq L/8 \big\}+1\big) \vee \left( r_{j} - (t_i-t_{i-1}) \right),\quad j\geq 1,\\
	J&:=\inf\{j: s_j-r_j=-t_i \}\vee\sup\{ j: s_j\leq -t_i \}.
	\end{align*}
	The remaining part of the proof essentially follows the same steps as for the second inequality in a).
\end{proof}	
\begin{proof}
	[Proof of Claim~\ref{cl:rho}]
	Boundedness, stationarity and adaptedness are direct consequences of the corresponding properties of the sequences $(L_i^\zeta(\eta))_{i\in\Z}$ and $((L_i^\zeta)'(\eta))_{i\in\Z}$ and Lemma~\ref{le:log mom fct and derivatives}. Display \eqref{eq: exponential decay incremts} is due to Lemma~\ref{le:expMixLi}. To show that the central limit theorem, we note that the sequence $(\rho_i)_{i\in\Z}$ fulfills the same conditions as the sequence $(\widetilde{L}_i)_{i\in\Z}$ in the proof of Lemma~\ref{le:CLT_S_n}   
\end{proof}
	
\subsection{Second moment of leading particles}
\label{sec:sec_mom_lead}
Recall the notation $N_{t}^{\mathcal{L}}$ from below \eqref{eq:defChi} and that of $\munder(t)$ from \eqref{eq:munder}. For the second moment of the leading particles we now prove the following upper bound.
 \begin{lemma}
 	\label{le:SecMomLeadingPart}
 	For every function $F$ fulfilling \eqref{eq:p_k} and for every $a>0$, there exists $\gamma_2=\gamma_2(F,a)<\infty$ such that  $\p$-a.s.,  for all $t$ large enough,
 	\begin{equation}
 	\label{eq:SecMomLeadingPartIneq}
 	\sup_{x\in[\munder^{(a)}(t)-1,\munder^{(a)}(t)+1]}\Eprobth_{x}^\xi\big[ \big(N_t^{\mathcal{L},a}\big)^2 \big] \leq t^{\gamma_2}.
 	\end{equation}
 \end{lemma}
 
 \begin{proof}
 	We omit the superscript $a$ in the quantities involved and use the same abbreviations as in the beginning of the proof of Lemma~\ref{le: first moment leading particles inequ}.
 	
 	We want to show \eqref{eq:SecMomLeadingPartIneq} with the help of the second-moment formula~\eqref{eq:f-k-2}. To this end, define the function $\varphi^\xi_t:[0,t]\to\mathbb{Z}\cup\{-\infty\}$,
 	\begin{align*}
 	\varphi_t(s)&:=\varphi_t^\xi(s) 
 	:= \lfloor \munder(t)\rfloor \wedge\sup \big\{ k\in\Z: s\in \big [t-T_{k+1}-5\chi_1(\munder(t)),t-T_{k}-5\chi_1(\munder(t)) \big) \big\},
 	\end{align*}
 	where $\sup \emptyset :=-\infty$ and $\chi_1$
	has been defined in \eqref{eq:defChi}. 
 	Due to $T_k=0$ for all $k\leq0$ (recall the notation $T_k$ from \eqref{eq:defT_n1} and \eqref{eq:defT_n}), we have $1\leq\varphi_t(0)\leq \lfloor\munder(t)\rfloor.$ Furthermore, $\varphi_t(t)=-\infty$, because $T_k\geq0$ and $\chi_1^\xi(\munder(t))\geq0$. To apply \eqref{eq:f-k-2}, the following upper bound will prove useful.
 	\begin{claim}
 		\label{claim:sec_mom_function}
 		We have 
 		\begin{equation}
 		\label{eq:repr_lead_part}
 		N_t^{\mathcal{L}}\leq \left|\left\{ \nu\in N(t):X_t^\nu\leq 0,X_s^\nu >   \varphi_t(s) \ \forall s\in\left[0,t \right] \right\}\right|
 		\end{equation}
 		and $\p$-a.s.\ for all $t$ large enough, the function $[0,t]\ni s\mapsto\varphi_t(s)$ is a non-increasing, c\`{a}dl\`{a}g step function. 
 	\end{claim}
 In order not to hinder the flow of reading, we postpone the proof of the latter claim to the end of the proof of Lemma~\ref{le:SecMomLeadingPart}. 
  	\begin{figure}[h!]
 		\centering
 		\includegraphics[height=0.25\linewidth]{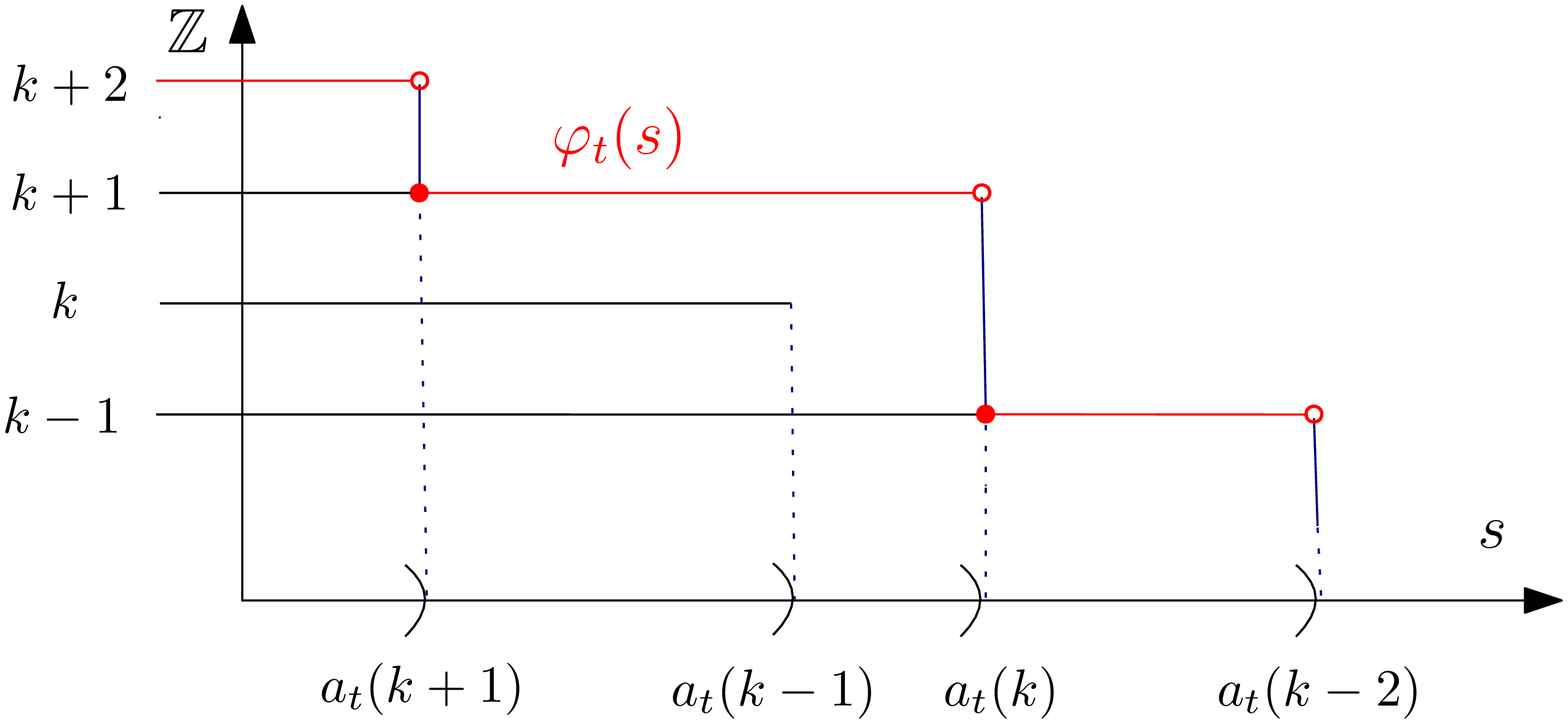}
 		\caption{Illustration of $\varphi_t$, which is the red line. We denote $a_t(k):=t-T_k-5\chi^\xi(\munder(t))$. Note that the sequence $(a_t(k))_{k\in\Z}$ does not have to be monotone and thus the interval  $[a_t(k+1),a_t(k))$ might be empty. In this case the graph of $\varphi_t$ jumps  at least two steps at time $s=a_t(k)$.}
 	\end{figure}
 	By the Feynman-Kac formula (cf.\ Proposition \ref{prop:many-to-few}) and \eqref{eq:repr_lead_part}, we have
 	\begin{align}
	\begin{split}
 	\Eprob_{x}^\xi\big[ (N_t^{\mathcal{L}})^2 \big] &\leq  \Eprob_{x}^\xi\left[ N_t^{\mathcal{L}} \right] + (m_2-2)\int_0^{t} E_{x}\left[ e^{\int_0^s\xi(B_r)\diff r}
 	\xi(B_s)\mathds{1}_{ \left\{B_r\geq\varphi_{t}(r)\ \forall r\in[0,s]\right\}  } \right. \label{eq:sec_mom_lead_fk} \\
 	&\qquad\qquad\qquad\left. \times \left( E_{y}\left[ e^{\int_0^{t-s}\xi(B_r)\diff r} \mathds{1}_{ \left\{B_r\geq\varphi_{t}(r+s)\ \forall r\in[0,t-s]\right\}, B_{t-s}\leq 0  } \right] \right)_{|_{y=B_s}}^2 \right] \diff s.
	\end{split}
 	\end{align}
 	For the first summand we have 
 	\begin{align}
 	\sup_{x\in [\munder(t)-1,\munder(t)+1]}\Eprob_{x}^\xi\left[ N_t^{\mathcal{L}} \right] &\leq \sup_{x\in [\munder(t)-1,\munder(t)+1]}\Eprob_{x}^\xi\left[ N^\leq(t,0) \right]  \leq c_1\Eprob_{\munder(t)+1}^\xi\left[ N^\leq(t,0) \right] \leq  \frac{c_1}{2}, \label{eq: second mom. leading part. first moment bound}
 	\end{align}
 	where the first inequality is due to the first inequality in  \eqref{eq:space_pert_ungenau} and the last one due to the definition of $\munder(t)$.  	Recall that the Markov property provides us with
 	\[ E_x\left[ e^{\int_0^s\xi(B_r)\diff r}
 	E_{y}\big[ e^{\int_0^{t-s}\xi(B_r)\diff r} \mathds{1}_{ \left\{B_{t-s}\leq 0  \right\} } \big]_{|_{y = B_s}}  \right] =\Eprob_x^\xi[N^\leq(t,0)]. \]
 Using $\xi\leq\es$ and the two previous displays, the second summand in \eqref{eq:sec_mom_lead_fk} can thus be bounded from above by
 	\begin{align}  \label{eq:moment_integral}
	\begin{split}
&\es \cdot (m_2-2)\sup_{x\in [\munder(t)-1,\munder(t)+1]} \Eprob_x^\xi[N^\leq(t,0)]  \cdot \int_0^t \sup_{y\geq\varphi_t(s)} E_y\big[ e^{\int_0^{t-s}\xi(B_r)\diff r}; B_{t-s}\leq 0 \big]\diff s\\
 	& \quad \leq \frac{\es(m_2-2)c_1}{2} \int_0^t \sup_{y\geq\varphi_t(s)} \Eprob_y^\xi[N^\leq(t-s,0)]\diff s.
	\end{split}
 	\end{align}
 	It thus suffices to upper bound $\sup_{y\geq\varphi_t(s)} \Eprob_y^\xi[N^\leq(t-s,0)]$ by a polynomial in $t.$
 	We treat different areas for $s$ and $y$ separately and we will need an additional claim, the proof of which will be provided after this  proof. It guarantees that the assumptions of the time perturbation Lemma~\ref{le:time_perturb} are satisfied in our setting. 
 	\begin{figure}[t]
 		\centering
 		\includegraphics[width=0.6\linewidth]{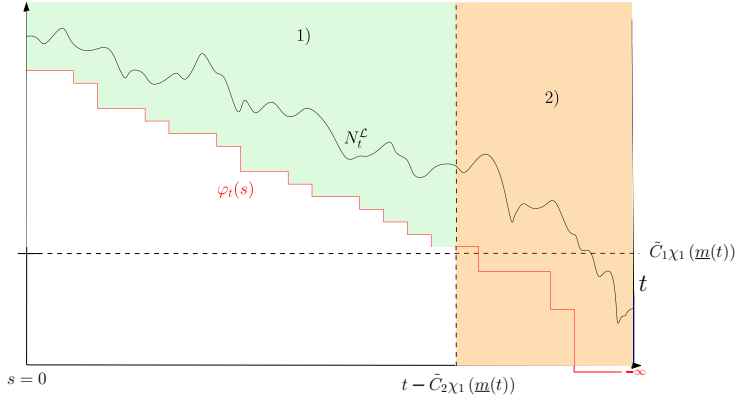}
 		\caption[]{Leading particles and the different areas in the proof of Lemma~\ref{le:SecMomLeadingPart}.} 
 	\end{figure}

 	\begin{claim}
 		\label{claim: hilflemma in 2nd moment proof}
 		There exists $\widetilde{C}_1\in (0,\infty)$ such that  $\p$-a.s.\  for all $t$ large enough and all $y\geq \widetilde{C}_1 \chi_1(
 		\munder(t) )$ we have $\frac{y}{T_y-1},\frac{y}{T_{y} + \overline{K}+ 5\chi_1\left( \munder(t) \right)}\in V$, where $V$ is defined in \eqref{eq:defVTriangle}.
 		Furthermore, there exists $\widetilde{C}_2=\widetilde{C}_2(\widetilde{C}_1) \in (0,\infty)$ such that  $\p$-a.s.\   for all $t$ large enough and all $s\in[ 0,t-\widetilde{C}_2\chi_1( \munder(t)))$ we have $\varphi_t(s)\geq \widetilde{C}_1 \chi_1(
 		\munder(t))$.
 	\end{claim}
 We choose $\gamma'>5\Cr{const_time_perturbation}\Cr{const approx of Tn}$. Then, 	recalling the definition of $\chi_1$ from \eqref{eq:defChi} and that $\p$-a.s.\ $\frac{\munder(t)}{t}\to v_0$ by Corollary~\ref{cor:breakpoint_limit}, for  $t$ large enough, the statements from Claim~\ref{claim: hilflemma in 2nd moment proof} hold true, we have $\Cr{const_time_perturbation}e^{\Cr{const_time_perturbation}(\overline{K}+1+5\chi_1(\munder(t)))}  \leq t^{\gamma'}$ and also $T_{\lfloor y\rfloor+1}\leq T_y+\overline{K}$ for all $y\geq\widetilde{C}_1\chi_1(\munder(t))$ by Corollary~\ref{cor:T_mt_sandwich}. \\ 
 $1)$ Let $s\in\big[0,t-\widetilde{C}_2\chi_1\left( \munder(t) \right)\big)$ and $y\geq \varphi_t(s)$. Then by Claim~\ref{claim: hilflemma in 2nd moment proof}, $y\geq \widetilde{C}_1 \chi_1\left( 	\munder(t) \right)$ and thus $\frac{y}{T_y-1},\frac{y}{T_y +\overline{K}+ 5\chi_1\left( \munder(t) \right)}\in V$.  By definition of $\varphi_t$ we have $s\geq t-T_{\lfloor y\rfloor+1}-5\chi_1\left( \munder(t) \right)$ and thus $T_y + \overline{K}+5\chi_1\left( \munder(t) \right)\geq T_{\lfloor y\rfloor+1}+5\chi_1\left( \munder(t) \right)\geq t-s$. Thus, by  Lemma~\ref{le:MomPAMIncrease}, we infer that $\Eprob_y^\xi\left[ N^\leq(t-s,0) \right]\leq 2\Eprob_y^\xi\left[ N^\leq\left( T_y + \overline{K}+5\chi_1\left( \munder(t) \right),0 \right)\right],$ and then the second inequality in \eqref{eq:time_pert_ungenau} entails that for all $t$ large enough,
 	\begin{align}
 	\sup_{\substack{0\leq s<t-\widetilde{C}_2\chi_1( \munder(t))\\ y\geq \varphi_t(s)}} \Eprob_y^\xi[N^\leq(t-s,0)] &\leq 2\Cr{const_time_perturbation}e^{\Cr{const_time_perturbation}(\overline{K}+1+5\chi_1(\munder(t)))}\sup_{y\in\R} \Eprob_y^\xi\left[ N^\leq((T_y-1)\vee 0,0) \right] 
 	\leq t^{\gamma'}. \label{eq:domain1}
 	\end{align}	
 	$2)$ The remaining part of the domain above the graph of $\varphi_t$ not controlled by $1)$ is a subset of 
 	\[\big\{ (s,y)\in [0,t]\times\mathbb{R}: t-\widetilde{C}_2\chi_1\left( \munder(t) \right)\leq s\leq t \big\}.\]
 	Recalling the definition of $\chi_1$ from \eqref{eq:defChi}    and that $\p$-a.s., $\frac{\munder(t)}{t}\to v_0$ by  Corollary~\ref{cor:breakpoint_limit}, choosing $\gamma''>\es\widetilde{C}_2\Cr{const approx of Tn}$, on the the above domain we get that $\p$-a.s., for all $t$ large enough,
 	\begin{align}
 	\Eprob_{y}^\xi\left[ N^\leq \left( t-s,0 \right) \right]\leq 2\Eprob_{y}^\xi\big[ N^\leq \big( \widetilde{C}_2\chi_1\left(\munder(t)\right),0 \big) \big]\leq 2e^{ \es \widetilde{C}_2\chi_1\left(\munder(t)\right) }\leq t^{\gamma''}.
 	\label{eq:domain2}
 	\end{align} 	
	
 	To conclude the proof, defining $\gamma_2:=1\vee\gamma'\vee\gamma''+1$, inequalities \eqref{eq: second mom. leading part. first moment bound}, \eqref{eq:moment_integral} and the estimates \eqref{eq:domain1} and \eqref{eq:domain2} for the term  $E_y^\xi\left[ N^\leq (t-s,0) \right]$  entail the statement of Lemma~\ref{le:SecMomLeadingPart}.
 \end{proof}

 \begin{proof}
 	[Proof of Claim~\ref{claim:sec_mom_function}] Let $a_t(k):=t-T_k-5\chi_1^\xi(\munder(t))$ and recall the definition of $N_t^\mathcal{L}$:
 	\[N_t^\mathcal{L}=\left| \big\{
 	\nu\in N(t):\ X_t^\nu \leq 0,\ H_k^\nu \geq a_t(k)\ \forall\ k\in\{1,\ldots,\lfloor \munder(t)\rfloor\} \big\} \right|.\]
 	To prove \eqref{eq:repr_lead_part}, note that $H_k^\nu \geq a_t(k)$ if and only if $X_s^\nu>k$ for all $s <a_t(k)$.  But  the property $X_s^\nu>k$ for all $s\in[0,a_t(k))$ and all $k\in\{1,\ldots,\lfloor \munder(t)\rfloor\}$ implies that $X_s^\nu>\sup\{ k\in\Z:\ s\in[a_t(k+1),a_t(k)) \}\wedge\lfloor \munder(t)\rfloor = \varphi_{t}(s)$ for all $s\in[0,t]$ and thus  \eqref{eq:repr_lead_part} is shown. The property of $\varphi_t$ being a c\`{a}dl\`{a}g step-function is a direct consequence of the use of left-closed, right-open intervals in the definition of $\varphi_t$. It remains to show  that $s\mapsto\varphi_{t}(s)$ is non-increasing. 
 	For this purpose, let us first prove by induction in $k=\lfloor\munder(t)\rfloor,\lfloor\munder(t)\rfloor-1,\ldots$  that for all $t$ large enough and all $k\leq\munder(t),$ 
 	\begin{equation}
 	\label{eq:phi_t_non-incr_1}
 	[0, a_t(k-1)) \subset \bigcup_{l=k}^{\lfloor\munder(t)\rfloor} [a_t(l),a_t(l-1))
 	\end{equation}
 	holds. By Corollary~\ref{cor:T_mt_sandwich} and Lemma~\ref{le:T_mtAlmostLinear}, there exist $\overline{K},\Cr{constant lemma T_mt eq t}>0$ such that $t-T_{\lfloor \munder(t)\rfloor-1 }\leq t- T_{\munder(t)}+ \overline{K}\leq \Cr{constant lemma T_mt eq t}+\overline{K}$ and thus $ a_{t}(\lfloor \munder(t)\rfloor)\leq 0$ for all $t$ large enough. Assume now that \eqref{eq:phi_t_non-incr_1} holds for some $k\leq\munder(t)$. Then
 	\[ \left[0, a_t(k-2) \right) \subset \left[0,a_t(k-1)\right) \cup \left[ a_t(k-1), a_t(k-2) \right) \subset \bigcup_{l=k-1}^{\lfloor\munder(t)\rfloor} [a_t(l),a_t(l-1)), \]
 	where the last inclusion is due to induction hypothesis. Thus, we have shown \eqref{eq:phi_t_non-incr_1}. Now let $0\leq s_1\leq s_2$.  Assume there exists $k_2$ such that  $s_2\in \left[ a_t(k_2),a_t(k_2-1) \right)$. Then by \eqref{eq:phi_t_non-incr_1}, there exists $k_1\in\Z$ with $k_2\leq k_1\leq\lfloor\munder(t)\rfloor$, 
 	such that $s_1\in\left[ a_t(k_1),a_t(k_1-1) \right)$. By definition we get $\varphi_t(s_1)\geq\varphi_t(s_2)$. If no such $k_2$ exists,  then  $\varphi_t(s_2)=-\infty\leq\varphi_{t}(s_1)$. 
 \end{proof}

 \begin{proof}[Proof of Claim~\ref{claim: hilflemma in 2nd moment proof}]
 	We write $V=[v_*,v^*]$. Since $\p$-a.s.\ we have $\frac{y}{T_y}\to v_0\in \text{int}(V)$ by \eqref{eq:LimitTn/n}, it follows that $\frac{y}{T_y-1},\frac{y}{T_y}\in V$ for all $y$ large enough. Among others, there exists $\varepsilon=\varepsilon(v_*,v^*,v_0)>0$ and $\mathcal{N}'(\xi)$ such that $v_*(1+\varepsilon)\leq\frac{y}{T_y}\leq (1-\varepsilon)v^*$ for all $y\geq\mathcal{N}'$. Choosing $\widetilde{C}_1>\frac{5 v^*}{\varepsilon}$, this implies $ 1\leq \frac{T_y +\overline{K}+ 5\chi_1\left( \munder(t) \right)}{T_y}\leq 1+\e$ for all $y\geq \widetilde{C}_1\cdot \chi_1\left( \munder(t) \right)$ and all $t$ large enough. Thus, we get
 	\[ v_* \leq  \frac{y}{T_y}\cdot\frac{T_y}{T_y+\overline{K} + 5\chi_1\left( \munder(t) \right)} \leq v^*\quad\text{for all }y\geq\widetilde{C}_1\chi_1(\munder(t))\text{ and all $t$ large enough}. \]
 	This gives the first part of the Claim~\ref{claim: hilflemma in 2nd moment proof}. 	
 	For the second part, recall that $T_y\leq \frac{1}{\underline{v}}y$ for all $y\geq\mathcal{N}'$. Furthermore, by the definition of $\varphi_t$ we have $\varphi_t(s)\geq \lfloor y\rfloor+1$ for all $s\in\left[ 0,t-T_{ \lfloor y\rfloor}-5\chi_1\left(\munder(t)\right) \right)$. Choosing $y:=\lfloor \widetilde{C}_15\chi_1\left( \munder(t) \right) \rfloor$ and $\widetilde{C}_2>\frac{\widetilde{C}_1}{\underline{v}}+1$, this implies that for $t$ large enough we get $\varphi_t(s)\geq \widetilde{C}_15\chi_1\left(\munder(t)\right)$ for all $s\in\big[ 0,t-\widetilde{C}_25\chi_1( \munder(t)  )\big)$.
 \end{proof}
	
\subsection{Proof of Theorem~\ref{th: log-distance breakpoint median}}
	Recall $\munder(t)=\sup\{x\in\R: u(t,x)\geq 1/2\}$ and $m(t)=\sup\{x\in\R: w(t,x)\geq 1/2\}$, where $u$ and $v$ are the solutions to \eqref{eq:PAM} and $w$ to \eqref{eq:KPP1}, respectively. We start with an amplification result.
	\begin{lemma}
	    \label{le:amplification}
	    For every $(p_k)_{k\in\N}$ fulfilling \eqref{eq:p_k} there exists $\Cl{const_ampl}=\Cr{const_ampl}((p_k))>1$ and $t_0>0$ such that  $\p$-a.s.,  for all $t\geq t_0,$
	    \begin{equation*}
	        \sup_{x\in\R} \probth_x^\xi\left( N(t,[x-1,x+1])\leq \Cr{const_ampl}^{t} \right) \leq \Cr{const_ampl}^{-t}.
	    \end{equation*}
	\end{lemma}
	\begin{proof}
	    For the proof it is enough to show the claim for binary branching with  rate $\xi(x)\equiv {\ei'}:=\ei(1-p_1)$ (which is the rate of first branching into more than one particle) by a straightforward coupling argument. Due to the spatial homogeneity of $\ei'$ it is enough to show 
	    \begin{equation*}
	        \prob_0^{\ei'}\left( N(t,[-1,1])\leq \Cr{const_ampl}^{t} \right) \leq \Cr{const_ampl}^{-t}
	    \end{equation*}
	    for all $t\geq t_0$, where $\prob_0^{\ei'}$ is the probability measure under which the branching Brownian motion starts with one particle in $0$ and has constant branching rate ${\ei'}>0$. Then for every $\varepsilon>0$ there exists $\delta=\delta(\e)>0$ such that 
	    \begin{equation}
	    \label{eq:many_part_delta}
	    \prob_0^{\ei'} \big( |N(t/3,[-\varepsilon,\varepsilon])||\geq \delta t \big) \geq 1-e^{-\delta t}
	    \end{equation}
	    for all $t$ large enough. 
	    Indeed, the probability that the initial particle does not leave the interval $[-\e t/2,\e t/2]$ before time $t/3$ is at least $1-e^{-c_1\e^2 t}$. If this happens, the particle produces more than $t{\ei'} /4$ offsprings with probability $1-e^{-c_2t}$ before time $t/3$, while each of these offsprings does not leave the interval $[-\e t,\e t]$ before time $t/3$ with probability  at least $1-e^{-c_1\e^2 t/2}$. Combining these observations and choosing $\delta(\varepsilon)>0$ small enough provides us with \eqref{eq:many_part_delta}. For a particle $\nu\in N(t/3)$ let $D^\nu(t/3+s)$ be the set of offsprings of $\nu$ in the interval $[X_{t/3}^\nu-1,X_{t/3}^\nu+1]$ at time $t/3+s$, $s\geq 0$. We will show the existence of some $p>0$ and $c>1$ such that 
	    \begin{equation}
	    \label{eq:breeding_enough}
	    \prob_0^{\ei'}\big( |D^\nu(t/3+s)| \geq c^{s}\big)\geq p
	    \end{equation}
	    for all $s$ large enough. 
	    To obtain \eqref{eq:breeding_enough}, let $r>0$ be such that 
	    \[ \inf_{y\in[-1,+1]} \Eprob_y^{\ei'} \big[ N(r,[-1,+1]) \big]=:\mu>1, \]
	    (the feasibility of such a choice of $r$ is a direct consequence of the Feynman-Kac formula). For $\nu\in D_\varepsilon(t/3)$ consider the following process under $\prob^{\ei'}_0$, conditionally on $X_{t/3}^\nu$:
	    \begin{itemize}
	    	\item the process starts with one particle at position $X_{t/3}^\nu$;
	    	\item between times $r(n-1)$ and $rn$, $n\in\N$, the process evolves as a branching Brownian motion with branching rate ${\ei'}$;
	    	\item at times $ rn$, $n\in\N$, particles outside of the interval $\big[X_{t/3}^\nu-1,X_{t/3}^\nu+1\big]$ are killed.
	    \end{itemize}
    	Using the Markov property, one readily observe that the number of particles of the latter process stochastically dominates the number of particles of a Galton-Watson process $(L_n)_{n\in\N}$ which starts with one particle and whose offspring distribution has expectation $\mu$. Then by \cite[Theorem~1, section~I.5]{athreya_ney_72}, the Galton-Watson process has positive probability to survive, i.e. $P(L_n>0\ \forall n\in\N)=:p_1>0$. Conditioned on surviving, there exists $c>1$ such that 
	    \begin{equation*}
	    P \big(L_k\geq c^k\, |\, L_n>0\ \forall n\in\N \big)\geq\frac{1}{2}\quad\text{for all }k\in\N.
	    \end{equation*}
	    One can see that for every $\nu\in N(t/3),$ inequality \eqref{eq:breeding_enough} holds true with the choice $p:=p_1/2$ for all $s \in r \cdot \N$. By a straightforward comparison argument, this extends to all $s\geq0.$ Therefore, we can now apply \eqref{eq:many_part_delta} and \eqref{eq:breeding_enough} in order to deduce
	    \begin{equation}
	    \label{eq:many_part}
	    \prob_0^{\ei'}\big( N(2t/3,[-\e t,\e t])\geq c^t \big) \geq 1-e^{c'(\e)t}. 
	    \end{equation}
	    Furthermore, we have 
	    \begin{equation*}
	    P_{\varepsilon t} \big(X_{t/3} \in [-1,1]\big) \geq c_3 t^{-1/2} e^{-3\e^2t/2} \geq \Big( \frac{1+c}{2} \Big)^{-t/3},
	    \end{equation*}
	    for all $t$ large enough 
	    $\e>0$ small enough and $c>1$ suitable, where for the last inequality we used that $\e$ does not depend on $c$. The latter inequality and a large deviation statement then gives for all $t\geq t_4\geq t_3$
	    \begin{equation}
	    \label{eq:many_part_in_interval}
	    \prob_0^{\ei'} \Big( N(t,[-1,1]) \geq \frac12 c^t\Big( \frac{1+c}{2} \Big)^{-t/3} \, \big \vert \,  N(2t/3,[-\e t,\e t])\geq c^t \Big) \geq 1-e^{-c_4 t}.
	    \end{equation}
	    Thus, for $t\geq t_0$, where $t_0$ is chosen large enough, by \eqref{eq:many_part} and \eqref{eq:many_part_in_interval}, in combination with $c>1,$ we infer the desired result. 
	\end{proof}
	
With the help of Lemmas~\ref{le: first moment leading particles inequ}, \ref{le:SecMomLeadingPart}, and \ref{le:amplification}, it is now possible to state a crucial result for the proof of Theorem~\ref{th: log-distance breakpoint median}.
	
\begin{proposition}
    \label{prop:LogDist}
	For every $q>0$, $F$ satisfying \eqref{eq:p_k} and $a>0$, there exist a constant $\Cr{constant theorem log dist}=\Cr{constant theorem log dist}(q,F)\in(0,\infty)$, 
	a $\p$-a.s.\ finite $C=C(t)=C(t,q,F,\xi)>0$ and a $\p$-a.s.\  finite random variable $\Cl[index_prob_cont_time]{time theorem log dist}=\Cr{time theorem log dist}(a,q,F,\xi)$  
	such that  for all  $t\geq\Cr{time theorem log dist},$ we have $C(t)\leq \Cr{constant theorem log dist}$ 
	and
	\begin{equation*}
	\probth^\xi_{\munder^{a}(t)-C\ln t}\left( N^\leq(t,0)\neq\emptyset \right)\geq 1-2t^{-q}.
	\end{equation*}
\end{proposition}
\begin{proof}
	For simplicity, we write $\munder(t):=\munder^{a}(t)$.	
    Without loss of generality, it is enough to show the claim for all $q>2(\gamma_1+\gamma_2)$, where $\gamma_1=\gamma_1(a)$ and $\gamma_2=\gamma_2(a)$ are defined in  Lemmas~\ref{le: first moment leading particles inequ} and \ref{le:SecMomLeadingPart}, respectively. Let further $\Cr{const_ampl}$ and $t_0$ be as in Lemma~\ref{le:amplification}, and $c_1$ be such that for $r:=c_1\ln t$ we have $\Cr{const_ampl}^{-r} = t^{-q}$.
    
	We claim that there exist 
	$\Cr{constant theorem log dist},C(t)$ and $\Cr{time theorem log dist}$ as above  such that $\munder(t-r)=\munder(t)-C(t)\ln t$, $C(t)\leq \Cr{constant theorem log dist}$ and the conclusions of Lemmas~\ref{le: first moment leading particles inequ} and \ref{le:SecMomLeadingPart} hold for all $t\geq \Cr{time theorem log dist}$. Indeed, writing $u(t,x)=\Eprob_x^\xi[N^\leq(t,0)]$, by the time and space perturbation Lemmas~\ref{le:time_perturb} and \ref{le:space_perturb}, defining $c_2:=\Cr{const_time_perturbation}\vee\Cr{const_space_perturbation}$, we deduce that
	\begin{align}
		    u(t-r,\munder(t)-\Cr{constant theorem log dist}\ln t)&\geq c_2^{-1}e^{\Cr{constant theorem log dist}\ln t/c_2} u(t-r,\munder(t)) \geq c_2^{-2} e^{\Cr{constant theorem log dist}\ln t/c_2-r/c_2} u(t,\munder(t)) 
		    \geq a, \label{eq:time_and_space_pert1}
	\end{align}
	where for the last inequality we choose $\Cr{constant theorem log dist}$ and $\Cr{time theorem log dist}=\Cr{time theorem log dist}(\varepsilon)$ large enough such that the last inequality holds for all $t\geq\Cr{time theorem log dist}$. As a consequence, we infer 
		\begin{equation}
		\label{eq:time_and_space_pert2}
		\munder(t-r)\geq\munder(t)-\Cr{constant theorem log dist}\ln t.
		\end{equation}
	By \eqref{eq:time_pert_ungenau}, we also get $u(t,\munder(t-r))\geq \Cr{const_time_perturbation}^{-1}e^{r/\Cr{const_time_perturbation}}u(t-r,\munder(t-r))\geq a$ for all $t$ large enough, i.e.\ 
	\[
		\munder(t-r)\leq\munder(t).
		\] 
	Thus, combining the two previous displays, we can find $C(t)\leq \Cr{constant theorem log dist}$ such that $\munder(t-r)=\munder(t)-C(t)\ln t$. Now let $x:=\munder(t)-C(t)\ln t=\munder(t-r)$. Conditioning on whether until time $r$ there are more or less than $\Cr{const_ampl}^r$ particles in $[x-1,x+1]$, we get for all $t\geq\Cr{time theorem log dist},$
	\begin{align*}
		\prob_x^\xi\left( N^\leq(t,0)\geq 1 \right) &\geq 1-\prob_x^\xi\big( N(r,[x-1,x+1])\leq \Cr{const_ampl}^r \big) - \sup_{y\in[x-1,x+1]}\Big( \prob_y^\xi\left( N^\leq(t-r,0)=0 \right) \Big)^{\Cr{const_ampl}^r} \\
		&\geq 1- \Cr{const_ampl}^{-r} - \sup_{y\in[x-1,x+1]}\Big( \prob^\xi_{y}\big( N^{\mathcal{L},a}_{t-r}=0  \big)\Big)^{\Cr{const_ampl}^r},
	\end{align*}
	using Lemma \ref{le:amplification} in the last inequality.
	Now using Cauchy-Schwarz as in \eqref{eq:CS}, in combination with Lemmas~\ref{le: first moment leading particles inequ} and \ref{le:SecMomLeadingPart}, we infer
	\begin{align*}
		\sup_{y\in[x-1,x+1]}\big( \prob^\xi_{y}( N^{\mathcal{L},a}_{t-r}=0 )\big)^{\Cr{const_ampl}^r} &\leq \left( 1-t^{-2\gamma_1-\gamma_2} \right)^{t^q} 
		\leq  t^{-q},
	\end{align*}
	where we adapt $\Cr{time theorem log dist}=\Cr{time theorem log dist}(a,q,\xi,q)$ such that  the last inequality holds for all $t\geq \Cr{time theorem log dist}$.
\end{proof}

	\begin{proof}[Proof of Theorem~\ref{th: log-distance breakpoint median}]
		1) We first prove the result under the additional assumption that $F$  fulfills \eqref{eq:p_k}. 
		Let $w^{\xi,F,w_0}$ be the solution to \eqref{eq:KPP1} with initial condition $w_0\in\mathcal{I}_{\text{F-KPP}},$ so in particular $0\leq w_0\leq \mathds{1}_{(-\infty,0]}$.  Because $F$ fulfills \eqref{eq:p_k}, by \eqref{eq:McKean} and the Markov property we infer
		\begin{align}
			w^{\xi,F,w_0}(t,x) &= \Eprob_x^\xi\Big[ 1-\prod_{\nu\in N(t)} \big(1-w_0(X_t^\nu)\big)  \Big] \geq \Eprob_x^\xi\Big[ 1-\prod_{\nu\in N(t)} \big(1-{w}_0(X_t^\nu) \big) ;N^\leq(t-s,0)\geq 1 \Big]\notag  \\
			&\geq \prob_x^\xi\big( N^\leq(t-s,0)\geq 1 \big) \cdot \inf_{y\leq0} w^{\ei,F,w_0}(s,y), \label{eq:sol_FKPP_markov}
		\end{align}  
		where $w=w^{\ei,F,{w}_0}$ solves the homogeneous equation $w_t=\frac12 w_{xx} + \ei\cdot F(w)$ with initial condition $w(0,\cdot)={w}_0$.  Then  we have  $w^{1,F,\widetilde{w}_0}=w^{\ei,F,w_0}(\frac{t}{\ei},\frac{x}{\sqrt{\ei}})$ with $\widetilde{w}_0(x):=w_0(x/\sqrt{\ei})$. Because $w^{\ei,F,w_0}(0,x)=0$ for $x>0$, conditions \cite[(8.1) and  (1.17)]{bramson1983convergence} are  fulfilled. 
		Together with \eqref{eq:KPP_initial} and \cite[Theorem~3, p.~141]{bramson1983convergence}, $w^{1,F,\widetilde{w}_0}$ (and thus also $w^{\ei,F,w_0}$) is a traveling wave solution, i.e., there exist $m^\ei(t)=\sqrt{2\ei}t + o(t)$ and some	 $g$ fulfilling  $\lim_{x\to-\infty}g(x)=1$ and $\lim_{x\to\infty}g(x)=0$ such that  
		\begin{equation}
		\label{eq:unif_conv}
		\sup_y \big| w^{\ei,F,w_0}(t,y + m^\ei(t)) - g(y) \big| \to 0, \quad t\to\infty.
		\end{equation}
		Now let $\e\in(0,1)$ and choose $\delta>0$ such that $\frac{\e}{1-\delta}\in(0,1)$. Then by \eqref{eq:unif_conv} we get 
		\[ \inf_{y\leq0} w^{\ei,F,w_0}(s,y)  = \inf_{y\leq -m^\ei(s)} w^{\ei,F,w_0}(s,y+m^\ei(s)) \geq 1-\delta\ \text{ for all }s\geq s_0(F,w_0,\delta,\ei),\]
		which, together with \eqref{eq:sol_FKPP_markov}, gives
		\begin{equation}
		\label{eq:comp_front_ini} 
		 m^{\xi,F,w_0,\e}(t) \geq m^{\xi,F,\mathds{1}_{(-\infty,0]},\frac{\e}{1-\delta}}(t-s_0)\quad\text{for all } t\geq s_0(F,w_0,\delta,\ei). 
		\end{equation}
		The inequality 
		\begin{equation*}
		m^{\xi,F,\mathds{1}_{(-\infty,0]},\frac{\e}{1-\delta}}(t-s_0) 
		{\geq}\ \munder^{\xi,\mathds{1}_{(-\infty,0]},\frac{\e}{1-\delta}}(t-s_0) - \Cr{constant theorem log dist}\ln(t), \quad\text{for all }t\geq \mathcal{T}_1(\xi,F,\e,\delta),
		\end{equation*}
		follows from Proposition~\ref{prop:LogDist}. 		
		By Corollary~\ref{cor:tightnessPAM}, for $C'>1$ from \eqref{eq:PAM_initial}   we get 
		\begin{align*}
		\munder^{\xi,\mathds{1}_{(-\infty,0]},\frac{\e}{1-\delta}}(t-s_0) &\geq \munder^{\xi,\mathds{1}_{(-\infty,0]},\frac{\e}{C'}}(t-s_0) - c_1(\e,\delta,C') \\
		&\geq \munder^{\xi,\mathds{1}_{(-\infty,0]},\frac{\e}{C}}(t)- c_2(\e,\delta, C', s_0)\quad\text{for all }t\geq \mathcal{T}_2(\xi,\e,\delta),
		\end{align*}
		where the second inequality can be obtained similarly to the argument in \eqref{eq:time_and_space_pert1} and \eqref{eq:time_and_space_pert2}. 
		Combining the above inequalities, we arrive at
		\begin{align*}
			m^{\xi,F,w_0,\e} (t) &\geq \munder^{\xi,\mathds{1}_{(-\infty,0]},\frac{\e}{C'}}(t) - C_1\ln(t) -  c_3 \\
			&= \munder^{\xi,C'\mathds{1}_{(-\infty,0]},\e}(t) - C_1\ln(t) -  c_3 \quad\text{for all }t\geq \mathcal{T}_3(\xi). 
		\end{align*}
		Now by \eqref{eq:PAM_initial} every $u_0\in\mathcal{I}_{\text{PAM}}$ is upper bounded by the function $C'\mathds{1}_{(-\infty,0]},$ and since we have $\munder^{\xi,\mathds{1}_{(-\infty,0]},\varepsilon}(t)\geq m^{\xi,F,w_0,\varepsilon}(t)$ for all $\varepsilon\in(0,1)$ and $w_0\in\mathcal{I}_{\text{F-KPP}}$, this finishes the proof for $F$ fulfilling \eqref{eq:p_k}.  
			    
	 2) Now let $F$ fulfill \eqref{eq:standard_condi} and $w_0$ be continuous. By a sandwich argument, it suffices to show that there exists some function $G$ fulfilling \eqref{eq:p_k}, such that $F(w)\geq G(w)$ for all $w\in[0,1]$. Indeed, by  Corollary~\ref{cor:mon_solutions} the  solutions $w^F$ and $w^G$ to  
	 \[ w^F_t-\frac{1}{2}w^F_{xx} - \xi(x)F(w^F)=0= w^G_t -\frac{1}{2}w^G_{xx} - \xi(x)G(w^G) \]
	 (which are classical by Proposition~\ref{prop:mckean1})
	 fulfill $w^F\geq w^G$. As a consequence, we infer that $m^{\xi,F,w_0,\e}(t)\geq m^{\xi,G,w_0,\e}(t)\geq \munder^{\xi,w_0,\varepsilon}(t)-C_1\ln(t)$, where the second  inequality is due to step 1). The claim for arbitrary $w_0\in\mathcal{I}_{\text{F-KPP}}$ is then true by an approximation argument of $w_0$ by continuous functions, that is if $F\geq G$, by Remark~\ref{rem:gen_approx} we have  $w^{w_0,F}\geq w^{w_0,G}$ and consequently $m^{\xi,w_0,F,\varepsilon}(t)\geq m^{\xi,w_0,G,\varepsilon}(t)$ for all $t\geq0$ and we can conclude.
	
		It remains to show that for every $F$ fulfilling \eqref{eq:standard_condi} there exists $G$ fulfilling \eqref{eq:p_k} such that $F(x)\geq G(x)$ for all $x\in[0,1]$. To do so, recall that \eqref{eq:standard_condi} implies that  there exists $M\in \N$ such that 
\begin{equation} \label{eq:FObs}	
1-F'(x)\leq \frac{M}{2}x \quad \text{ and } \quad F(1-x)\geq xM^{-1}
\quad \text{ for all }x\in[0,M^{-1}].
\end{equation} 
		Define $G_n(x):=\frac{1-x}{n} \big( 1-(1-x)^n \big)$, $x\in[0,1]$, $n\in\N$. Then each $G_n,$ $n\in\N,$ satisfies \eqref{eq:p_k} 
		with $p_1=1-n^{-1}$ and $p_{n+1}=n^{-1},$
		and our goal is to show that $G_n\leq F$ for all $x\in[0,1]$ and all $n$ large enough.
	
		We start with noting that $G_{n+1}\leq G_n$ as functions on  $x\in[0,1],$ for all $n\in\N,$ and that $G_n\downarrow 0$ uniformly as $n$ tends to infinity. Thus, since $F$ is continuous and $F >0$ on $(0,1)$ due to 
		\eqref{eq:standard_condi}, we only have to take care of the neighborhoods of $0$ and $1.$	 From \eqref{eq:FObs} we immediately get $G_M(x)\leq (1-x)M^{-1}\leq F(x)$ for all $x\in[1-M^{-1},1]$. To infer the desired inequality for $x\in[0,1-M^{-1}],$
	    Taylor expansion yields
		\begin{align*}
		(1-x)^M &\geq   1-Mx+\binom{M}{2}x^2-\binom{M}{3}x^3.
		\end{align*}
		Then for $M$ large enough and for all $x\in[0,M^{-1}]$ we get
		\begin{align*}
		G_M(x)&\leq (1-x)\Big( x-\frac{M-1}{2}x^2 + \frac{(M-1)(M-2)}{6}x^3 \Big) \leq x-\frac{M}{3}x^2 \\
		&\leq x-\frac{M}{4}x^2 = \int_0^x (1-Mt/2)\diff t \leq \int_0^x F'(t)\diff t = F(x),
		\end{align*}
		where the last inequality is due to \eqref{eq:FObs} again.
	  This finishes the proof.
	\end{proof}

\subsection{Remarks on $v_c$ and $v_0$}
\label{sec:vel}
By the same argument as in \cite[Lemma A.4]{drewitz_cerny}, one can show that a rich class of potentials $\xi$ satisfies \eqref{eq:VEL}, i.e.\ the inequality $v_c< v_0$. It is natural to ask if $v_c <v_0$ is always fulfilled in our setting or not; i.e., do there exist potentials which satisfy  \eqref{eq:SASS} but not the inequality $v_c < v_0$?

In order to answer this question, we will take advantage of the following result.
\begin{claim}
\label{claim:v_0_variational}
\begin{equation}
\label{eq:v_0_altern}
v_0=\inf_{\eta\leq 0} \frac{\eta-\text{\emph{\es}}}{L(\eta)}.
\end{equation}
\end{claim}
\begin{proof} 
As shown in \cite[p.\ 514ff.]{Fr-85}, the function 
\[
I:\, (0,\infty)\ni y\mapsto \sup_{\eta\leq-\es}\big( y\eta-L(\eta+\es) \big)
\]
is strictly decreasing, finite, convex, fulfills $\lim_{y\downarrow0}I(y)=+\infty$ and $\lim_{y\uparrow\infty}I(y)=-\infty$, there exists a unique $v^*>0$ such that $I(1/v^*)=0$, and one has
\[ v^*=\inf_{\eta\leq 0} \frac{\eta-\es}{L(\eta)}. \]
We now show $v^*=v_0$. 
For this purpose, let $w$ and $u$ be solutions to \eqref{eq:KPP1} and \eqref{eq:PAM}, respectively, both with initial condition $\mathds{1}_{[-1,0]}$. Then $w\leq u$ and thus by \cite[Lemma~7.6.3]{Fr-85} and Proposition~\ref{prop:lyapunov} we have for all $v>v^*$ that $\p$-a.s.,
\begin{align*}
\Lambda(v)=\lim_{t\to\infty}\frac{1}{t} \ln u(t,vt) &\geq \liminf_{t\to\infty} \frac{1}{t} \ln w(t,vt)\geq -v I(1/v). 
\end{align*}
Since $\Lambda$ and $I$ are continuous, passing to the limit $v\downarrow v^*$ we deduce that $\Lambda(v^*)\geq -v^*I(1/v^*)=0$. Furthermore,  $\Lambda(v)<0$ for all $v>v_0$ and thus we infer
$v_0 \geq v^*.$

To get the converse inequality, we use that for all $v>0$ and all $\eta\leq0$ we have
\begin{align*}
u(t,vt) &= E_{vt} \big[ e^{\int_0^t(\zeta(B_s)+\es)\diff s} ; B_t\in[-1,0] \big] = E_{vt} \big[ e^{\int_0^t(\zeta(B_s)+\es)\diff s}; B_t\in[-1,0], H_0\leq t\big] \\
&\leq e^{(\es-\eta)t} E_{vt} \big[ e^{\int_0^{H_0}(\zeta(B_s)+\eta)\diff s} \big]. 
\end{align*}
In combination with \eqref{eq:LLNEmpLogMG} this yields that for all $v>0,$
\begin{align*}
\Lambda(v)=\lim_{t\to\infty} \frac{1}{t}\ln u(t,vt) &\leq \inf_{\eta\leq 0} \big ( (\es-\eta) + vL(\eta) \big ) =-v\sup_{\eta\leq-\es}\big ( \frac{\eta}{v} - L(\eta+\es) \big ) = -vI(1/v).
\end{align*}
But then $\Lambda(v^*)\leq -v^*I(1/v^*)=0$ and thus we must have $v^*\geq v_0$.
\end{proof}

We now formulate our main result on the relation between $v_0$ and $v_c.$ It implies that our results do not apply for all potentials fulfilling \eqref{eq:POT}, \eqref{eq:STAT} and \eqref{eq:MIX}.
\begin{proposition}
	\label{le:v_c>v_0}
	There exist potentials $\xi$ fulfilling \eqref{eq:POT}, \eqref{eq:STAT} and \eqref{eq:MIX}, and such that  $v_c> v_0;$ i.e., condition \eqref{eq:VEL} is violated.
\end{proposition}
\begin{proof}
	Recalling \eqref{eq:v_0_altern}, the definition $v_c=\frac{1}{L'(0-)}$ from Lemma~\ref{le:expectedLogMGfct} d), it is sufficient to show $L(0) + \es\cdot L'(0-) <0$, which means
	\begin{equation}
	\label{eq:proof_v1}
	\E \Big[ \ln E_1\big[ e^{\int_0^{H_0} \zeta(B_s)\diff s} \big] \Big] + \es\cdot \E \Big[ \frac{E_1\big[H_0e^{\int_0^{H_0} \zeta(B_s)\diff s}\big]}{E_1\big[ e^{\int_0^{H_0} \zeta(B_s)\diff s}\big]} \Big]<0.
	\end{equation}
	To establish the latter, let $\widetilde{\omega}$ be a one-dimensional Poisson point process with  intensity one. In a slight abuse of notation,  $\widetilde{\omega}=(\widetilde{\omega}^i)_i$ can be seen as a mapping from $\Omega$ into the set of all locally finite point configurations; i.e., $\widetilde{\omega}=(\widetilde{\omega}^i)_i$ can be interpreted as a  random set of countably many points in $\R$, satisfying $|\{i:\widetilde{\omega}^i\in B\}|<\infty$ for every bounded Borel set and $\widetilde{\omega}^i\neq \widetilde{\omega}^j$ for all $i\neq j$.	See \cite{resnick} for further details. 
	Now denote by  $\omega=(\omega^i)_i$ be the point process that is obtained from $\widetilde \omega$ by deleting {\em simultaneously} all points in $\widetilde{\omega}$ which have distance $1$ or less to their nearest neighbor in $\widetilde \omega.$ 
	(see \cite[p.\ 47]{matern1986} for details). Let $\varphi(x)$ be a mollifier with support $[-1/2,1/2]$, non-decreasing for $x\leq0$, non-increasing for $x\geq0$ with $\varphi(0)=1$ and let $\varphi^{(\varepsilon)}(x):=\varphi(x/\varepsilon)$, $\varepsilon>0$. Finally, for $\varepsilon\in(0,1)$ and $a>0,$ define the potential $\zeta(x)=\zeta^{(\varepsilon,a)}(x):=-a + a\sum_{i}\varphi^{(\varepsilon)}(x-\omega^i)$. One can easily check that $\zeta$ is the corresponding
	shifted potential as in \eqref{eq:zetaBd} of some $\xi$ 
    fulfilling \eqref{eq:POT}, \eqref{eq:STAT} and \eqref{eq:MIX}, i.e.\ $a=\es-\ei$. We will choose $a>\ln 2$, $\ei\in(0,\frac{a}{8})$  and $\varepsilon(a)>0$  suitably at the end of the proof. Let us now consider both summands in \eqref{eq:proof_v1} separately.

	1) We observe that $\p$-a.s., $\zeta^{(\varepsilon,a)}\downarrow-a$ as $\varepsilon\downarrow0$ for all $x\in\R,$ as well as by \cite[(2.0.1), p.~204]{handbook_brownian_motion} 
	\[-\sqrt{2a} = \ln E_1\big[ e^{-aH_0} \big]\leq \ln E_1\big[ e^{\int_0^{H_0} \zeta^{(\varepsilon,a)}(B_s)\diff s} \big] \leq 0\]
	for all $\varepsilon\in(0,1).$ Thus, by dominated convergence, for all $a>0$ there exists $\varepsilon_1=\varepsilon_1(a)>0$, such that 
	\begin{equation}
	\label{eq:proof_v2}
	\E\Big[ \ln E_1\big[ e^{\int_0^{H_0} \zeta^{(\varepsilon_1,a)}(B_s)\diff s} \big] \Big] \leq -\frac{3}{4} \sqrt{2a}.
	\end{equation}
	
	2) To bound the second summand in \eqref{eq:proof_v1}, we lower bound its denominator by
\begin{equation} \label{eq:denBd}
 E_1\big[ e^{\int_0^{H_0} \zeta(B_s)\diff s}\big] \geq E_1\big[ e^{-a H_0} \big] = e^{-\sqrt{2a}}.
 \end{equation}
	For the numerator, define $\mathbb{J}$ to be the set of possible point configurations of the process $(\omega^i)_i$. Let us  first check that for all $a>0$ there exists $\varepsilon=\varepsilon(a)>0$, such that 
	\begin{equation}
	\label{eq:proof_v3}
	\sup_{(\omega^i)_i\in\mathbb{J}}E_1\big[H_0e^{\int_0^{H_0} ( -a + a\sum_{i}\varphi^{(\varepsilon)}(B_s-\omega^i) )\diff s}\big]<\infty.
	\end{equation}
	Indeed,  letting  $g^{\varepsilon,(\omega^i)_{i}}(x):=\sum_{i}\varphi^{(\varepsilon)}(x-\omega^i)$, we have
	\begin{align}\label{eq:proof_v4}
	\begin{split}
	E_1\big[H_0e^{\int_0^{H_0} ( -a + a\cdot  g^{\varepsilon,(\omega^i)_{i}}(B_s) )\diff s}\big] &= \sum_{n=0}^\infty E_1\big[H_0e^{\int_0^{H_0} ( -a + a\cdot g^{\varepsilon,(\omega^i)_{i}}(B_s) )\diff s};H_0\in[n,n+1)\big] \\
	&\leq \sum_{n=0}^\infty (n+1) E_1\big[e^{\int_0^{n} ( -a + a\cdot g^{\varepsilon,(\omega^i)_{i}}(B_s))\diff s}\big].
	\end{split}
	\end{align}
	Note that by the property of all point configurations in $\mathbb{J}$ to have points with mutual distance at least one, we have 
	\[ 
	\sup_{(\omega^i)_i\in\mathbb{J}} \sup_{x\in\R} E_x\Big[ a\int_0^1 g^{\varepsilon,(\omega^i)_{i}}(B_s)\diff s \Big]
	 \leq  a\int_0^1 E_0\big[ \mathds{1}_{A_\varepsilon} (B_s)\big]\diff s 	\leq \frac{1}{2} \]
	for all $\varepsilon(a)>0$ small enough, where $A_\varepsilon:=\bigcup_{i\in\Z}[-\varepsilon/2+i,\varepsilon/2+i]$. Using 	Kasminskii's lemma (cf.\ e.g.\ \cite[Lemma~1.2.1]{sznitman98}) we infer
	$\sup_{x\in\R}E_x \big[ e^{a\int_0^1 g^{\varepsilon,(\omega^i)_{i}}(B_s)\diff s} \big]\leq 2$. An $(n-1)$-fold application of the Markov property at times $1, \ldots, n-1$ supplies us with  $\sup_{x\in\R}E_x \big[ e^{a\int_0^n  g^{\varepsilon,(\omega^i)_{i}}(B_s)\diff s} \big]\leq 2^n$ for all $n\in\N$ and all $(\omega^i)_i\in\mathbb{J}$. Plugging this into \eqref{eq:proof_v4} we infer 
	\begin{align*}
	\sup_{(\omega^i)_i\in\mathbb{J}} E_1\Big[H_0e^{\int_0^{H_0} ( -a + a\cdot g^{\varepsilon,(\omega^i)_{i}}(B_s) )\diff s}\Big] &\leq  \sum_{n=0}^{\infty}(n+1)e^{-na}2^n,
	\end{align*}
	so the right-hand side in \eqref{eq:proof_v4} is finite, and \eqref{eq:proof_v3} holds true for all $a>\ln 2$ and $\varepsilon(a)$ small enough as well. Since $g^{\varepsilon,(\omega^i)_{i}}$ decreases $\p$-a.s.\ to $0$ monotonically as $\varepsilon\downarrow0$, we infer
	\begin{equation} \label{eq:numBd}
	\lim_{\varepsilon\downarrow0}\E\Big[E_1\big[H_0e^{\int_0^{H_0} \zeta^{(\varepsilon,a)}(B_s)\diff s}\big]\Big]=E_1\big[H_0 e^{-aH_0}\big]=-\frac{\diff}{\diff a}E_1\big[e^{-aH_0}\big]=-\frac{\diff}{\diff a}(e^{-\sqrt{2a}})=\frac{1}{\sqrt{2a}}e^{-\sqrt{2a}},
	\end{equation}
	 using \cite[(2.0.1), p.~204]{handbook_brownian_motion} in the third equality.
	Thus, combining \eqref{eq:denBd} and \eqref{eq:numBd} we infer that there exists $\varepsilon_2(a)>0$ such that the second summand on the left-hand side of \eqref{eq:proof_v1} is upper bounded by $\es\cdot \frac{4/3}{\sqrt{2a}}=(a+\ei)\cdot \frac{4/3}{\sqrt{2a}}.$ 
	Using this in combination with \eqref{eq:proof_v2}, we infer that  for all $a>\ln 2$ (which is sufficient for \eqref{eq:proof_v3}) and  $\varepsilon \in (0, \varepsilon_1(a)\wedge\varepsilon_2(a))$, choosing $\ei \in (0,\frac{a}{8}),$ we  get that the left-hand side in \eqref{eq:proof_v1} is upper bounded by 
	$-\frac{3}{4}\sqrt{2a} + (a+\ei)\cdot \frac{4/3}{\sqrt{2a}}<0$.
\end{proof}

\appendix

\section{Appendix}
\label{sec:appendix}
\subsection{Properties of the logarithmic moment generating functions}
\begin{lemma}
	\label{le:log mom fct and derivatives}
	We recall that $P_x^{\zeta,\eta}$ has been defined  in \eqref{eq:P_x}. 
	\begin{enumerate}
		\item $L,$ $L_x^\zeta$, and $\overline{L}_x^\zeta,$ for $x\in\R$, (see \eqref{eq:LogMG_x} to \eqref{eq:expectedLogMG} for their definitions) are infinitely differentiable on $(-\infty,0).$ Furthermore, for all $\eta<0$ we have
		\begin{align}
		\big( L_x^\zeta \big)'(\eta) &= \frac{E_x\Big[ e^{\int_0^{H_{\lceil x\rceil-1}}(\zeta(B_r)+\eta)\diff r}H_{\lceil x\rceil-1} \Big]}{E_x\Big[ e^{\int_0^{H_{\lceil x\rceil-1}}(\zeta(B_r)+\eta)\diff r} \Big]} = E_x^{\zeta,\eta}[ \tau_{\lceil x\rceil-1} ],\quad x\in\R,\label{eq:L_x'} \\
		\big( \overline{L}_x^\zeta \big)'(\eta) &= \frac{1}{x}E_x^{\zeta,\eta}\big[ H_0 \big],\quad x>0,\label{eq:Lbar_x'} \\
		L'(\eta) &=\E\Big[ \frac{E_1\big[ e^{\int_0^{H_0}(\zeta(B_r)+\eta)\diff r}H_0 \big]}{E_1\big[ e^{\int_0^{H_0}(\zeta(B_r)+\eta)\diff r} \big]} \Big] = \E\big[ E_1^{\zeta,\eta}[ H_0 ] \big],\label{eq:L'} 
		\intertext{and}	
		\big( L_x^\zeta \big)''(\eta) &=\frac{E_x\big[ e^{\int_0^{H_{\lceil x\rceil-1}}(\zeta(B_r)+\eta)\diff r}H_{\lceil x\rceil-1}^2 \big]}{E_x\big[ e^{\int_0^{H_{\lceil x\rceil-1}}(\zeta(B_r)+\eta)\diff r} \big]} -\Bigg(  \frac{E_x\big[ e^{\int_0^{H_{\lceil x\rceil-1}}(\zeta(B_r)+\eta)\diff r}H_{\lceil x\rceil-1} \big]}{E_x\big[ e^{\int_0^{H_{\lceil x\rceil-1}}(\zeta(B_r)+\eta)\diff r} \big]} \Bigg)^2 \notag \\
		&=E_x^{\zeta,\eta}\big[ \tau_{\lceil x\rceil-1}^2 \big] - \big( E_x^{\zeta,\eta}[ \tau_{\lceil x\rceil-1} ]\big)^2 =\text{\emph{Var}}_x^{\zeta,\eta}(\tau_{\lceil x\rceil-1}) >0 ,\quad x\in\R,\label{eq:L_x''} \\
		\big( \overline{L}_x^\zeta \big)''(\eta) &=\frac{1}{x}\text{\emph{Var}}_x^{\zeta,\eta}(H_0),\quad x>0,\label{eq:Lbar_x''}\\
		L''(\eta) &= \E\Bigg[ \frac{E_1\big[ e^{\int_0^{H_0}(\zeta(B_r)+\eta)\diff r}H_0^2 \big]}{E_1\big[ e^{\int_0^{H_0}(\zeta(B_r)+\eta)\diff r} \big]} -\Big(  \frac{E_1\big[ e^{\int_0^{H_0}(\zeta(B_r)+\eta)\diff r}H_0 \big]}{E_1\big[ e^{\int_0^{H_0}(\zeta(B_r)+\eta)\diff r} \big]} \Big)^2 \Bigg] \notag \\
		&=\E\Big[ E_1^{\zeta,\eta} [ H_0^2 ] - \big( E_1^{\zeta,\eta} [ H_0] \big)^2 \Big] = \E\big[ \text{\emph{Var}}_1^{\zeta,\eta}( H_0 ) \big] >0.\label{eq:L''}
		\end{align} 
		\item For each compact interval $\triangle\subset(-\infty,0)$ there exists a constant $\Cl{constDerLMG}=\Cr{constDerLMG}(\triangle)>0$, such that the following inequalities hold $\p$-a.s.:
		\begin{align*}
		\hspace{-4mm}  -\Cr{constDerLMG} \leq \inf_{\eta\in\triangle,x\geq 1} \big\{ L_{\lfloor x\rfloor}^\zeta(\eta),\overline{L}_x^\zeta(\eta),L(\eta)  \big\} &\leq \sup_{\eta\in\triangle,x\geq 1}  \big\{ L_{\lfloor x\rfloor}^\zeta(\eta),\overline{L}_x^\zeta(\eta),L(\eta) \big\} \leq -\Cr{constDerLMG}^{-1},\\
		\hspace{-4mm}  \Cr{constDerLMG}^{-1}\leq \inf_{\eta\in\triangle,x\geq 1} \big\{ (L_{\lfloor x\rfloor}^\zeta)'(\eta),(\overline{L}_x^\zeta)'(\eta),L'(\eta) \big\}  &\leq \sup_{\eta\in\triangle,x\geq 1}  \big\{ (L_{\lfloor x\rfloor}^\zeta)'(\eta),(\overline{L}_x^\zeta)'(\eta),L'(\eta) \big\}  \leq \Cr{constDerLMG},\\
		\hspace{-4mm}  \Cr{constDerLMG}^{-1}\leq \inf_{\eta\in\triangle,x\geq 1} \big\{ (L_{\lfloor x\rfloor}^\zeta)''(\eta),(\overline{L}_x^\zeta)''(\eta),L''(\eta) \big\}  &\leq \sup_{\eta\in\triangle,x\geq 1}  \big\{ (L_{\lfloor x\rfloor}^\zeta)''(\eta),(\overline{L}_x^\zeta)''(\eta),L''(\eta) \big\}  \leq \Cr{constDerLMG}.
		\end{align*} 
	\end{enumerate} 
\end{lemma}
\begin{proof}
	$(a)$ Due to the convexity of the exponential function we have $\big|\frac{e^{h x}-1}{h}\big|\leq xe^{hx}\vee 1$ for all $x\geq0$ and $h\in\mathbb{R}$. If we choose $h_0:=\frac{|\eta|}{2}$, then since $\zeta\leq0$, we have that for all  $|h|\leq h_0,$
	\begin{equation}
	\label{eq:L_x'proof1}
	\Big| \frac{1}{h}e^{\int_0^{H_{y}}(\zeta(B_r)+\eta)\diff r}\big( e^{hH_{y}}-1 \big) \Big| \leq e^{\eta H_{y}/2}(H_y\vee 1).
	\end{equation}
	Due to $H_ye^{\eta H_y}\leq \frac{1}{|\eta|}$ for all $H_y\geq0$ and $\eta<0,$ as well as $\lim_{h\to0}\frac{e^{hH_y}-1}{h}=H_ye^{hH_y}$ for all $H_y\geq0$,  dominated convergence yields for all $\eta<0$ that
	\begin{equation}
	\frac{\diff}{\diff \eta} E_x\big[ e^{\int_0^{H_{\lceil x\rceil-1}}(\zeta(B_r)+\eta)\diff r} \big] = E_x\big[ e^{\int_0^{H_{\lceil x\rceil-1}}(\zeta(B_r)+\eta)\diff r}H_{\lceil x\rceil-1} \big]. \label{eq:L_x'proof2} 
	\end{equation}
	Then \eqref{eq:L_x'} is a consequence of the chain rule and the fact that the expectation on the left-hand side in \eqref{eq:L_x'proof2} is positive. Then \eqref{eq:Lbar_x'} follows from linearity of the derivative. To show \eqref{eq:L'}, we have to apply dominated convergence once more. This time, we additionally need that the expectation on the left-hand side in \eqref{eq:L_x'proof2} for $x=1$ is strictly bounded from below by the constant $E_1[e^{-(\es-\ei-\eta)H_0}]=e^{-\sqrt{2(\es-\ei-\eta)}}>0$ due to \eqref{eq:POT} and \cite[(2.0.1), p.~204]{handbook_brownian_motion}. 
	Using the mean value theorem and \eqref{eq:L_x'proof1} entail the existence of $c_1>0$ such that 
	\begin{equation*}
	\sup_{|h|\leq |\eta|/2}\Big| \frac{1}{h}\big(  L_1^\zeta(\eta+h) - L_1^\zeta(\eta) \big) \Big| \leq \sup_{|h|\leq |\eta|/2} (L_1^\zeta)'(\eta+h) \leq c_1 E_1\big[e^{\eta H_{0}/2} (H_0\vee 1) \big]\leq c_1\big(\frac{2}{|\eta|}\vee 1\big).
	\end{equation*}
 Using \eqref{eq:L_x'} for $x=1$ and dominated convergence, we arrive at \eqref{eq:L'}.
	
	By induction and similar arguments as above, it follows that for all $n\in\mathbb{N}$, the function $(-\infty,0)\ni\eta\mapsto E_x\big[(H_{y})^n e^{\int_0^{H_{y}}(\zeta(B_r)+\eta)\diff r} \big]$ is positive $\p$-a.s.\  and  differentiable. Due to \eqref{eq:POT}, one can interchange expectation and differentiation in $\eta,$ yielding \eqref{eq:L_x''} and $\eqref{eq:Lbar_x''}$. The first equality in \eqref{eq:L''} is then again a consequence of dominated convergence. The strict inequalities in \eqref{eq:L_x''} and  \eqref{eq:L''} are due to the fact that under $P_1,$ and thus also $\p$-a.s.\  under $P_1^{\zeta,\eta}$, the random variable $H_0$ is non-degenerate.
	
	To show (b), observe that the function $\zeta\mapsto E_x\big[e^{\int_0^{H_{\lceil x\rceil -1}}(\zeta(B_s)+\eta)\diff s}\big]$ is nondecreasing. Consequently, using the notation $\triangle=[\undereta,\overeta]$, we have
	\begin{align*}
	-\infty &< e^{-\sqrt{2(\es-\ei+|\undereta|})} \leq E_x\big[ e^{(\ei-\es+\undereta)H_{\lceil x\rceil -1}}  \big]  \leq \inf_{\eta\in\triangle} \essinf_{\zeta}  E_x\big[e^{\int_0^{H_{\lceil x\rceil -1}}(\zeta(B_s)+\eta)\diff s}\big] \\
	&\leq \sup_{\eta\in\triangle} \esssup_{\zeta}  E_x\big[e^{\int_0^{H_{\lceil x\rceil -1}}(\zeta(B_s)+\eta)\diff s}\big] \leq E_x\big[ e^{\overeta H_{\lceil x\rceil -1}}  \big] = e^{(x-\lceil x\rceil+1)\sqrt{2|\overeta|}},
	\end{align*}
	where we used \cite[(2.0.1), p.~204]{handbook_brownian_motion} for the second inequality and last equality. 
	Due to the inequality $e^{-xy}x\leq\frac{2}{y}e^{-xy/2}$ for all $x\geq 0$ and $y>0$, these estimates can be used to derive similar bounds for $E_x\big[ e^{\int_0^{H_{\lceil x\rceil-1}}(\zeta(B_r)+\eta)\diff r}H_{\lceil x\rceil-1}^k \big]$, $k=1,2$. Thus, estimating the numerator and the denominator of the corresponding expressions in \eqref{eq:L_x'} to \eqref{eq:L''}, we can conclude. 
\end{proof} 

\subsection{Exponential mixing}
\begin{lemma}
	\label{le:expMixLi}
	Let $\F^k$ be as defined in \eqref{eq:MIX}, i.e.\ $\F^k=\sigma(\xi(x):x\geq k)$, and let $\triangle \subset (-\infty,0)$ be a compact interval. Then there exists a constant $C_\triangle>0$ such that  $\p$-a.s.,  for all $i,j\in\Z$ with $i< j$, and all $\eta\in\triangle,$ 
	\begin{align}
	\big|\E\big[ L_i^\zeta(\eta)|\F^j \big]-L(\eta)\big|&\leq C_\triangle \cdot \Big(\psi\big({\frac{j-i}{2}}\big)+e^{-(j-i)/C_\triangle}\Big), \label{eq:expMix1}\\
	0\leq \Big(  \esssup_{\xi(k): k\geq j} L_{i}^\zeta(\eta)   \Big) - L_{i}^\zeta(\eta) &\leq C_\triangle \cdot \Big(\psi\big({\frac{j-i}{2}}\big)+e^{-(j-i)/C_\triangle}\Big) \label{eq:expMix2},
	\intertext{as well as}
	\Big|\E\big[ (L_i^\zeta)'(\eta) |\F^j \big]-L'(\eta) \Big|&\leq C_\triangle \cdot \Big(\psi\big({\frac{j-i}{2}}\big)+e^{-(j-i)/C_\triangle}\Big),\label{eq:expMix3} \\
	0\leq \Big(   \esssup_{\xi(k): k\geq j} \big(L_{i}^\zeta\big)'(\eta) \Big)- \big(L_{i}^\zeta\big)'(\eta) &\leq C_\triangle\cdot \Big(\psi\big({\frac{j-i}{2}}\big)+e^{-(j-i)/C_\triangle}\Big) \label{eq:expMix4}.
	\end{align}
\end{lemma}
\begin{proof}
	By translation invariance, it is enough to prove \eqref{eq:expMix1} for $i=0$ and $j\geq2$ (the case $j=1$ follows immediately from the uniform boundedness of $L^\zeta_0$ and $L$ on $\triangle$ due to Lemma~\ref{le:log mom fct and derivatives}). To show \eqref{eq:expMix1}, let $\eta\in\triangle$ and write $L_0^\zeta(\eta)=\ln(A+B)$ with
	\begin{align*}
	A&=A(\eta):= E_0\Big[ e^{\int_0^{H_{-1}}(\zeta(B_s)+\eta)\diff s}; \sup_{0\leq s\leq H_{-1}} B_s<\lceil j/2\rceil \Big] 
	\intertext{and}
	B&=B(\eta):=E_0\Big[ e^{\int_0^{H_{-1}}(\zeta(B_s)+\eta)\diff s}; \sup_{0\leq s\leq H_{-1}} B_s\geq\lceil j/2\rceil \Big].
	\end{align*}
	Then $A\leq E_0[e^{H_{-1}\eta}]	\leq c_{\triangle,1}<1,$ and,  using \eqref{eq:zetaBd}, at the same time we have 
	\begin{align*}
	A&\geq E_0\Big[ e^{-(\es-\ei+|\eta|)H_{-1}};\sup_{0\leq s\leq H_{-1}}B_s<\lceil j/2\rceil\Big] \geq c_{\triangle,2}>0 \quad\text{for all }j\geq2.
	\end{align*}
	 To bound $B$, we condition on $H_{-1}$ to happen before or after time $j$ and use the reflection principle for Brownian motion and the tail estimate from \cite[Lemma~1.1]{bovier_gaussian_proc_on_trees} to infer for all $j\geq 2$
	\begin{align*}
	0\leq B&\leq    P_0\Big( \sup_{0\leq s\leq j} B_s\geq \lceil j/2\rceil\Big) + e^{\eta j} =2P_0(B_j\geq \lceil j/2\rceil) + e^{\eta j}
	 \leq 2e^{-j(1/8\wedge |\eta|)}.
	\end{align*}
	As $\ln(1+x)\leq x,$ the above implies that for all $j\geq 2$ 
	\begin{equation*}
	\ln(A) \leq L_0^\zeta(\eta) = \ln(A) + \ln\Big(1+\frac{B}{A}\Big) \leq \ln(A) + c_{\triangle,3}e^{-j/c_{\triangle,3}}.
	\end{equation*}
	Since  $L_0^\zeta$ is continuous on $\triangle$, the latter display $\p$-a.s.\ holds uniformly for all $\eta\in\triangle.$
	Now $\ln(A)$ is $\F_{\lceil j/2\rceil}$-measurable and bounded, so by \eqref{eq:MIX} 
	\begin{align}
	\sup_{\eta\in\triangle} \left|\E[ L_0^\zeta(\eta)|\F^j ]-L(\eta)\right| &\leq \sup_{\eta\in\triangle} \left| \E[ \ln(A)-\E[\ln(A)] | \F^j  ]  \right| + 2c_{\triangle,3}e^{-j/c_{\triangle,3}} \notag \\
	&\leq C_\triangle\cdot\big(\psi(j/2)+e^{-j/C_\triangle}\big) \label{eq:L_0_mixing_proof}.
	\end{align}    
	The proof of \eqref{eq:expMix3} is similar. Indeed, using the same notation  we have $(L_0^\zeta)'=\frac{A'}{A+B}+ \frac{B'}{A+B}$. Then by $A+B\geq c_{\triangle,2}$ and $e^{\eta H_{-1}}H_{-1}\leq \frac{2}{|\eta|}e^{\eta H_{-1}/2}$, we can use above calculation to conclude that $\frac{B'}{A+B}$ decays exponentially to $0$ as $j\to\infty$. Further, $\frac{A'}{A}-\frac{BA'}{A(A+B)}=  \frac{A'}{A+B}\leq  \frac{A'}{A}$, $F_{\lceil j/2\rceil}$-measurability of $\frac{A'}{A}$ and  above estimates    give a similar bound as in \eqref{eq:L_0_mixing_proof} for $(L_0^\zeta)'(\eta)$ and $L'(\eta)$ instead of $(L_0^\zeta)(\eta)$ and $L(\eta)$. 
	Finally,  \eqref{eq:expMix2} and \eqref{eq:expMix4} follow by the same arguments as \eqref{eq:expMix1} and \eqref{eq:expMix3}.
\end{proof}

\subsection{Properties of the Lyapunov exponent}
\begin{proposition}
	\label{prop:lyapunov}
	Assume \eqref{eq:POT} and \eqref{eq:STAT}
	and let $u^{(x,\delta)}$ be a solution to \eqref{eq:PAM} with initial condition $\mathds{1}_{[x-\delta,x+\delta]}$. 
	Then the limit $\Lambda(v):=\lim_{t\to\infty} \frac{1}{t} \ln u^{(x,\delta)} (t,vt)$, $v\in\R$, exists $\p$-a.s., is non-random and independent of $x\in\R$ and $\delta>0$. The function $[0,\infty)\ni v\mapsto \Lambda(v)$ is concave, $\Lambda(0)=\text{\emph{$\es$}}$  and $\lim_{v\to\infty} \frac{\Lambda(v)}{v}=-\infty$. In particular, there exists a unique $v_0>0$ such that $\Lambda(v_0)=0$. 
\end{proposition}
\begin{proof}
	Let $x=0$, $\lambda\in(0,1)$ and $v_1,v_2\in\R$, and  set $v:=\lambda v_1 + (1-\lambda) v_2$ and $A:=[(1-\lambda)v_2t-\delta,(1-\lambda)v_2t + \delta]$.
	By \eqref{eq:feynman_kac}, the solution to \eqref{eq:PAM}, admits the Feynman-Kac representation $u^{(x,\delta)}(t,vt)=E_{vt}\big[ e^{\int_0^t \xi(B_s)\diff s}; B_t\in[x-\delta,x+\delta] \big]$. 
	 Then by the Markov property, for all $y\in[vt-\delta,vt+\delta]$ we have
	\begin{equation}
	\label{eq:concave_Lyap}
	\begin{split}
	\ln E_{y} \big[ e^{\int_0^t \xi(B_s)\diff s}; B_t\in[-\delta,\delta] \big] &\geq \ln E_{y} \big[ e^{\int_0^{\lambda t} \xi(B_s)\diff s}; B_{\lambda t}\in A \big] \\
	&\quad + \inf_{z\in A} \ln E_z \big[ e^{\int_0^{(1-\lambda)t} \xi(B_s)\diff s}; B_{(1-\lambda)t}\in[-\delta,\delta] \big].
	\end{split}
	\end{equation}
	 Defining $\overline{\mu}^{\delta}_{s,t}(v):= \inf_{y\in[vt-\delta,vt+\delta]} \ln E_y\big[ e^{\int_0^{t-s}\xi(B_r)\diff r}; B_{t-s}\in[vs-\delta,vs+\delta] \big]$, $s<t$, by the same argument as in the last display one can see that $\overline{\mu}^{\delta}_{s,t}(v)\geq \overline{\mu}^{\delta}_{s,u}(v) + \overline{\mu}^{\delta}_{u,t}(v)$ for all $s<u<t$. Furthermore for the $h$-lateral shift $\theta_h$ on $\xi$ (i.e.\ $\xi(\cdot)\circ\theta_h=\xi(\cdot+h)$) we have $\overline{\mu}^{\delta}_{s,t}(v)\circ\theta_h=\overline{\mu}^{\delta}_{s+\frac{h}{v},t+\frac{h}{v}}(v)$ for every $h\in\R$, all $v\neq 0$ and $s<t$, and by \eqref{eq:POT} and a simple Gaussian computation there exists $K_{\delta,v} \in (0,\infty)$ such that  $\overline{\mu}^{\delta}_{0,t}(v)\geq -K_{\delta,v}\cdot t$ for all $t\geq 1$. By \eqref{eq:POT} we also have $\overline{\mu}_{0,t}^\delta(v)\leq \es\cdot t$ and thus $\overline{\mu}_{0,t}^\delta(v)\in L^1$ for all $t>0$. 
	Thus, Kingman's  subadditive ergodic theorem \cite{kingman73} implies that the limit $\Lambda^{\delta}(v):=\lim_{t\to\infty}\frac{1}{t}\overline{\mu}^{\delta}_{0,t}(v)$  exists $\p$-a.s. Furthermore, by \eqref{eq:MIX}, $\xi$ is mixing and thus ergodic, so the $\Lambda^\delta(v)$  is non-random. By standard estimates,  the limit is independent of $\delta>0$ and $x\in\R$.
	
	To show the concavity of $v\mapsto \Lambda(v)$, for $v_1,v_2 \ge0$ and dividing by $t$, the left-hand side of \eqref{eq:concave_Lyap} converges $\p$-a.s.\  to $\Lambda(\lambda v_1 + (1-\lambda) v_2)$, while the  second term on the right-hand side of \eqref{eq:concave_Lyap} converges $\p$-a.s.\  to $(1-\lambda)\Lambda(v_2)$.	Dividing by $t$ and using  \eqref{eq:STAT}, the first term on the right-hand side of \eqref{eq:concave_Lyap} converges in distribution to the constant $\lambda\Lambda(v_1)$, proving the concavity of $v\mapsto \Lambda(v)$.
	
	By \cite[Theorem 7.5.2]{Fr-85} we have $\lim_{t\to\infty}\frac1t \ln u^{\mathds{1}_{(-\delta,\delta)}}(t,0) =\es$, independent of $\delta>0$, giving $\Lambda(0)=\es$. Due to \eqref{eq:POT} and a standard Gaussian computation we have  $\lim_{v\to\infty}\frac{\Lambda(v)}{v}=-\infty$. This, together with $\Lambda(0)=\es>0$ and the concavity of $v\mapsto \Lambda(v),$ implies the existence of a unique $v_0>0$, such that $\Lambda(v_0)=0$. 
\end{proof}

\subsection{A Hoeffding-type inequality for dependent random variables}
\begin{lemma}[{\cite[Theorem 2.4]{rio}}]
	Let $(X_i)_{i\in\Z}$ be a sequence of real-valued bounded random variables, $\widetilde{\F}_i:=\sigma(X_j:j\leq i)$, $S_n:=\sum_{i=1}^{n}X_i,$ and let $(m_1,\ldots,m_n)$ be an $n$-tuple of positive reals such that for all $i\in\{1,\ldots,n\},$ 
	\begin{equation*}
	\sup_{j \in \{i, i+1, \ldots, n\}}\Big( \|X_i^2\|_\infty + 2\| X_i\sum_{k=i+1}^j\E[X_k|\widetilde{\F}_i]\|_\infty \Big) \leq m_i,
	\end{equation*}
	with the convention $\sum_{k=i+1}^i\E[X_k|\widetilde{\F}_i]=0$. Then for every $x>0,$
	\begin{equation*}
	\p\left( |S_n|\geq x \right) \leq \sqrt{e}\exp\left\{ -x^2/(2m_1+\cdots+2m_n) \right\}.
	\end{equation*}
\end{lemma}

Rearranging the quantities in the above result, we arrive at the following corollary which we primarily pronounce explicitly since its formulation tailor-made for our purposes.

\begin{corollary}
	\label{cor:rioDepExp}
	Let $(Y_i)_{i\in\mathbb{Z}}$ be a sequence of real-valued bounded random variables, $\widetilde{\F}^k:=\sigma(Y_j:j\geq k),$ and let $(m_1,\ldots,m_{n})$ be an $n$-tuple of positive real numbers such that 
 for all $i\in\{1,\ldots,n\},$
	\begin{equation}
	\label{eq:rioDepExp}
	\sup_{j \in \{1, \ldots, i\}}\Big( \|Y_i^2\|_\infty + 2 \Big \Vert Y_i\sum_{k=j}^{i-1}\E[Y_k|{\widetilde{\F}}^i]\Big \Vert_\infty \Big) \leq m_i,
	\end{equation}
	with the convention $\sum_{k=i}^{i-1}\E[Y_k|\widetilde{\F}^i]=0$. Then for every $x>0,$
	\begin{equation*}
	\p\Big( \Big| \sum_{i=1}^{n} Y_i \Big|\geq x \Big) \leq \sqrt{e}\exp\left\{ -x^2/(2m_1+\cdots+2m_{n}) \right\}.
	\end{equation*}
\end{corollary}

Recall that $u^{u_0}$ denotes the solution to \eqref{eq:PAM} with initial condition $u_0$.
\begin{lemma}\label{le:MomPAMIncrease}
	For all $x\in\mathbb{R}$ and $0\leq s\leq t$ we have
	\( u^{\mathds{1}_{(-\infty,0]}}(s,x) \leq 2u^{\mathds{1}_{(-\infty,0]}}(t,x). \)
\end{lemma}
\begin{proof}
	By the Feynman-Kac formula, $\xi\geq0$  and the Markov property at time $s$ we have
	\begin{align*}
	u^{\mathds{1}_{(-\infty,0]}}(t,x) &= E_x\left[ e^{\int_0^t\xi(B_r)\diff r};B_t\leq 0 \right] \geq E_x\left[ e^{\int_0^s\xi(B_r)\diff r};B_s\leq 0,B_t\leq 0 \right] \\
	&\geq \Eprob_x^\xi\left[N^\leq(s,0)\right]P_0(B_{t-s}\leq 0) = \frac{1}{2}u^{\mathds{1}_{(-\infty,0]}}(s,x).
	\end{align*}
\end{proof}

\begin{lemma}
	\label{le:leadingPart} 
	Let $N_{s,u,t}^\mathcal{L}$ be as in \eqref{eq:defLeadingPart}.
	Then there exists $c>0$  such that  $\p$-a.s.\  
	\[\Eprobth^\xi_x\big[ N_{t,t,t}^{\mathcal{L}} \big]\geq c \Eprobth^\xi_{\lfloor x\rfloor}\big[ N_{t-1,t+1,t}^{\mathcal{L}} \big]\quad\text{ for all }t\geq1\text{ and }x\geq 1.\] 
\end{lemma}
\begin{proof}
	Write $A_{y,z,t}:=\big\{H_k + y \geq z+t - T_k - 5\chi_1(\munder(t))\ \forall k\in\{1,\ldots,\lfloor\munder(t)\rfloor\}\big\}$ and $p(r):=\inf_{y\leq 0}P_y(B_r\leq 0)=\frac12$.  Then by the Feynman-Kac formula
	\begin{align*}
	\Eprob^\xi_x\big[ N_{t,t,t}^\mathcal{L} \big]&=E_x\big[ e^{\int_0^{t} \xi(B_s) \diff s}; B_{t}\leq 0, A_{0,0,t} \big] \\
	&\geq   E_x\Big[ e^{\int_{H_{\lfloor x\rfloor}}^{t-1+H_{\lfloor x\rfloor}} \xi(B_s) \diff s}; B_{t}\leq 0, B_{t-1+H_{\lfloor x\rfloor}}\leq 0,  A_{H_{\lfloor x\rfloor},1,t}, H_{\lfloor x\rfloor}\leq 1 \Big] \\
	&\geq  E_x\Big[ E_{\lfloor x\rfloor}\big[ e^{\int_0^{t-1} \xi(B_s)\diff s}\times \inf_{r\in[0,1]}p(r); B_{t-1}\leq 0,A_{0,1,t} \big];H_{\lfloor x\rfloor}\leq 1 \Big] \\
	&\geq \frac{1}{2}P_x(H_{\lfloor x\rfloor}\leq 1) E_{\lfloor x\rfloor}\big[ e^{\int_0^{t-1} \xi(B_s)\diff s}; B_{t-1}\leq 0,A_{0,1,t} \big] = c\Eprob^\xi_{\lfloor x\rfloor}\big[ N_{t-1,t+1,t}^{\mathcal{L}} \big],
	\end{align*}
	with $c:= \frac{1}{2}P_x(H_{\lfloor x\rfloor}\leq 1)$.
\end{proof}
The next result is an inequality of Harnack-type flavor for the solution to \eqref{eq:PAM}. 
\begin{lemma}
	\label{le:harnack}
	There exists a constant $\Cl{Harnack}\in(0,\infty)$ such that
	$\p$-a.s.\ 
	for all $y\in\R$, $t\geq1$ and all $u_0\in\mathcal{I}_{\text{PAM}}$ we have
	\begin{equation}
		\label{eq:harnack_appendix}
		u^{u_0}(t,y) \leq \Cr{Harnack} \inf_{x\in[y-1,y+1]} u^{u_0}(t+1,x).
	\end{equation}	
\end{lemma}
\begin{proof}
	For $x\in\R$ and $t>0$ let $f_{t,x}$ be the probability density of a Brownian motion at time $t$, starting in $x$. Let us first show that for all $y,z\in\R$ we have 
	\begin{equation}
		\label{eq:prob_dens_mon}
		f_{1,y}(z) \leq \sqrt{2/e} \inf_{x\in[y-1,y+1]} f_{2,x}(z).
	\end{equation}	
	Indeed, using $f_{t,y}(z)=\frac{1}{\sqrt{2\pi t}}e^{-\frac{(z-y)^2}{2t}}$, \eqref{eq:prob_dens_mon} follows from 
	\begin{align*}
		\inf_{x\in[y-1,y+1]} \left\{ 2(z-y)^2 - (z-x)^2 \right\} &=\begin{cases}
			 2(z-y)^2- (z-y+1)^2  =  (z-y-1)^2 - 2, &\quad y<z,\\
			 2(z-y)^2- (z-y-1)^2  = (z-y+1)^2 - 2, &\quad y\geq z,\\
		\end{cases}\\
		&\geq -2.
	\end{align*}
	The Feynman-Kac formula in combination with the Markov property applied at time $1$ now supplies us with
	\begin{align*}
		u^{u_0}(t,y) &= E_y \big[ e^{\int_0^t \xi(B_r)\diff r} u_0(B_t) \big] \leq e^{\es} \int_\R f_{1,y}(z) E_z\big[ e^{\int_0^{t-1} \xi(B_r)\diff r} u_0(B_{t-1}) \big] \diff z \\
		&\leq e^{\es} \sqrt{2/e} \inf_{x\in[y-1,y+1]} \int_\R f_{2,x}(z)  E_z\big[ e^{\int_0^{t-1} \xi(B_r)\diff r} u_0(B_{t-1}) \big] \diff z \\
		&\leq e^{\es} \sqrt{2/e} \inf_{x\in[y-1,y+1]} u^{u_0}(t+1,x),
	\end{align*}
	where in the second inequality we used \eqref{eq:prob_dens_mon}. Then  \eqref{eq:harnack_appendix} follows choosing $\Cr{Harnack}=e^{\es} \sqrt{2/e}$.  
\end{proof}

\subsection{Proof of the McKean representation}
\label{sec:mckean_proof}
\begin{proof}[Proof of Proposition~\ref{prop:mckean1}]
	We start with considering $w_0$ continuous and $w_0(x)\in[0,1]$ for all $x\in\R$. Define $u:=1-w$ and $u_0:=1-w_0$. Then by \cite[Remark~4.4.4]{karatzas_shreve} (providing the existence of a $C^{1,2}$-solution to \eqref{eq:PAM}) and \cite[Corollary~4.4.5]{karatzas_shreve} (Feynman-Kac representation),  the function $u^{(1)}(t,x)=E_x\big[ e^{-\int_0^t\xi(B_r)\diff r} u_0(B_t) \big]$ is a classical solution to $u^{(1)}_t(t,x)=\frac{1}{2}\Delta u^{(1)}(t,x) - \xi(x)u^{(1)}(t,x)$ and we have  $u^{(1)}(0,x)=u_0(x)$. Furthermore, by the same argument, for all $s<t$ and $k\in\N$ fixed, the function
	\begin{align*}
	u^{(2)}(t,x;s,k)&:= E_x \Big[ \xi(B_{t-s})e^{-\int_0^{t-s}\xi(B_r)\diff r}u(s,B_{t-s})^k  \Big]
	\end{align*}
	is a classical solution to  $u^{(2)}_t(t,x;s,k)=\frac{1}{2}\Delta u^{(2)}(t,x;s,k) - \xi(x)u^{(2)}(t,x;s,k)$. By dominated convergence, taking advantage of the uniform continuity of the first and second order derivatives, for every fixed $t'\leq t$, we can interchange limits to obtain the identities
	\begin{align}
	\label{eq:interchange_limits}
	\begin{split}
	\sum_k p_k \int_0^{t'} \frac{\partial}{\partial t} u^{(2)}(t,x;s,k) \diff s &=  \frac{\partial}{\partial t} \sum_k p_k \int_0^{t'}  u^{(2)}(t,x;s,k) \diff s,\\
	\sum_k p_k \int_0^{t'}  \big(\frac{1}{2}\Delta-\xi(x)\big) u^{(2)}(t,x;s,k) \diff s &=  \big(\frac{1}{2}\Delta-\xi(x)\big) \sum_k p_k \int_0^{t'}  u^{(2)}(t,x;s,k) \diff s.
	 \end{split}
 	\end{align}
 	Conditioning on the first splitting time of the initial particle and on how many children are born,  we have
 	\begin{align*}
 	u(t,x) 
 	&=E_x\big[ e^{-\int_0^t\xi(B_r)\diff r} u_0(B_t) \big] + \sum_k p_k \int_0^t E_x\Big[  \xi(B_s)e^{-\int_0^s\xi(B_s)}u(t-s,B_s)^k \Big] \diff s \\
 	&=u^{(1)}(t,x) + \sum_k p_k \int_0^t E_x\Big[  \xi(B_{t-s})e^{-\int_0^{t-s}\xi(B_{r})\diff r}u(s,B_{t-s})^k \Big] \diff s, 
 	\end{align*}
 	where we used the Markov property of the process $(B_t)_{t\geq0}$. We deduce that for $h\neq0,$
 	\begin{align*}
 	\frac{1}{h}\big( u(t+h,x)-u(t,x) \big) &= \frac{1}{h}\big( u^{(1)}(t+h,x)-u^{(1)}(t,x) \big) \\
 	&\quad + \sum_k p_k E_x \Big[ \frac{1}{h}\int_t^{t+h} \xi(B_{t+h-s})e^{ -\int_0^{t+h-s}\xi(B_r)\diff r } u(s,B_{t+h-s})^k \diff s \Big] \\
 	&\quad +\sum_k p_k \int_0^t \frac{1}{h} \big( u^{(2)}(t+h,x;s,k) - u^{(2)}(t,x;s,k) \big)\diff s.
 	\end{align*}
 	Invoking once again the arguments from the beginning of the proof, as $h$ tends to zero, the first summand converges to $\frac{1}{2}\Delta u^{(1)}(t,x) - \xi(x)u^{(1)}(t,x)$ and by dominated convergence, the second summand converges to $\xi(x)\sum_k p_k  u(t,x)^k$. For the third term, we mention that the integrand is uniformly bounded since $u^{(2)}_t$ is continuous. Again, due to dominated convergence  and  \eqref{eq:interchange_limits}, the latter term converges to
 	\begin{align*}
 	\big(\frac{1}{2}\Delta-\xi(x)\big) \sum_k p_k \int_0^{t}  u^{(2)}(t,x;s,k) \diff s.
 	\end{align*}
 	All in all, combining these observations we arrive at
 	\begin{align*}
 	u_t(t,x) &= u^{(1)}_t(t,x) + \xi(x)\sum_k p_k  u(t,x)^k + \big(\frac{1}{2}\Delta-\xi(x)\big) \sum_k p_k \int_0^{t}  u^{(2)}(t,x;s,k) \diff s \\
 	&=\big(\frac{1}{2}\Delta-\xi(x)\big) u(t,x) + \xi(x)\sum_k p_ku(t,x)^k,
 	\end{align*}
 	which is equivalent $w$ being a solution to \eqref{eq:KPP1}. Thus we have shown the claim for continuous $w_0$. By Remark~\ref{rem:gen_approx}, for $w_0\in\mathcal{I}_{\text{F-KPP}}$ the result follows by approximation. 	
\end{proof}

\subsection{Existence and monotonicity of solutions} 
We start with an existence result for generalized and also classical solutions.

\begin{proposition}[{\cite[Theorem 7.4.1]{Fr-85}}] \label{prop:genClass}
	Let $w_0$ be measurable, non-negative and bounded, and let $F$ fulfill \eqref{eq:standard_condi}. Then $\p$-a.s.\ 
	there exists a unique, bounded generalized solution $w\in C\big( (0,\infty)\times\R \big)$ to \eqref{eq:KPP1}, which satisfies $0\leq w(t,x)\leq 1\vee \sup_x w_0(x)$ for all $(t,x)\in[0,\infty)\times\R$. 
	If additionally $w_0$ is continuous and $x\mapsto \xi(x,\omega)$ is H\"older continuous for $\p$-a.a.\ $\omega\in\Omega$, then $\p$-a.s.\ the generalized solution $w$ is a classical one, i.e. $w\in C^{1,2}\big((0,\infty)\times\R\big)\cap C\big([0,\infty)\times\R\big)$  and $w$ solves \eqref{eq:KPP1}.
\end{proposition}

For the next lemma, we introduce the differential operator 
\begin{equation*}
(\mathcal{L}_G w)(t,x) := w_t(t,x) - \frac{1}{2} w_{xx}(t,x) - G(x,w(t,x)), 
\end{equation*}
where $G$ is uniformly Lipschitz-continuous in $w$, i.e., there exists a constant $\alpha>0$, such that 
\begin{equation}
\label{eq:lipschitz_G}
|G(x,u)-G(x,v)|\leq \alpha|u-v|\quad\forall\ x,u,v\in\R.
\end{equation}
The next lemma is  in the spirit of  \cite[{Proposition 2.1}]{aronson_diffusion}.
\begin{lemma}
	\label{le:maximum_princ}
	Let $T>0$, $Q:=(0,T)\times\R$ and $G$ be such that \eqref{eq:lipschitz_G} holds. Let $w^{(1)}$ and $w^{(2)}$ be non-negative and bounded functions on $\overline{Q}$, such that for $i\in\{1,2\}$, $w^{(i)}_x$ and $w^{(i)}_{xx}$ are continuous on $Q,$ and such that $w^{(i)}_t$ exists in $Q.$ If $\mathcal{L}_{G}w^{(1)} \leq \mathcal{L}_{G}w^{(2)}$ on $Q$ and $0\leq w^{(1)}(0,x)\leq w^{(2)}(0,x)$ for all $x\in\R$, then also $w^{(1)} \leq w^{(2)}$ on $Q.$
\end{lemma}
\begin{proof}
	From the assumptions we infer that
	\begin{align*}
	w^{(2)}_t(t,x)-w^{(1)}_t(t,x)  &\geq  \frac12(w^{(2)}_{xx}(t,x)-w^{(1)}_{xx}(t,x)) +   G(x,w^{(2)}(t,x))-G(x,w^{(1)}(t,x)). 
	\end{align*}
	Then, recalling  \eqref{eq:lipschitz_G}	and letting 
	\[ v(t,x):= e^{-2\alpha t}\big(w^{(2)}(t,x)-w^{(1)}(t,x)\big), \]
	we get
	\begin{align*}
	v_t (t,x) -\frac12 v_{xx}(t,x) &\geq  -2\alpha v(t,x) + e^{-2\alpha t}\big( G(x,w^{(2)}(t,x))-G(x,w^{(1)}(t,x)) \big) \\
	&\geq \big( -2-\text{sgn}(v(t,x)) \big)\alpha\cdot v(t,x).
	\end{align*}
	Now the the first factor on the right-hand side is negative and bounded. Applying the maximum principle \cite[{Theorem~8.1.4}]{krylov96} to $-v$ then implies that $v \geq0$ on $ Q$ and we can conclude. 
\end{proof}

As an important application we get that the solution $w$ to $\mathcal{L}_Gw=0$ is monotone in $G$ and in the initial condition.

\begin{corollary}
	\label{cor:mon_solutions}
	Let $T>0$, $Q:=(0,T)\times\R$, and let $G_1$ and $G_2$ fulfill ${G_1}\leq {G_2}$  on $\R\times[0,\infty).$ Furthermore, assume that  $G_2$ satisfies \eqref{eq:lipschitz_G}. In addition, let  $w^{(1)}$ and $w^{(2)}$  be non-negative and bounded functions on $\overline{Q}$, such that for $i\in\{1,2\}$, $w^{(i)}_x$ and $w^{(i)}_{xx}$ are continuous on $Q$ and  $w^{(i)}_t$ exist on $Q$. If  $\mathcal{L}_{G_1}w^{(1)} = \mathcal{L}_{G_2}w^{(2)}$ and  $w^{(1)}(0,\cdot)\leq w^{(2)}(0,\cdot)$ on $x\in\R$, then  we have  $w^{(1)} \leq w^{(2)}$ on $Q.$
\end{corollary}
\begin{proof}
	Since the function $w^{(2)}$ is non-negative, we have $\mathcal{L}_{G_2}w^{(2)}=\mathcal{L}_{G_1}w^{(1)}\geq \mathcal{L}_{G_2}w^{(1)}$. Then by Lemma~\ref{le:maximum_princ}, we have  $w^{(1)}(t,x)\leq w^{(2)}(t,x)$ for all $(t,x)\in Q$.
\end{proof}

\begin{corollary}
	Let $G$ fulfill \eqref{eq:lipschitz_G} and $G(x,0)=G(x,1)=0$ for all $x\in\R$. Let $w$ be a solution to $\mathcal{L}_Gw=0$ with $0\leq w(0,x)\leq 1$. Then $0 \leq w(t,x) \leq 1$ for all $(t,x)\in[0,\infty)\times\R$.
\end{corollary}
\begin{proof}
	The functions $w^{(1)}(t,x)=0$ and $w^{(2)}(t,x)=1$ are solutions to $\mathcal{L}w^{(1)}=\mathcal{L}w^{(2)}=0$ and $w^{(1)}(0,x)\leq w(0,x)\leq w^{(2)}(0,x)$. The claim then follows from Lemma~\ref{le:maximum_princ}.
\end{proof}

\bibliographystyle{abbrv}
\bibliography{bibliothek}

\end{document}